\documentclass[preprint,11pt]{elsarticle}

\usepackage[a4paper]{geometry}
\usepackage[colorlinks,linkcolor=black,urlcolor=black,citecolor=black]{hyperref}
\usepackage{amsmath, amsthm, amssymb, mathtools, tikz, enumitem, booktabs}
\usepackage[bf,singlelinecheck=false]{caption}
\usepackage[ruled]{algorithm}
\usepackage[noend]{algpseudocode}
\usepackage{algorithmicx, comment}
\usepackage{aliascnt}
\usepackage{multirow}
\usepackage[T1]{fontenc}

\theoremstyle{definition}
\newtheorem{theorem}{Theorem}[section]
\newtheorem{definition}[theorem]{Definition}
\newtheorem{lemma}[theorem]{Lemma}
\newtheorem{example}[theorem]{Example}
\newtheorem{remark}[theorem]{Remark}
\newtheorem{cor}[theorem]{Corollary}

\makeatletter
\let\c@algorithm\relax
\let\c@figure\relax
\let\c@table\relax
\let\c@theorem\relax
\makeatother
\newaliascnt{algorithm}{common}
\newaliascnt{figure}{common}
\newaliascnt{table}{common}
\newaliascnt{theorem}{common}

\usetikzlibrary{arrows, shapes, patterns, decorations.pathmorphing}
\newcommand{\arc}[2]{\path{} (#1) edge [->,thick] node {} (#2);}
\newcommand{\arcWavy}[2]{\path{} (#1) edge [->,style={decorate, decoration=snake}] node {} (#2);}
\newcommand{\arcThick}[2]{\path{} (#1) edge [->,line width=0.75mm] node {} (#2);}
\newcommand{\arcSym}[2]{\path{} (#1) edge [<->,thick] node {} (#2);}
\newcommand{\arcDash}[2]{\path{} (#1) edge [->,dashed,thick] node {} (#2);}
\newcommand{\arcSymDash}[2]{\path{} (#1) edge [<->,dashed,thick] node {} (#2);}
\newcommand{\arcSymDot}[2]{\path{} (#1) edge [<->,dotted,thick] node {} (#2);}
\newcommand{\looparcL}[1]{\path{} (#1) edge [loop left,thick] node {} (#1);}
\newcommand{\looparcR}[1]{\path{} (#1) edge [loop right,thick] node {} (#1);}

\MakeRobust{\Call}

\def\<{\langle}
\def\>{\rangle}
\newcommand{\idOmega}{\textrm{id}_{\Omega}}
\newcommand{\set}[2]{\{#1\,:\,#2\}}
\newcommand{\N}{\mathbb{N}}
\newcommand{\n}{\{1,\ldots,n\}}
\newcommand{\Fail}{\textsc{Fail}}
\newcommand{\labelSet}{\mathfrak{L}}
\newcommand{\labelFunc}{\textsc{Label}}
\newcommand{\exLabel}[1]{\textit{#1}}
\newcommand{\stackS}{S}
\newcommand{\stackT}{T}
\newcommand{\stackU}{U}
\newcommand{\stackV}{V}
\newcommand{\nauty}{\textsc{nauty}}
\newcommand{\bliss}{\textsc{bliss}}
\newcommand{\approxFunc}{\textsc{Approx}}
\newcommand{\splitFunc}{\textsc{Split}}
\newcommand{\fixedFunc}{\textsc{Fixed}}
\newcommand{\canonFunc}{\textsc{Canon}}
\newcommand{\equitableFunc}{\textsc{Equitable}}
\newcommand{\Sn}[1]{\mathcal{S}_{#1}}

\newcommand\Approx[2]{\ifthenelse{\equal{#1}{#2}}
                       {\textsc{Approx} (#1)}
                       {\ifthenelse{\equal{#2}{}}
                         {\textsc{Approx} (#1)}
                         {\Call{Approx}{#1,#2}}}}
\newcommand{\Split}[2]{\ifthenelse{\equal{#2}{}}
                        {\Call{Split}{#1}}
                        {\Call{Split}{#1,#2}}}
\newcommand{\Fixed}[1]{\Call{Fixed}{#1}}
\newcommand{\Canon}[1]{\Call{Canon}{#1}}
\newcommand{\Equitable}[1]{\Call{Equitable}{#1}}

\newcommand{\Sym}[1]{\operatorname{Sym}\!\left(#1\right)}
\newcommand{\Auto}[1]{\operatorname{Aut}\!\left(#1\right)}
\newcommand{\Iso}[2]{\ifthenelse{\equal{#2}{}}
                      {\operatorname{Iso}\!\left(#1\right)}
                      {\operatorname{Iso}\!\left(#1,#2\right)}}

\newcommand{\Stacks}[1]{\operatorname{\textsc{DigraphStacks}}\!\left(#1\right)}
\newcommand{\Squash}[1]{\Call{Squash}{#1}}
\newcommand{\LabelledDigraphs}[2]{\operatorname{\textsc{LabelledDigraphs}}
                                  \!\left(#1,#2\right)}
\newcommand{\EmptyStack}[1]{\operatorname{\textsc{EmptyStack}}\!\left(#1\right)}

\algnewcommand\algorithmicswitch{\textbf{switch}}
\algnewcommand\algorithmiccase{\textbf{case}}
\algdef{SE}[SWITCH]{Switch}{EndSwitch}[1]{\algorithmicswitch\ #1
\ \algorithmicdo{}
}{\algorithmicend\ \algorithmicswitch}
\algdef{SE}[CASE]{Case}{EndCase}[1]{\algorithmiccase\ #1:}{\algorithmicend\
\algorithmiccase}
\algtext*{EndSwitch}
\algtext*{EndCase}

\begin{document}

\begin{frontmatter}

\title{Permutation group algorithms based on directed graphs\\(extended version)}

\author{Christopher Jefferson}
\ead{caj21@st-andrews.ac.uk}
\ead[url]{https://caj.host.cs.st-andrews.ac.uk}

\author{Markus Pfeiffer}
\ead{markus.pfeiffer@st-andrews.ac.uk}
\ead[url]{https://www.morphism.de/~markusp}

\author{Wilf A.\ Wilson}
\address{University of St~Andrews\\School of Computer Science\\North Haugh\\St Andrews\\KY16 9SX\\Scotland}
\ead{waw7@st-andrews.ac.uk}
\ead[url]{https://wilf.me}

\author{Rebecca Waldecker}
\address{Martin-Luther-Universit\"at Halle-Wittenberg\\Institut f\"ur Mathematik\\06099 Halle\\Germany}
\ead{rebecca.waldecker@mathematik.uni-halle.de}
\ead[url]{https://www2.mathematik.uni-halle.de/waldecker/index-english.html}

\journal{Journal of Algebra}

\begin{abstract}
  We introduce a new framework for solving an important class of computational
  problems involving finite permutation groups, which includes calculating set
  stabilisers, intersections of subgroups, and isomorphisms of combinatorial
  structures.
  Our techniques generalise `partition backtrack', which is the
  current state-of-the-art algorithm introduced by Jeffrey Leon in 1991, and
  which has inspired our work.
  Our backtrack search algorithms are organised around vertex- and arc-labelled
  directed graphs, which allow us to represent many problems more richly than
  do ordered partitions.
  We present the theory underpinning our framework,
  and we include the results of experiments showing that our techniques often
  result in smaller search spaces than does partition backtrack.
  An implementation of our algorithms is available as free software in the
  \textsc{GraphBacktracking} package for \textsc{GAP}.
\end{abstract}

\end{frontmatter}

\noindent\textbf{Note:}
This is an extended version of \emph{Permutation group algorithms
based on directed graphs}~\cite{shorterGBpaper}.  The shorter article was
derived from this one, according to the comments of referees;
it includes some improved exposition, and
omits some proofs, examples, and other details.
We recommend that the reader begins with the shorter article.

\section{Introduction}\label{sec-intro}

In~\cite{leon1991}, Jeffrey Leon introduced the \emph{partition backtrack
algorithm} for intersecting subgroups of finite
symmetric groups, or their cosets, in which membership of an individual
permutation can be easily decided.
Many of the most important problems in computational permutation group theory
can be formulated in this way, and thus can be solved with partition
backtrack.
These include the computation of
point and set stabilisers and transporters;
normalisers and centralisers of subsets and subgroups;
automorphisms and isomorphisms of a wide range of combinatorial structures;
element and subgroup conjugacy;
and any conjunction of such problems.
These problems have differing time complexities (see~\cite[Chapter~3]{seress},
for example),
but for many of them, partition backtrack currently solves the problem in the
fastest known way.

Leon's algorithm performs a backtrack search through the
elements of the symmetric group, which it organises around a collection of
ordered partitions.
Partition backtrack builds upon the `individualisation-refinement' technique of
McKay, which he most recently described in~\cite{practical2}, and which
is used to compute automorphism groups and canonical labellings
of finite graphs.
Leon's algorithm encodes information about the given problem into the ordered
partitions, which it then uses to cleverly prune (i.e., omit superfluous parts of) the
search space.
Despite its excellent performance in many instances, this technique has
exponential worst-case complexity, and there remain many important
examples of problems that are beyond its reach. There is, therefore, still scope
for improvement.

Several extensions to partition backtrack have taken further inspiration from
the graph-based ideas of {M}c{K}ay.
Thei{\ss}en, for instance, used orbital graphs in~\cite{theissen} to significantly
improve the computation of normalisers.  This theme was taken up
in~\cite{newrefiners}, by the first three authors of the present paper,
for intersections and set stabilisers.
The techniques described in~\cite{newrefiners,theissen} encode some information
about certain orbital graphs into the ordered partitions of the search, thereby
enabling better pruning of the search space.
This suggests that even more powerful pruning, and ultimately better
performance, could be obtained by using graphs directly, at the expense of the
increased computation required at each node of the remaining search.
In the present paper, we investigate precisely this idea. More specifically, we
demonstrate the possibility and feasibility of placing graphs (in fact, vertex-
and arc-labelled directed graphs) at the heart of backtrack search algorithms in
the symmetric group, thereby generalising partition backtrack.

The purpose of this paper is to give the \emph{theoretical} basis for our ideas,
along with some initial experimental data.  In particular, at this point, we do
not concern ourselves with the time complexity or speed of our algorithms, and
we do not discuss their implementation details.  However, we do intend for
our algorithms to be practical, and we expect that with sufficient further
development into their implementations, our algorithms should perform
competitively against, and even beat, partition backtrack for many classes of
problems.

Note that although this paper is heavily influenced by the work of
Leon~\cite{leon1997, leon1991}, we intend for it to be understandable
without prior knowledge of his work.

This paper is organised as follows.
In Section~\ref{sec-defs-prelims}, we present our notation,
introduce and refer to standard concepts in graph theory and group
theory, and discuss labelled digraphs.
In Section~\ref{sec-stacks}, we introduce stacks of labelled
digraphs, which are the fundamental structures around which we organise our
search algorithms.  The remaining tools that are crucial for our algorithms are
isomorphism approximators and fixed-point approximators
(Section~\ref{sec-approximations}), refiners (Section~\ref{sec-refiners}), and
splitters (Section~\ref{sec-splitter}).
We present our algorithms and prove their correctness in
Section~\ref{sec-search}, and in Section~\ref{sec-experiments}  we give details
of various experiments that compare our algorithms with the
current state-of-the-art techniques.
We conclude, in Section~\ref{sec-end}, with brief comments on the results of
this paper and the directions that they suggest for further investigation.

\subsection*{Acknowledgements}

The authors would like to thank the DFG (\textbf{Grant no.~WA 3089/6-1}) and the
Volkswagenstiftung (\textbf{Grant no.~93764}) for financially supporting this work and projects leading up to it.
The first author gratefully acknowledges funding from the Royal Society
(\textbf{Grant code URF\textbackslash R\textbackslash 180015}).
Special thanks go to Paula H\"ahndel for frequent
discussions on topics related to this work, and for suggestions on how to improve this paper.
We also thank Ruth Hoffmann for suggesting further helpful improvements.

\section{Preliminaries}\label{sec-defs-prelims}

Throughout this paper, $\Omega$ denotes some finite totally-ordered set on which
we define all of our groups, digraphs, and related objects.  For example, every
group in this paper is a finite permutation group on $\Omega$, i.e.\ a subgroup
of $\Sym{\Omega}$, the symmetric group on $\Omega$.
We follow the standard
group-theoretic notation and terminology from the literature, such as that used
in~\cite{dixonmortimer}, and write $\cdot$ for the composition of maps in
$\Sym{\Omega}$, or we omit a symbol for this binary operation altogether. We
write $\N$ for the set $\{1, 2, 3, \ldots\}$ of all natural numbers, and $\N_{0}
\coloneqq \N \cup \{0\}$.
If $n \in \N$, then $\Sn{n} \coloneqq \Sym{\{1,\ldots,n\}}$.

For many types of objects that we define on $\Omega$, we give a way of applying
elements of $\Sym{\Omega}$ to them (denoted by exponentiation) in a way that is
structure-preserving.  For example, if we have a graph with vertex set $\Omega$,
then we can apply the same element of $\Sym{\Omega}$ to every vertex,
and obtain a new graph with the same vertex set, $\Omega$.
This principle is used throughout this article, mainly for graphs or digraphs with
vertex set $\Omega$, but also for sets or lists of elements in $\Omega$,
and for sets or lists of subsets of $\Omega$ (such as partitions of $\Omega$).

Let $\mathcal{O}$ and $\mathcal{Q}$ be digraphs with vertex set $\Omega$
(or partitions, lists etc., as mentioned above).
Then we say that a permutation $g \in \Sym{\Omega}$ \emph{induces an
isomorphism from $\mathcal{O}$ to $\mathcal{Q}$} if and only if it defines a
map from $\mathcal{O}$ to $\mathcal{Q}$, i.e.\ $\mathcal{O}^{g} = \mathcal{Q}$, and if it is structure-preserving. For digraphs this means that arcs are preserved, for partitions it means that the number and sizes of cells are preserved.

We write $\Iso{\mathcal{O}}{\mathcal{Q}}$ for the set of
isomorphisms from $\mathcal{O}$ to $\mathcal{Q}$ that are induced by elements of
$\Sym{\Omega}$.  If $\Iso{\mathcal{O}}{\mathcal{Q}}$ is non-empty, then we
call $\mathcal{O}$ and $\mathcal{Q}$ \emph{isomorphic}, sometimes
denoted by
$\mathcal{O} \cong \mathcal{Q}$.  Similarly, we consider $\Auto{\mathcal{O}}
\leq \Sym{\Omega}$ to be the subgroup of $\Sym{\Omega}$ consisting of all
elements that induce isomorphisms from $\mathcal{O}$ to itself, i.e.\
\emph{automorphisms}.  Note that, for all $g \in
\Iso{\mathcal{O}}{\mathcal{Q}}$,
$\Auto{\mathcal{O}}^{g} (= \set{g^{-1} h g}{h \in \Auto{\mathcal{O}}}) =
\Auto{\mathcal{Q}}$.  In particular, if $\mathcal{O} \cong \mathcal{Q}$, then
$\Iso{\mathcal{O}}{\mathcal{Q}}$ is a right coset of $\Auto{\mathcal{O}}$ and a
left coset of $\Auto{\mathcal{Q}}$ in $\Sym{\Omega}$.

\subsection{Ordered partitions}\label{sec-defn-ordered-partition}

An \emph{ordered partition} of $\Omega$ is a list of non-empty disjoint subsets
of $\Omega$, called \emph{cells}, whose union is $\Omega$. The `ordering' is
thus defined between cells, not within a cell.  For example, the list $[\{3,
7\}, \{1\}, \{2, 4, 5\}, \{6\}]$ is an ordered partition of $\{1,\ldots,7\}$.
The group $\Sym{\Omega}$ acts on the set of ordered partitions
of $\Omega$ by acting on its entries:
if $g \in \Sym{\Omega}$ and $\Pi \coloneqq [C_{1}, \ldots, C_{k}]$
is an ordered partition of $\Omega$ for some $k \in \N$,
then the action is defined via
$\Pi^{g} \coloneqq [C_{1}^{g}, \ldots, C_{k}^{g}]$.

If $k, l \in \N$
and $\Pi_{1} \coloneqq [C_{1}, \ldots, C_{k}]$
and $\Pi_{2} \coloneqq [D_{1}, \ldots, D_{l}]$
are ordered partitions of $\Omega$, then a permutation $g \in \Sym{\Omega}$
induces an isomorphism from $\Pi_{1}$ to $\Pi_{2}$ if and only if $C_{i}^{g} =
D_{i}$ for all $i \in \{1, \ldots, k\}$.
Since $\Sym{\Omega}$ acts $|\Omega|$-transitively on $\Omega$, it follows that
$\Pi_{1}$ and $\Pi_{2}$ are isomorphic
if and only if
$k = l$ and $|C_{i}| = |D_{i}|$ for all $i \in \{1, \ldots, k\}$.
In addition, the automorphism group of $\Pi_{1}$ induced by $\Sym{\Omega}$ is
isomorphic to $\Sym{C_{1}} \times \cdots \times \Sym{C_{k}}$ in a natural way.

\subsection{Labelled digraphs}

A \emph{graph} with vertex set $\Omega$ is a pair $(\Omega, E)$, where $E$ is a set of $2$-subsets of
$\Omega$. A \emph{directed graph} with vertex set $\Omega$, or \emph{digraph} for
short, is a pair $(\Omega, A)$, where $A \subseteq \Omega \times \Omega$ is a
set of pairs of elements in $\Omega$ called \emph{arcs}. The elements of
$\Omega$ are called \emph{vertices} in the context of graphs and digraphs.  Our
definition allows a digraph to have \emph{loops}, which are arcs of the form
$(\alpha, \alpha)$ for some vertex $\alpha \in \Omega$.

Our techniques for searching in $\Sym{\Omega}$ are built around digraphs in
which each vertex and arc is given a \textit{label} from a label set $\labelSet$.  We
define a \emph{vertex- and arc-labelled digraph}, or \emph{labelled digraph} for
short, to be a triple $(\Omega, A, \labelFunc)$, where $(\Omega, A)$ is a
digraph and $\labelFunc$ is a function from $\Omega \cup A$ to $\labelSet$.
More precisely, for any $\delta \in \Omega$ and $(\alpha, \beta) \in A$, the
label of the vertex $\delta$ is $\labelFunc(\delta) \in \labelSet$, and the
label of the arc $(\alpha, \beta)$ is $\labelFunc(\alpha, \beta) \in
  \labelSet$.  We call such a function a \emph{labelling function}.

In a theoretical sense, the properties of the labels themselves are unimportant,
since we only use them to distinguish certain vertices or arcs from others,
and thereby break symmetries.  For convenience, therefore, we fix
$\labelSet$ as some non-empty set that contains every label that we require, and
which serves as the codomain of every labelling function.  Thus two labelled
digraphs on $\Omega$ are equal if and only if their sets of arcs are equal, and
any vertex or arc has the same label in both labelled digraphs.
For the concepts in Section~\ref{sec-equitable}, we require
some arbitrary but fixed total ordering to be defined on $\labelSet$.

The symmetric group on $\Omega$ acts on the sets of graphs and digraphs with
vertex
set $\Omega$, respectively, and on their labelled variants, in a natural way.
We give more details about this for labelled digraphs; the forthcoming notions
are defined analogously for the other kinds of graphs and digraphs that we have
mentioned.  Let $\LabelledDigraphs{\Omega}{\labelSet}$ denote the class of all
labelled digraphs on $\Omega$ with labels in $\labelSet$, let
$\Gamma = (\Omega, A, \labelFunc) \in \LabelledDigraphs{\Omega}{\labelSet}$ and
$g \in \Sym{\Omega}$.  Then we define
$\Gamma^{g} = (\Omega, A^{g}, \labelFunc^{g}) \in
\LabelledDigraphs{\Omega}{\labelSet}$, where:
\begin{enumerate}[label=\textrm{(\roman*)}]
  \item
    $A^{g} \coloneqq \set{(\alpha^{g}, \beta^{g})}{(\alpha, \beta) \in A}$,

  \item
    $\labelFunc^{g}(\delta) \coloneqq
          \labelFunc(\delta^{g^{-1}})$ for all $\delta \in \Omega$, and

  \item
    $\labelFunc^{g}(\alpha, \beta) \coloneqq
     \labelFunc(\alpha^{g^{-1}}, \beta^{g^{-1}})$
    for all $(\alpha, \beta) \in A^{g}$.
\end{enumerate}
In other words, the arcs are mapped according to $g$, and the label of a vertex
or arc in $\Gamma^{g}$ is the label of its preimage in $\Gamma$.  This implies
that the labels that appear in $\Gamma^{g}$ are exactly those that appear in
$\Gamma$.  This gives rise to a group action of $\Sym{\Omega}$ on
$\LabelledDigraphs{\Omega}{\labelSet}$, since the identity permutation
$\idOmega$ fixes any labelled digraph $\Gamma$, and $\Gamma^{g h} =
{(\Gamma^{g})}^{h}$ for all $g, h \in \Sym{\Omega}$.

Let $\Gamma, \Delta \in \LabelledDigraphs{\Omega}{\labelSet}$.  A permutation
$g \in \Sym{\Omega}$ induces an isomorphism from $\Gamma$ to $\Delta$ if and
only if $\Gamma^{g} = \Delta$.  This means that $g$ maps each vertex to a vertex
with the same label, maps each arc to an arc with the same label, and maps pairs
of vertices in $\Omega$ that do not form arcs to pairs that do not form arcs.

The action of a permutation on a labelled digraph is illustrated in
Example~\ref{ex-digraph-action}.

\begin{example}\label{ex-digraph-action}

  Let $\Omega = \{1, \ldots, 5\}$,
  $A = \{ (2, 2), (2, 3), (3, 2), (3, 5), (5, 1), (5, 4) \} \subseteq
    \Omega \times \Omega$, and $\labelSet = \{
      \exLabel{black},\,
      \exLabel{white},\,
      \exLabel{solid},\,
      \exLabel{dashed}
    \}$.
  We define a labelling function $\labelFunc : \Omega \cup A \to \labelSet$
  as follows:
  for all $\delta \in \Omega$ and all $(\alpha, \beta) \in A$, let
  \[\labelFunc(\delta) =
    \begin{cases}
      \exLabel{black}   & \text{if}~\delta\ \text{is prime,}     \\
      \exLabel{white} & \text{otherwise,}
    \end{cases}
    \qquad
    \text{and}
    \qquad
    \labelFunc(\alpha, \beta) =
    \begin{cases}
      \exLabel{solid}  & \text{if}~\alpha \leq \beta, \\
      \exLabel{dashed} & \text{if}~\alpha > \beta.
    \end{cases}
  \]
  \begin{figure}[!ht]
  \centering
  \begin{tikzpicture}
    \tikzstyle{black}=[circle, draw=white, fill=black!100]
    \tikzstyle{white}=[circle, draw=black]

    \node[white] (1) at (2, 2) {$\mathbf{1}$};
    \node[black] (2) at (2, 0) {$\color{white}\mathbf{2}$};
    \node[black] (3) at (0, 0) {$\color{white}\mathbf{3}$};
    \node[white] (4) at (0, 2) {$\mathbf{4}$};
    \node[black] (5) at (1, 1) {$\color{white}\mathbf{5}$};

    \node[white] (11) at (7, 2) {$\mathbf{5}$};
    \node[black] (12) at (7, 0) {$\color{white}\mathbf{3}$};
    \node[black] (13) at (5, 0) {$\color{white}\mathbf{2}$};
    \node[white] (14) at (5, 2) {$\mathbf{4}$};
    \node[black] (15) at (6, 1) {$\color{white}\mathbf{1}$};

    \node (21) at (1, -1) {$\Gamma$};
    \node (31) at (6, -1) {$\Gamma^{(1\,5)(2\,3)}$};

    \path (2) edge [->, bend left=20, thick] node {} (3);
    \path (3) edge [->, bend left=20, dashed, thick] node {} (2);
    \looparcR{2}
    \arc{3}{5}
    \arcDash{5}{1}
    \arcDash{5}{4}

    \path (12) edge [->, bend left=20, thick] node {} (13);
    \path (13) edge [->, bend left=20, dashed, thick] node {} (12);
    \looparcR{12}
    \arc{13}{15}
    \arcDash{15}{11}
    \arcDash{15}{14}
  \end{tikzpicture}
  \caption{\label{fig-digraph-action}The labelled digraphs $\Gamma$ and
    $\Gamma^{(1\,5)(2\,3)}$ from Example~\ref{ex-digraph-action}.}
  \end{figure}
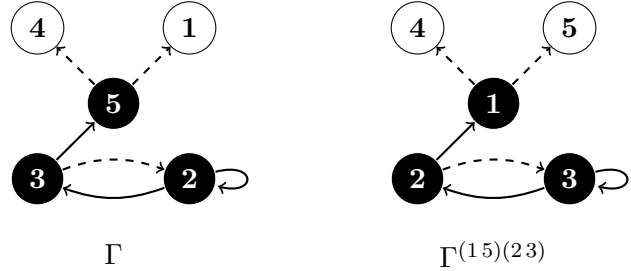

  The diagram on the left of Figure~\ref{fig-digraph-action} depicts the
  labelled digraph $\Gamma \coloneqq (\Omega, A, \labelFunc)$, and the diagram on
  the right of Figure~\ref{fig-digraph-action} depicts $\Gamma^{(1\,5)(2\,3)}$,
  where each vertex and arc has a style corresponding to its
  label.  Note that the diagrams look identical,
  except that the vertices are numbered differently,
  according to $(1\,5)(2\,3)$.  This permutation induces an isomorphism from
  $\Gamma$ to $\Gamma^{(1\,5)(2\,3)}$, by definition, but it does not induce an
  automorphism of $\Gamma$, since $\Gamma \neq \Gamma^{(1\,5)(2\,3)}$.
  This can be seen, for instance, by noting that there is a loop at $2$ in
  $\Gamma$, but not in $\Gamma^{(1\,5)(2\,3)}$, or by noting that the vertex $1$
  has the label \exLabel{white} in $\Gamma$, while it has the label
  \exLabel{black} in $\Gamma^{(1\,5)(2\,3)}$.

  The unique non-trivial automorphism of $\Gamma$ induced by $\Sym{\Omega}$ is
  the transposition $(1\,4)$.  Since the set
  $\Iso{\Gamma}{\Gamma^{(1\,5)(2\,3)}}$ of induced isomorphisms from $\Gamma$
  to $\Gamma^{(1\,5)(2\,3)}$ is the right coset of $\Auto{\Gamma}$ that contains
  $(1\,5)(2\,3)$, it follows that the second and final isomorphism from
  $\Gamma$ to $\Gamma^{(1\,5)(2\,3)}$ is the permutation
  $
  (1\,4\,5)(2\,3)
  =
  (1\,4) \cdot (1\,5)(2\,3)
  $.
  Indeed,
  \[
    \Gamma^{(1\,4) \cdot (1\,5)(2\,3)}
    = {\left(\Gamma^{(1\,4)}\right)}^{(1\,5)(2\,3)}
    =        \Gamma^{(1\,5)(2\,3)}.
  \]
\end{example}

We have chosen to build our techniques around labelled digraphs because then
they can be straightforwardly applied to a wide range of the graphs and digraphs
that occur in practice. This is because graphs, digraphs, and so-called
multigraphs and multidigraphs can be converted into labelled digraphs in such a way
that the sets of isomorphisms that we are interested in do not change.

\subsection{Orbital graphs}

Some previous work, such as that of Thei{\ss}en~\cite{theissen} and an
article~\cite{newrefiners} by the first three authors of this paper, shows that
orbital graphs can be useful for representing properties of groups and cosets
when performing a partition backtrack search in $\Sym{\Omega}$.

\begin{definition}[Orbital graph]
  Let $G \leq \Sym{\Omega}$, and let $\alpha, \beta \in \Omega$ be such that
  $\alpha \neq \beta$.  Then the \emph{orbital graph of $G$ with base-pair
    $(\alpha, \beta)$} is the digraph
  $\left(\Omega,\,\set{(\alpha^{g}, \beta^{g})}{g \in G}\right)$, which is
  denoted by $\Gamma(G, \Omega, (\alpha, \beta))$.
\end{definition}

Although an orbital graph is a \emph{digraph} rather than a \emph{graph}, we
retain the original name because it has become standard in the literature.
The next lemma is a well-known result about orbital graphs
(see for example~\cite[Section~3.2]{dixonmortimer} or~\cite[Lemma~17]{newrefiners}).

\begin{lemma}\label{lem-orbgraph}
  Let $G \leq \Sym{\Omega}$.  Then $G$ acts on each of its orbital graphs
  as an arc-transitive group of digraph automorphisms.
  (This means that, given any two arcs, there exists some $g \in G$
   mapping one to the other.)
\end{lemma}

\begin{figure}[!ht]
  \small
  \centering
  \begin{tikzpicture}
    \foreach \x in {1,2,3,4,5,6} {
        \node[circle, draw=black] (\x) at (-\x*60+120:1.3cm) {$\mathbf{\x}$};};
    \foreach \x/\y in {1/2,2/3,3/4,4/5,5/6,6/1} {\arc{\x}{\y}};
    \node at (0, -2) {$\Gamma\left(C_{6}, \{1, \ldots, 6\}, (1, 2)\right)$};
  \end{tikzpicture}
  \qquad
  \begin{tikzpicture}
    \foreach \x in {1,2,3,4,5,6} {
        \node[circle, draw=black] (\x) at (-\x*60+120:1.3cm) {$\mathbf{\x}$};};
    \foreach \x/\y in {1/3,2/4,3/5,4/6,5/1,6/2} {\arc{\x}{\y}};
    \node at (0, -2) {$\Gamma\left(C_{6}, \{1, \ldots, 6\}, (1, 3)\right)$};
  \end{tikzpicture}
  \qquad
  \begin{tikzpicture}
    \foreach \x in {1,2,3,4,5,6} {
        \node[circle, draw=black] (\x) at (-\x*60+120:1.3cm) {$\mathbf{\x}$};};
    \foreach \x/\y in {1/4,2/5,3/6} {\arcSym{\x}{\y}};
    \node at (0, -2) {$\Gamma\left(C_{6}, \{1, \ldots, 6\}, (1, 4)\right)$};
  \end{tikzpicture}
  \caption{\label{fig-orbitals-cyclic}Diagrams of three orbital graphs of the
    group $C_{6} \coloneqq \< (1\,2\,3\,4\,5\,6) \> \le \Sn{6}$
    with, from left to right, base-pairs $(1, 2)$, $(1, 3)$, and $(1, 4)$.
    The automorphism group
    of $\Gamma\left(C_{6}, \{1, \ldots, 6\}, (1, 2)\right)$ induced by
    $\Sn{6}$ is $C_{6}$, whereas
    the automorphism groups of the other two orbital graphs properly contain
    $C_{6}$.}
\end{figure}
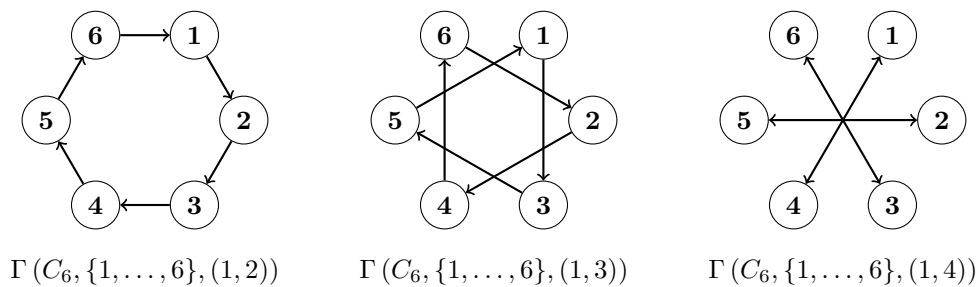

Lemma~\ref{lem-orbgraph} implies that the group of
digraph automorphisms (induced by $\Sym{\Omega}$) of an orbital graph of a group $G$ is
an overestimate for $G$.  Approximations obtained in this way can differ greatly
in their precision.
At one extreme, a group is \emph{absolutely orbital graph recognisable}, in
the terminology of~\cite{orbitalgraphs}, if it is equal to the induced automorphism
group of each of its orbital graphs.  The dihedral group $\< (1\,2\,3\,4),
(1\,3) \>$ of order $8$ in $\Sn{4}$ has this property, for
example.
At the other extreme, a group that acts $2$-transitively on $\Omega$ has a
unique orbital graph,
which contains every possible non-loop arc, and thus has automorphism group
$\Sym{\Omega}$.  It follows that $2$-transitive groups cannot be differentiated
by the automorphism groups of their orbital graphs.

There are many further kinds of behaviours between these extremes:
a group is called \emph{$2$-closed} if it is equal to the intersection of the
automorphism groups of its orbital graphs.  We can consider a $2$-closed group
to be one where the collection of its orbital graphs represents the group
exactly.  These groups are particularly well-suited to the techniques of our
paper, because they can be encoded in a stack of labelled digraphs capturing all relevant information. We will introduce this idea in Section~\ref{sec-stacks}.  Groups that are absolutely orbital graph
recognisable are $2$-closed, but there are many $2$-closed groups
that are not absolutely orbital graph recognisable.  The Klein four-group
$V \coloneqq \< (1\,2)(3\,4), (1\,3)(2\,4) \>$ is $2$-closed, for instance, even though
none of its orbital graphs has automorphism group equal to $V$.

\begin{figure}[!ht]
  \small
  \centering
  \begin{tikzpicture}
    \tikzstyle{white}=[circle, draw=black]

    \node[white] (11) at (0,   1.5) {$\mathbf{1}$};
    \node[white] (12) at (1.5, 1.5) {$\mathbf{2}$};
    \node[white] (13) at (0,     0) {$\mathbf{3}$};
    \node[white] (14) at (1.5,   0) {$\mathbf{4}$};
    \node (100) at (0.75, -0.75) {$\Gamma\left(V, \{1, 2, 3, 4\}, (1, 2)\right)$};
    \arcSym{11}{12}
    \arcSym{13}{14}

    \node[white] (21) at (4,   1.5) {$\mathbf{1}$};
    \node[white] (22) at (5.5, 1.5) {$\mathbf{2}$};
    \node[white] (23) at (4,     0) {$\mathbf{3}$};
    \node[white] (24) at (5.5,   0) {$\mathbf{4}$};
    \node (200) at (4.75, -0.75) {$\Gamma\left(V, \{1, 2, 3, 4\}, (1, 3)\right)$};
    \arcSym{21}{23}
    \arcSym{22}{24}

    \node[white] (31) at (8,   1.5) {$\mathbf{1}$};
    \node[white] (32) at (9.5, 1.5) {$\mathbf{2}$};
    \node[white] (33) at (8,     0) {$\mathbf{3}$};
    \node[white] (34) at (9.5,   0) {$\mathbf{4}$};
    \node (300) at (8.75, -0.75) {$\Gamma\left(V, \{1, 2, 3, 4\}, (1, 4)\right)$};
    \arcSym{31}{34}
    \arcSym{32}{33}
  \end{tikzpicture}
  \caption{\label{fig-orbitals-klein}The orbital graphs of the Klein four-group $V
  \coloneqq \< (1\,2)(3\,4), (1\,3)(2\,4) \>$.}\label{fig-klein-4}
\end{figure}

\begin{example}\label{ex-klein-4}
  The automorphism groups of the orbital graphs of
  the Klein four-group $V \coloneqq \< (1\,2)(3\,4), (1\,3)(2\,4) \>$
  are dihedral groups with $8$ elements.
  (See Figure~\ref{fig-klein-4}.)
  However, the intersection of any two such automorphism
  groups is $V$. Therefore $V$ is $2$-closed, but not absolutely orbital
  graph recognisable.
\end{example}

Any subgroup $G \leq \Sym{\Omega}$ leaves each of its orbits on $\Omega$
invariant. In other words, if $O_{1}, \ldots, O_{k} \subseteq \Omega$ are the
distinct orbits of $G$ on $\Omega$, then $G$ is contained in the
stabiliser
\(
  \set{g \in \Sym{\Omega}}{[O_{1}^{g}, \ldots, O_{k}^{g}] = [O_{1}, \ldots, O_{k}]}
\)
of $[O_{1}, \ldots, O_{k}]$ in $\Sym{\Omega}$, which is isomorphic to the direct
product $\Sym{O_{1}} \times \cdots \times \Sym{O_{k}}$.
As discussed later in Example~\ref{ex-perfect-list-of-sets},
stabilisers of this kind can be perfectly represented by labelled digraphs.  This means that,
for any non-transitive group $G$, we can use its orbits to produce a labelled
digraph whose automorphism group both contains $G$, and is \emph{properly}
contained in $\Sym{\Omega}$. In particular, this labelled digraph represents $G$
better than does any labelled digraph whose automorphism group is
$\Sym{\Omega}$, which is the worst possible case.

In~\cite{newrefiners}, the authors say that an orbital graph $\Gamma$ of a group
$G$ is \emph{futile} if and only if $\Auto{\Gamma}$ is the stabiliser of a list
of the orbits of $G$.  In essence, this means that the orbital graph is no
better at representing $G$ than the set of orbits of $G$.  Such an orbital graph
has little computational value, since the orbits of a group can be represented
by an ordered partition, which can be constructed, computed with, and stored
much more cheaply than can an orbital graph.

\section{Stacks of labelled digraphs}\label{sec-stacks}

In this section we introduce labelled digraph stacks.  We organise our search
algorithms around these stacks, much like how partition backtrack is organised
around stacks of ordered partitions.  The essential idea is to represent the
subsets of $\Sym{\Omega}$, for whose intersection we are searching, as the set
of isomorphisms from a suitable labelled digraph stack to another.  We explain
this in Section~\ref{sec-search}.

A \emph{labelled digraph stack} on $\Omega$ is a finite (possibly empty) list of
labelled digraphs on $\Omega$.  We denote the collection of all labelled digraph
stacks on $\Omega$ by $\Stacks{\Omega}$.  The \emph{length} of a labelled
digraph stack $\stackS$, written $|\stackS|$, is the number of entries that it
contains.  A labelled digraph stack of length $0$ is called \emph{empty}, and we
denote the empty labelled digraph stack on $\Omega$ by $\EmptyStack{\Omega}$.
We use a notation typical for lists, whereby if $i \in \{1, \ldots,
|\stackS|\}$, then $\stackS[i]$ denotes the $i^{\text{th}}$ labelled digraph in
the stack $\stackS$.

We allow any labelled digraph stack on $\Omega$ to be appended
onto the end of another.  If $\stackS, \stackT \in
\Stacks{\Omega}$ have lengths $k$ and $l$, respectively, then we define $\stackS
\Vert \stackT$ to be the labelled digraph stack
$\left[
    \stackS[1], \ldots, \stackS[k],
    \stackT[1], \ldots, \stackT[l]
    \right]$
of length $k + l$ formed by appending $\stackT$ to $\stackS$.

We define an action of $\Sym{\Omega}$ on $\Stacks{\Omega}$  via the action of
$\Sym{\Omega}$ on the set of all labelled digraphs on $\Omega$.  More
specifically, for all $\stackS \in \Stacks{\Omega}$ and $g \in \Sym{\Omega}$, we
define $\stackS^{g}$ to be the labelled digraph stack of length $|\stackS|$ with
$\stackS^{g}[i] = {\stackS[i]}^{g}$ for all $i \in \{1, \ldots, |\stackS|\}$.
In other words, $\stackS^{g}$ is the labelled digraph stack obtained from
$\stackS$ by applying $g$ to each of its entries.  An isomorphism from $\stackS$
to another labelled digraph stack $\stackT$ (induced by $\Sym{\Omega}$) is
therefore a permutation $g \in \Sym{\Omega}$ such that $\stackS^{g} = \stackT$.
In particular, only digraph stacks of equal lengths can be isomorphic, which
means that results concerning isomorphisms of labelled digraph stacks only need
to consider those with equal lengths.  Note that every permutation in
$\Sym{\Omega}$ induces an automorphism of $\EmptyStack{\Omega}$.

\begin{remark}\label{rmk-stack-iso-auto}
Let $\stackS, \stackT, \stackU, \stackV \in
  \Stacks{\Omega}$. It follows from the definitions that
\[
  \Iso{\stackS}{\stackT} =
  \begin{cases}
    \varnothing
     & \text{if\ } |\stackS| \neq |\stackT|, \\
    \bigcap_{i = 1}^{|\stackS|} \Iso{\stackS[i]}{\stackT[i]}
     & \text{if\ } |\stackS| = |\stackT|,
  \end{cases}
  \quad\text{and that\ }
  \Auto{\stackS} = \bigcap_{i = 1}^{|\stackS|}
  \Auto{\stackS[i]}.
\]
In addition $\Auto{\stackS \Vert \stackU} \leq \Auto{\stackS}$,
and if $|\stackS| = |\stackT|$, then
$\Iso{\stackS \Vert \stackU}{\stackT \Vert \stackV}
  \subseteq
  \Iso{\stackS}{\stackT}$.
Roughly speaking, the automorphism group of a labelled digraph stack, and the
set of isomorphisms from one labelled digraph stack to another one of equal
length, potentially become smaller as new entries are added to the stacks.
\end{remark}

We illustrate some of the foregoing concepts in Example~\ref{ex-stack}.

\begin{example}\label{ex-stack}
  Let $\Omega = \{1, \ldots, 6\}$ and $\labelSet = \{\exLabel{black},\,
  \exLabel{white},\,
  \exLabel{solid},\, \exLabel{dashed}\,\}$.  Here we define a
  labelled digraph stack $\stackS$ on $\Omega$ that has length $3$, by
  describing each of its members.

  We define the first entry of $\stackS$ via the orbital graph of $K \coloneqq
  \<(1\,2)(3\,4)(5\,6), (2\,4\,6) \>$ with base-pair $(1, 3)$. The automorphism
  group of this orbital graph (as always, induced by $\Sym{\Omega}$) is $K$
  itself; in other words, this orbital graph perfectly represents $K$ by its
  automorphism group.
  In order to define $\stackS[1]$, we convert this orbital graph into a labelled
  digraph by assigning the label \exLabel{white} to each vertex and
  assigning the label \exLabel{solid} to each arc.
  This does not change the automorphism group of the digraph.

  We define the second entry of $\stackS$ to be the labelled digraph on $\Omega$
  without arcs, whose vertices $1$ and $2$ are labelled \exLabel{black}, and
  whose remaining vertices are labelled \exLabel{white}.  The automorphism group
  of this labelled digraph is the setwise stabiliser of $\{1, 2\}$ in
  $\Sym{\Omega}$.

  We define the third entry of $\stackS$ to be the labelled digraph $\stackS[3]$
  shown in Figure~\ref{fig-ex-stack}, with arcs and labels as depicted there;
  its automorphism group is $\< (1\,2), (3\,4)(5\,6) \>$.

  Given the automorphism groups of the individual entries of $\stackS$, as
  described above, it follows that the automorphism group of $\stackS$  consists
  of precisely those elements of $K$ that stabilise the set $\{1, 2\}$, and that
  are automorphisms of the labelled digraph $\stackS[3]$. Hence this group is
  $\< (1\,2)(3\,4)(5\,6) \>$.
  Since $(1\,2)$ is an automorphism of $\stackS[2]$ and $\stackS[3]$, but not of
  $\stackS[1]$, it follows that $\stackS^{(1\,2)} = [{\stackS[1]}^{(1\,2)},
  \stackS[2], \stackS[3]] \neq \stackS$.
  We also note that $\Iso{\stackS}{\stackS^{(1\,2)}}$ is the right coset
  $\Auto{\stackS} \cdot (1\,2) = \{ (1\,2), (3\,4)(5\,6) \}$ of $\Auto{\stackS}$
  in $\Sym{\Omega}$.

  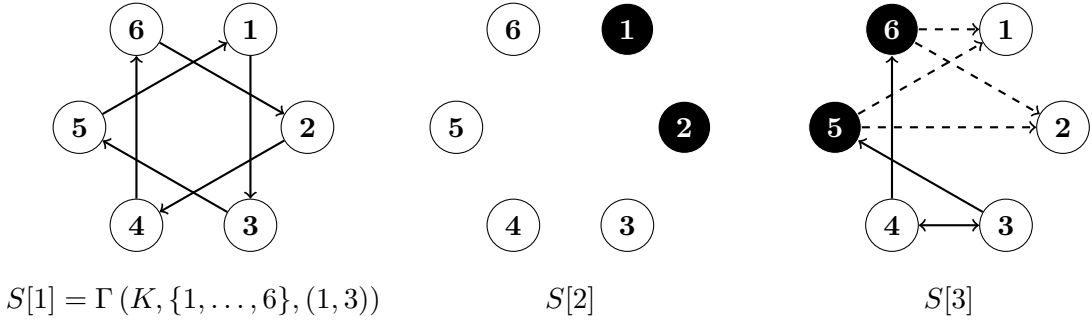
\begin{figure}[!ht]
    \centering
    \begin{tikzpicture}
      \foreach \x in {1,2,3,4,5,6} {
          \tikzstyle{white}=[circle, draw=black]
          \node[white] (\x) at (-\x*60+120:1.5cm) {$\mathbf{\x}$};};
      \foreach \x/\y in {1/3,2/4,3/5,4/6,5/1,6/2} {\arc{\x}{\y}};
      \node at (0, -2.3)
      {$\stackS[1] = \Gamma\left(K, \{1, \ldots, 6\}, (1, 3)\right)$};
    \end{tikzpicture}
    \quad
    \begin{tikzpicture}
      \tikzstyle{white}=[circle, draw=black]
      \tikzstyle{black}=[circle, draw=white, fill=black!100]
      \foreach \x in {1,2} {
          \node[black] (\x) at (-\x*60+120:1.5cm) {$\color{white}\mathbf{\x}$};};
      \foreach \x in {3,4,5,6} {
          \node[white] (\x) at (-\x*60+120:1.5cm) {$\mathbf{\x}$};};
      \node at (0, -2.3) {$\stackS[2]$};
    \end{tikzpicture}
    \qquad\quad
    \begin{tikzpicture}
      \tikzstyle{white}=[circle, draw=black]
      \tikzstyle{black}=[circle, draw=white, fill=black!100]
      \foreach \x in {1,2,3,4} {
          \node[white] (\x) at (-\x*60+120:1.5cm) {$\mathbf{\x}$};};
      \foreach \x in {5,6} {
          \node[black] (\x) at (-\x*60+120:1.5cm) {$\color{white}\mathbf{\x}$};};
      \arcDash{5}{1}
      \arcDash{6}{1}
      \arcDash{5}{2}
      \arcDash{6}{2}
      \arc{3}{5}
      \arc{4}{6}
      \arcSym{3}{4}
      \node at (0, -2.3) {$\stackS[3]$};
    \end{tikzpicture}
    \caption{\label{fig-ex-stack}Diagrams of the labelled digraphs in the
      labelled digraph stack $\stackS$ from Example~\ref{ex-stack}. The
      vertices and arcs of these labelled digraphs are styled according to their
      labels, which are chosen from the set $\{\exLabel{black},\,\exLabel{white},\,
        \exLabel{solid},\, \exLabel{dashed}\,\}$.}
  \end{figure}
\end{example}

As well as the obvious difference of being defined in terms of labelled digraphs
rather than ordered partitions, there are further conceptual differences between
labelled digraph stacks and the ordered partitions stacks that Leon uses
in~\cite{leon1997, leon1991} for his search algorithms.
For example, the entries of a labelled digraph stack on $\Omega$ can be any
labelled digraphs on $\Omega$, whereas each subsequent ordered partition in one
of Leon's ordered partition stacks is required to be finer than the previous
entry (for some definition of `finer').
We explore this further in Section~\ref{sec-squashed-stack}.
Furthermore, one can simply write down the automorphisms and isomorphisms of
ordered partition stacks induced by $\Sym{\Omega}$ with trivial
calculation, but this is computationally expensive
for labelled digraph stacks, in general.  This reflects the fact
that a far greater range of sets of permutations can be represented by labelled
digraph stacks than can be represented by ordered partition stacks.

\subsection{The squashed labelled digraph of a
  stack}\label{sec-squashed-stack}

As mentioned previously, in the definition of a labelled digraph stack, we have
not included any requirement of entries of a stack becoming `finer'.
This is because it can be computationally expensive to find out
the automorphism groups of labelled digraphs and their stacks,
and we therefore do not wish to require that the automorphism groups of a labelled
digraph stack and its entries are always known.

Moreover, without a requirement of entries becoming `finer', it is much easier to
append new
labelled digraphs to a stack, which is the primary topic of
Section~\ref{sec-refiners}.
The computational purpose
of extending a stack is simply to add new information about the current part of
the search space; there is no need to duplicate old information.  The
automorphism groups of the pre-existing entries of a stack can always be
obtained from the entries themselves, and so from this perspective, it is not
necessary for each new entry to contain old information about the previous
entries.

On the other hand, having a labelled digraph whose automorphism group is equal
to that of a given labelled digraph stack (analogous to the final entry of an
ordered partition stack) proves to be convenient for our exposition, especially
for Section~\ref{sec-approximations}, even though it is not fundamentally
required for the correctness of our algorithms. However, we define this special
labelled digraph to be a new object that is defined from the stack, rather than
being part of the stack itself.  More specifically, in the remainder of
Section~\ref{sec-squashed-stack}, we introduce a way of converting labelled
digraph stacks into labelled digraphs in a way that preserves isomorphisms.
This is a short way of saying that the sets of isomorphisms that we are interested in do not change in the process.

In order to make the following definition, we first fix a special symbol $\#
\not\in \labelSet$ that is never
to be used as the label of a vertex or an arc
in any labelled digraph.

\begin{definition}\label{defn-squashed}
  Let $\stackS$ be a labelled digraph stack on $\Omega$, with $\stackS[i]
  \coloneqq (\Omega, A_{i}, \labelFunc_{i})$ being some labelled digraph on
  $\Omega$ for each $i \in \{1, \ldots, |\stackS|\}$.  Then the \emph{squashed
  labelled digraph} of $\stackS$, denoted by $\Squash{\stackS}$, is the labelled
  digraph $(\Omega, A, \labelFunc)$, where
  \begin{itemize}
    \item

      $A = \bigcup_{i = 1}^{|\stackS|} A_{i}$,

    \item

      $\labelFunc(\delta) =
        [\labelFunc_{1}(\delta), \ldots,
        \labelFunc_{|\stackS|}(\delta)]$ for all $\delta \in \Omega$,
      and

    \item

        $\labelFunc(\alpha, \beta)$ is the list of length $|\stackS|$
        for all $(\alpha, \beta) \in \bigcup_{i = 1}^{|\stackS|} A_{i}$,
        where
        \[\labelFunc(\alpha, \beta)[i] =
          \begin{cases}
            \labelFunc_{i}(\alpha, \beta)
             & \text{\ if\ } (\alpha, \beta) \in A_{i}, \\
            \#
             & \text{\ if\ } (\alpha, \beta) \not\in A_{i},
          \end{cases}
          \quad\text{for all}\ i \in \{1, \ldots, |\stackS|\}.
        \]
  \end{itemize}
\end{definition}

Note that the labelling function of the squashed labelled digraph of a stack can
be used to
reconstruct all information about the stack from which it was created.  We
also point out that $\Squash{\stackS}^{g} = \Squash{\stackS^{g}}$ for all
$\stackS \in \Stacks{\Omega}$ and $g \in \Sym{\Omega}$.

In the following lemma, we prove that the set of isomorphisms induced by
$\Sym{\Omega}$ from one labelled digraph stack $\stackS$ to another $\stackT$
consists of exactly those elements of $\Sym{\Omega}$ that induce
isomorphisms from $\Squash{\stackS}$ to $\Squash{\stackT}$.

\begin{lemma}\label{lem-squash-same-iso}
  Let $\stackS, \stackT \in \Stacks{\Omega}$.  Then
  \[\Iso{\stackS}{\stackT}
  =
  \Iso{\Squash{\stackS}}{\Squash{\stackT}}.\]
\end{lemma}

\begin{proof}
  If $\stackS$ and $\stackT$ have different lengths, then they are
  non-isomorphic by definition, and $\Squash{\stackS}$ and
  $\Squash{\stackT}$ are non-isomorphic because their labels have
  different lengths.

  For the remainder of the proof, we suppose that $\stackS$ and
  $\stackT$ have some common length $k \in \N_{0}$.
  Let $\mu$ and $\nu$ denote the labelling functions of $\Squash{\stackS}$
  and $\Squash{\stackT}$, respectively, and for each $i \in \{1, \ldots,
    k\}$, let
  $\stackS[i] = (\Omega, A_{i}, \sigma_{i})$ and
  $\stackT[i] = (\Omega, B_{i}, \tau_{i})$.

  The sets whose equality we wish to prove are subsets of
  $\Sym{\Omega}$, so let $g \in \Sym{\Omega}$ be arbitrary. We prove that $g
    \in \Iso{\stackS}{\stackT}$ if and only if $g \in
    \Iso{\Squash{\stackS}}{\Squash{\stackT}}$ by just following the
  relevant definitions closely.
  \begin{align*}
    g \in \Iso{\stackS}{\stackT}
     & \Leftrightarrow\ {\stackS[i]}^{g} = \stackT[i]
    \text{\ for all\ } i \in \{1, \ldots, |\stackS|\}
    \\
     & \Leftrightarrow\
    A_{i}^{g} = B_{i}
    \text{\ and\ }
    \sigma_{i}^{g} = \tau_{i}
    \text{\ for each\ } i \in \{1, \ldots, k\}
    \\
     & \Leftrightarrow\
    A_{i}^{g} = B_{i},\
    \sigma_{i}(\delta) = \tau_{i}(\delta^{g}),
    \text{\ and\ }
    {\sigma_{i}(\alpha, \beta)}
    = \tau_{i}(\alpha^{g}, \beta^{g})
    \\
     & \hspace{1.5cm}
    \text{\ for each\ }
    i \in \{1, \ldots, k\},\
    \delta \in \Omega,
    \text{\ and\ }
    (\alpha, \beta) \in A_{i}
    \\
     & \Leftrightarrow\
    A_{i}^{g} = B_{i},\
    \mu(\delta) = \nu(\delta^{g}),
    \text{\ and\ }
    {\mu(\alpha, \beta)} = \nu(\alpha^{g}, \beta^{g})
    \\
     & \hspace{1.5cm}
    \text{\ for each\ }
    i \in \{1, \ldots, k\},\
    \delta \in \Omega,
    \text{\ and\ }
    (\alpha, \beta) \in A_{1} \cup \cdots \cup A_{k}
    \\
     & \Leftrightarrow\
    {(A_{1} \cup \cdots \cup A_{k})}^{g} =
    B_{1} \cup \cdots \cup B_{k},\
    \mu(\delta) = \nu(\delta^{g}),
    \text{\ and\ }
    \\
     & \hspace{1.5cm}
    {\mu(\alpha, \beta)} = \nu(\alpha^{g}, \beta^{g})
    \text{\ for each\ } \delta \in \Omega
    \text{\ and\ }
    (\alpha, \beta) \in A_{1} \cup \cdots \cup A_{k}
    \\
     & \Leftrightarrow\
    {(A_{1} \cup \cdots \cup A_{k})}^{g} =
    B_{1} \cup \cdots \cup B_{k}
    \text{\ and\ }
    \mu^{g} = \nu
    \\
     & \Leftrightarrow\ g \in
    \Iso{\Squash{\stackS}}{\Squash{\stackT}}.
    \qedhere
  \end{align*}
\end{proof}

\begin{example}\label{ex-squashed}
  Let $\stackS$ be the labelled digraph stack from
  Example~\ref{ex-stack}. Since $|\stackS| = 3$, the labels of vertices and
  arcs in $\Squash{\stackS}$ are lists of length $3$. The vertex labels of
  $\Squash{\stackS}$ are:
  \begin{itemize}
    \item
          $\labelFunc(1) = \labelFunc(2) =
            [\exLabel{white}, \exLabel{black}, \exLabel{white}]$,
          shown as \exLabel{black} in Figure~\ref{fig-ex-squashed},
    \item
          $\labelFunc(3) = \labelFunc(4) =
            [\exLabel{white}, \exLabel{white}, \exLabel{white}]$,
          shown as \exLabel{white} in Figure~\ref{fig-ex-squashed},
          and
    \item
          $\labelFunc(5) = \labelFunc(6) =
            [\exLabel{white}, \exLabel{white}, \exLabel{black}]$,
          shown as \exLabel{grey} in Figure~\ref{fig-ex-squashed}.
  \end{itemize}
  There are ten arcs in $\Squash{\stackS}$, which in total have five
  different labels:
  \begin{itemize}
    \item
          $\labelFunc(1, 3) = \labelFunc(2, 4) =
            [\exLabel{solid}, \#, \#]$,
          shown as \exLabel{thin} in Figure~\ref{fig-ex-squashed},
    \item
          $\labelFunc(3, 4) = \labelFunc(4, 3) = [\#, \#,
            \exLabel{solid}]$,
          shown as \exLabel{dotted} in Figure~\ref{fig-ex-squashed},
    \item
          $\labelFunc(5, 2) = \labelFunc(6, 1) = [\#, \#,
            \exLabel{dashed}]$,
          shown as \exLabel{dashed} in Figure~\ref{fig-ex-squashed},
    \item
          $\labelFunc(3, 5) = \labelFunc(4, 6) =
            [\exLabel{solid}, \#, \exLabel{solid}]$,
          shown as \exLabel{thick} in Figure~\ref{fig-ex-squashed}, and
    \item
          $\labelFunc(5, 1) = \labelFunc(6, 2) =
            [\exLabel{solid}, \#, \exLabel{dashed}]$,
          shown as \exLabel{wavy} in Figure~\ref{fig-ex-squashed}.
  \end{itemize}
  \begin{figure}[!ht]
    \centering
    \begin{tikzpicture}
      \tikzstyle{white}=[circle, draw=black]
      \tikzstyle{black}=[circle, draw=white, fill=black!100]
      \tikzstyle{grey}=[circle, draw=black, fill=gray!25]
      \foreach \x in {1,2} {
          \node[black] (\x) at (-\x*60+120:1.6cm) {$\color{white}\mathbf{\x}$};};
      \foreach \x in {3,4} {
          \node[white] (\x) at (-\x*60+120:1.6cm) {$\mathbf{\x}$};};
      \foreach \x in {5,6} {
          \node[grey] (\x) at (-\x*60+120:1.6cm) {$\mathbf{\x}$};};
      \arc{1}{3}
      \arc{2}{4}
      \arcSymDot{3}{4}
      \arcDash{5}{2}
      \arcDash{6}{1}
      \arcThick{3}{5}
      \arcThick{4}{6}
      \arcWavy{5}{1}
      \arcWavy{6}{2}
    \end{tikzpicture}
    \caption{\label{fig-ex-squashed}
      A depiction of the squashed labelled digraph
      $\Squash{\stackS}$ from Example~\ref{ex-squashed}, which is
      constructed from the labelled digraph stack $\stackS$ from
      Example~\ref{ex-stack}.}
  \end{figure}
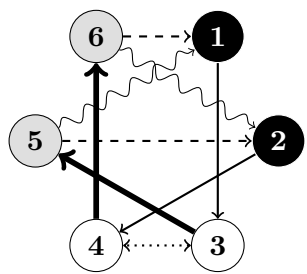
  Since automorphisms of labelled digraphs preserve the sets of vertices with
  any particular label, it is clear that $\Auto{\Squash{\stackS}} \leq \<
  (1\,2), (3\,4), (5\,6) \>$. This containment is proper, since
  $\Auto{\Squash{\stackS}} = \Auto{\stackS}$ by Lemma~\ref{lem-squash-same-iso},
  and $\Auto{\stackS} = \<(1\,2)(3\,4)(5\,6)\>$, as discussed in
  Example~\ref{ex-stack}.
  Indeed, inspection of the arc labels in $\Squash{\stackS}$ shows that any
  automorphism that interchanges the pair of points in any of $\{1, 2\}$, $\{3,
  4\}$, or $\{5, 6\}$ also interchanges the other pairs.
\end{example}

\section{Approximating isomorphisms and fixed points of
  stacks}\label{sec-approximations}

One might assume that organising a search around some kind of object (where the
set of elements we are searching for is overestimated by the set of isomorphisms from one such object to another)
requires knowing exactly what these isomorphisms are.  When searching with
labelled digraphs stacks, for instance, this would entail performing many
potentially-expensive labelled digraph isomorphism computations. However, as we
show in this paper, this is not necessary.  One may instead overestimate the
set of isomorphisms rather than compute them exactly.  Unsurprisingly, worse
approximations typically lead to larger searches, but since an overestimate of
an overestimate is again just an overestimate, doing this does
not significantly change the search technique.

There is therefore a trade-off
between the accuracy of such
overestimates, and the amount of effort spent in computing them. In
Definition~\ref{defn-approx-iso}, we introduce the concept of an isomorphism
approximator for pairs of labelled digraphs stacks, which is a vital
component of the algorithms described in Section~\ref{sec-search}. Later
we give several examples of such functions.

\begin{definition}\label{defn-approx-iso}
  An \emph{isomorphism approximator} for labelled digraph stacks is a function
  $\approxFunc$ that maps a pair of labelled digraph stacks on $\Omega$
  to either the empty set $\varnothing$, or a right coset of a subgroup of
  $\Sym{\Omega}$,
  such that the following statements hold for all $\stackS, \stackT \in
  \Stacks{\Omega}$
  (we write $\Approx{\stackS}{\stackS}$ as an abbreviation for
  $\textsc{Approx}(\stackS, \stackS)$):
  \begin{enumerate}[label=\textrm{(\roman*)}]
    \item\label{item-approx-true-overestimate}
      $\Iso{\stackS}{\stackT} \subseteq \Approx{\stackS}{\stackT}$.

    \item\label{item-approx-different-lengths}
      If $|\stackS| \neq |\stackT|$, then
      $\Approx{\stackS}{\stackT} = \varnothing$.

    \item\label{item-approx-right-coset-of-aut}
      If $\Approx{\stackS}{\stackT} \neq \varnothing$, then
      $\Approx{\stackS}{\stackT} = \Approx{\stackS}{\stackS} \cdot h$ for some
      $h \in \Sym{\Omega}$.
  \end{enumerate}
\end{definition}

Let $\approxFunc$ be an isomorphism approximator and let $\stackS, \stackT \in
\Stacks{\Omega}$.
As discussed previously, the set $\Iso{\stackS}{\stackT}$ of isomorphisms
induced by $\Sym{\Omega}$ from $\stackS$ to $\stackT$ is either empty, or it is
a right coset of $\Auto{\stackS}$.
Since $\idOmega \in \Iso{\stackS}{\stackS} = \Auto{\stackS}$, it follows by
definition that $\Approx{\stackS}{\stackS}$ is a subgroup of $\Sym{\Omega}$ that
contains $\Auto{\stackS}$, the automorphism group of $\stackS$ induced by
$\Sym{\Omega}$.  In other words, $\Approx{\stackS}{\stackS}$ is an overestimate
for $\Auto{\stackS}$.
The value of $\Approx{\stackS}{\stackT}$ should be interpreted as follows.  By
Definition~\ref{defn-approx-iso}\ref{item-approx-true-overestimate},
$\Approx{\stackS}{\stackT}$
gives a true overestimate for $\Iso{\stackS}{\stackT}$.  Therefore, if
$\Approx{\stackS}{\stackT} = \varnothing$, then the approximator has correctly
determined that $\stackS$ and $\stackT$ are non-isomorphic.  In particular,
by Definition~\ref{defn-approx-iso}\ref{item-approx-different-lengths},
an isomorphism approximator correctly determines that stacks of different
lengths are non-isomorphic. Otherwise, the approximator returns a right coset
in $\Sym{\Omega}$ of its overestimate for $\Auto{\stackS}$.

For practical purposes, it is most convenient for a computer
implementation of an isomorphism approximator to return a coset of the form
$\Approx{\stackS}{\stackS} \cdot h$ by explicitly giving the group
$\Approx{\stackS}{\stackS}$, typically by a list of generators,
along with a coset representative.

Any sensible isomorphism approximator returns $\varnothing$ for labelled
digraph stacks where the $i$\textsuperscript{th} entries contain different
numbers of arcs or vertices with any label.  However, for simplicity,
the definition only contains conditions that our techniques require.

In Section~\ref{sec-refiners-via-stack}, we require the ability to approximate
a set of fixed points of the automorphism group of any labelled digraph stack.
A point $\omega \in \Omega$ is \emph{a fixed point of a
subgroup} $G \leq \Sym{\Omega}$ if and only if $\omega^{g}=\omega$ for all $g
\in G$.
This is particularly useful when it comes to using orbits and orbital graphs in
our search techniques. For stacks of ordered partitions, it is possible to
simply read
off the fixed points, but once again, this is something that can be much more
computationally expensive for
stacks of labelled digraphs.  Therefore we introduce the following
definition.

\begin{definition}\label{defn-approx-fixed}
  A \emph{fixed-point approximator} for labelled digraph stacks is a function
  $\fixedFunc$ that maps each labelled digraph stack on $\Omega$ to
  a finite list in $\Omega$,
  such that for each $\stackS \in \Stacks{\Omega}$:
  \begin{enumerate}[label=\textrm{(\roman*)}]

    \item\label{item-fixed}
      Each entry in \(\Fixed{\stackS}\) is a fixed point of $\Auto{\stackS}$,
      and

    \item\label{item-fixed-invariant}
      \({\Fixed{\stackS}}^{g} = \Fixed{{\stackS}^{g}}\) for all \(g
      \in \Sym{\Omega}\).

  \end{enumerate}
\end{definition}

Definition~\ref{defn-approx-fixed}\ref{item-fixed-invariant}
ensures that a fixed-point approximator is compatible with the techniques that
we describe in Section~\ref{sec-refiners-via-stack}.  A fixed-point approximator
is permitted to return lists with duplicate entries, although duplicate entries
would seem to have no practical benefit.

\subsection{Computing automorphisms and isomorphisms
  exactly}\label{sec-exact-approx}

For Definition~\ref{defn-nauty-approx}, we require the concept of a canoniser of
labelled digraphs.

\begin{definition}\label{defn-canoniser}
  A \emph{canoniser} of labelled digraphs is a function $\canonFunc$ from
  the set of labelled digraphs on $\Omega$ to $\Sym{\Omega}$ such that, for all
  labelled digraphs $\Gamma$ and $\Delta$, $\Gamma^{\Canon{\Gamma}} =
    \Delta^{\Canon{\Delta}}$ if and only if $\Gamma$ and $\Delta$ are
  isomorphic.
\end{definition}

In essence, a canoniser assigns each object to a permutation that maps the
object to some canonically chosen member of its isomorphism class.
Canonisers are defined analogously for vertex-labelled digraphs
(i.e.\ digraphs where the labelling function is defined on the set of vertices
only).
There are several widely-used computational tools for canonising vertex-labelled
digraphs,
such as \bliss~\cite{bliss} and \nauty~\cite{practical2}.
These tools
compute the automorphism group of a vertex-labelled digraph at the same time as
they canonise it.
Since it is relatively easy to convert labelled digraphs into vertex-labelled
digraphs in a way that preserves isomorphisms, it is possible to use such tools
to canonise and compute automorphism groups of labelled digraphs.

\begin{definition}[Canonising and computing automorphisms
    exactly]\label{defn-nauty-approx}
  Let $\canonFunc$ be a canoniser of labelled digraphs.
  We define functions $\fixedFunc$ and $\approxFunc$ as follows:
  for all $\stackS, \stackT \in \Stacks{\Omega}$, let $g =
  \Canon{\Squash{\stackS}}$ and $h = \Canon{\Squash{\stackT}}$,
  let $L$ be the list
  $[i \in \Omega\,:\,i \text{\ is fixed by\,}\Auto{\Squash{\stackS}^{g}}]$,
  ordered as usual in $\Omega$,
  and define
  \begin{align*}
    \Fixed{\stackS}
     & = L^{g^{-1}}, \text{\ and} \\
    \Approx{\stackS}{\stackT}
     & =
    \begin{cases}
      \Auto{\Squash{\stackS}} \cdot g h^{-1}
       & \text{if}\ \Squash{\stackS}^{g} = \Squash{\stackT}^{h},    \\
      \varnothing
       & \text{otherwise.}
    \end{cases}
  \end{align*}
\end{definition}

\begin{lemma}\label{lem-nauty-approx}
  Let the functions $\approxFunc$ and $\fixedFunc$ be given as in
  Definition~\ref{defn-nauty-approx}.
  Then $\approxFunc$ is an isomorphism approximator, and
  $\fixedFunc$ is a fixed-point approximator.
  Moreover, for all $\stackS, \stackT \in \Stacks{\Omega}$,
  $\Approx{\stackS}{\stackT} = \Iso{\stackS}{\stackT}$.
\end{lemma}

\begin{proof}
  Throughout the proof, we repeatedly use Lemma~\ref{lem-squash-same-iso} and
  Definition~\ref{defn-canoniser}. As in Definition~\ref{defn-canoniser}, let $g
  = \Canon{\Squash{\stackS}}$ and $h = \Canon{\Squash{\stackT}}$.

  First, we show that $\Approx{\stackS}{\stackT} = \Iso{\stackS}{\stackT}$,
  which implies that
  Definition~\ref{defn-approx-iso}\ref{item-approx-true-overestimate}
  and~\ref{item-approx-different-lengths} hold.
  If $\stackS \not\cong \stackT$, then $\Squash{\stackS}^{g} \neq
  \Squash{\stackT}^{h}$, and so $\Approx{\stackS}{\stackT} =
  \Iso{\stackS}{\stackT} = \varnothing$.
  Otherwise $\stackS \cong \stackT$, in which case $g h^{-1} \in
  \Iso{\Squash{\stackS}}{\Squash{\stackT}} = \Iso{\stackS}{\stackT}$.
  Therefore
  \[
  \Approx{\stackS}{\stackT}
  =
  \Auto{\Squash{\stackS}} \cdot gh^{-1}
  =
  \Auto{\stackS} \cdot gh^{-1}
  =
  \Iso{\stackS}{\stackT}.
  \]
  Definition~\ref{defn-approx-iso}\ref{item-approx-right-coset-of-aut} clearly
  holds.
  Therefore $\approxFunc$ is an isomorphism approximator.

  Define
  $L = {[i \in \Omega\,:\,i \text{\ is fixed by\,}
        \Auto{\Squash{\stackS}^{g}}]}$,
  ordered as usual in $\Omega$.
  Since
  $
  \Auto{\stackS}^{g}
  =
  \Auto{\Squash{\stackS}}^{g}
  =
  \Auto{\Squash{\stackS}^{g}}
  $,
  it follows that $L$ consists of
  fixed points of $\Auto{\stackS}^{g}$, and so $\Fixed{\stackS}$ (which equals
    $L^{g^{-1}}$) consists of fixed points of $\Auto{\stackS}$.
  Therefore Definition~\ref{defn-approx-fixed}\ref{item-fixed} holds.  To show
  that Definition~\ref{defn-approx-fixed}\ref{item-fixed-invariant} holds, let
  $x \in \Sym{\Omega}$ be arbitrary and define $r =
  \Canon{\Squash{\stackS^{x}}}$.  Since
  $\Squash{\stackS}$ and $\Squash{\stackS^{x}}$ are isomorphic, it
  follows that $\Squash{\stackS}^{g} = \Squash{\stackS^{x}}^{r}$.  In
  particular,
  $L = {[i \in \Omega\,:\,i \text{\ is fixed by\,}
        \Auto{\Squash{\stackS^{x}}^{r}}]}$,
  and $g^{-1} x r$ is an automorphism of $\Squash{\stackS}^{g}$,
  which means that $g^{-1} x r$ fixes every entry of $L$.  Thus
  \[
    {\Fixed{\stackS}}^{x}
    = L^{g^{-1} x}
    = L^{(g^{-1} x r) r^{-1}}
    = L^{r^{-1}}
    = \Fixed{\stackS^{x}}. \qedhere
  \]
\end{proof}

\subsection{Approximations via equitable vertex labellings}\label{sec-equitable}

In order to present the approximator functions of this section, we require the
notion of an equitable vertex labelling for a labelled digraph.
Here we use the term \emph{vertex labelling} as an abbreviation for the restriction
of a digraph labelling function to the set of vertices, $\Omega$.

\subsubsection{Equitable vertex labellings}

\begin{definition}\label{defn-equitable}
  The vertex labelling of a labelled digraph $(\Omega, A,
    \labelFunc)$ is \emph{equitable} if and only if, for all vertices $\alpha,
    \beta \in \Omega$ with the same label, and for all vertex labels $y$ and arc
  labels $z$:
  \begin{align*}
     & |\set{(\alpha, \delta) \in A}{\labelFunc(\delta) = y\ \text{and}\
      \labelFunc(\alpha, \delta) = z}|
    = \\ & \hspace{5cm}
    |\set{(\beta,  \delta) \in A}{\labelFunc(\delta) = y\ \text{and}\
      \labelFunc(\beta,  \delta) = z}|,\ \text{and} \\
     & |\set{(\delta, \alpha) \in A}{\labelFunc(\delta) = y\ \text{and}\
      \labelFunc(\delta, \alpha) = z}|
    = \\ & \hspace{5cm}
    |\set{(\delta, \beta) \in A}{\labelFunc(\delta) = y\ \text{and}\
      \labelFunc(\delta, \beta) = z}|.
  \end{align*}
  In other words, the vertex labelling is equitable if and only if, for all
  vertex labels $x$ and $y$ and arc labels $z$, every vertex with label $x$ has
  some common number of \emph{out-neighbours} with label $y$ via arcs with label
  $z$, and similarly, every vertex with label $x$ has some common number of
  \emph{in-neighbours} with label $y$ via arcs with label $z$.
\end{definition}

By including arc labels, Definition~\ref{defn-equitable} extends the well-known
concepts of equitable colourings~\cite[Section~3.1]{practical2} and
partitions~\cite[Defintion~29]{newrefiners} of vertex-labelled graphs and
digraphs, and enables us to estimate
automorphism groups and
sets of isomorphisms.

It is possible to define a procedure that takes a labelled digraph $\Gamma$, and
returns a new equitable vertex labelling for $\Gamma$, where vertices with the
same equitable label have the same original label in $\Gamma$. The approximation for
$\Auto{\Gamma}$ that can be obtained from such an equitable vertex labelling
procedure turns out to be a potentially better approximation for
$\Auto{\Gamma}$ than the one
derived from the original vertex labelling.
We present an example of such a procedure in Algorithm~\ref{alg-equitable},
which is an adaptation of existing algorithms for computing equitable partitions
of vertex-labelled digraphs, such as those
in~\cite[Algorithm~1]{practical2} and~\cite[Algorithm~2]{newrefiners}.

In the following lemma, we present several properties of the function defined by
Algorithm~\ref{alg-equitable}, and then we present and discuss the algorithm.
Note that~\ref{item-equitable-preserves} and~\ref{item-equitable-isos} follow
from~\ref{item-equitable-map}, which itself follows from the careful ordering of
the lists in Algorithm~\ref{alg-equitable}.
The proof is otherwise omitted, because it is
mathematically straightforward.

\begin{lemma}\label{lem-equitable}
  Let $\equitableFunc$ be the function defined by Algorithm~\ref{alg-equitable},
  and let $\Gamma$ and $\Delta$ be labelled digraphs on $\Omega$.
  Then there exist $k, l \in \N_{0}$, labels $x_{1}, \ldots, x_{k}, y_{1},
  \ldots, y_{l}$, and subsets $U_{1}, \ldots, U_{k}, V_{1}, \ldots, V_{l}
  \subseteq \Omega$ such that
  \[
    \Equitable{\Gamma} =
    [(x_{1}, U_{1}), \ldots, (x_{k}, U_{k})]
    \ \text{and}\
    \Equitable{\Delta} =
    [(y_{1}, V_{1}), \ldots, (y_{l}, V_{l})].
  \]
  Then the following hold:
  \begin{enumerate}[label=\textrm{(\roman*)}]
    \item
      $\Equitable{\Gamma}$ defines an equitable vertex-labelling for $\Gamma$.

    \item\label{item-equitable-map}
      $\Equitable{{\Gamma}^{g}} =
        [(x_{1}, U_{1}^{g}), \ldots, (x_{k}, U_{k}^{g})]$
        for all $g \in \Sym{\Omega}$.

    \item\label{item-equitable-preserves}
      $\Auto{\Gamma} \leq
       \set{g \in \Sym{\Omega}}
                     {[O_{1}^{g}, \ldots, O_{k}^{g}] = [O_{1}, \ldots, O_{k}]}$.

    \item\label{item-equitable-isos}
      $\Iso{\Gamma}{\Delta}
        \begin{cases}
          = \varnothing,
          \quad\text{if}\ k \neq l
          \ \text{or}\ x_{i} \neq y_{i}
          \ \text{for any}\ i,   \\
          \subseteq
            \set{g \in \Sym{\Omega}}
                {[U_{1}^{g}, \ldots, U_{k}^{g}] = [V_{1}, \ldots, V_{k}]},
           \quad\text{otherwise}.
        \end{cases}$
  \end{enumerate}
\end{lemma}

\begin{algorithm}[!ht]
  \caption{$\equitableFunc$: Equitable vertex labelling for
     a labelled digraph.}\label{alg-equitable}
  \begin{algorithmic}[1]
    \item[\textbf{Input:}]
    A labelled digraph $\Gamma \coloneqq (\Omega, A, \labelFunc)$,
    with labels from a totally-ordered set.

    \item[\textbf{Output:}]
    A list that defines an equitable vertex labelling for $\Gamma$,
    such that:

    vertices with the same equitable label have the same original label, and

    vertices in the same orbit of $\Auto{\Gamma}$ have the same equitable label.

    \vspace{1mm}

    \State{$\textsc{NewLabels}\coloneqq
    \set{(x, \{\alpha\in\Omega  :  \labelFunc{(\alpha)}  =  x\})}
        {x \in \labelFunc(\Omega)}$, a set of pairs.}

    \State{Convert \textsc{NewLabels} into a list, ordered by first component.}

    \State{$\textsc{ToProcess} \coloneqq \textsc{NewLabels}$.}

    \While{$\textsc{ToProcess}$ is non-empty and
      $|\textsc{NewLabels}| < |\Omega|$}

    \State{Remove the first entry $(x, U)$ of $\textsc{ToProcess}$.}

    \State{$L \coloneqq
    \set{\labelFunc{(\alpha, \beta)}}
        {(\alpha,\beta)\in A,\,\text{and}\,\alpha\in U\,\text{or}\,\beta\in U}$.}

    \State{Convert $L$ into a list, ordered by the ordering of labels.}

    \For{$(y, V) \in \textsc{NewLabels}$}

    \For{$\alpha \in V$ and $i \in \{1, \ldots, |L|\}$}

      \State{$f(\alpha)[i] \coloneqq ( |\{\beta\in
      U\!:\!\labelFunc{(\alpha,\!\beta)}\!=\!L[i]\}|,\, |\{\beta\in
      U\!:\!\labelFunc{(\beta,\!\alpha)}\!=\!L[i]\}|)$.}

      {\hfill\emph{$\triangleright$ $f$ is a function, and
      $f(\alpha)$ is a list of $|L|$ elements of
      $\N_{0} \times \N_{0}$.}}

    \EndFor{}

    \State{Partition $V$ into $V_{1},\ldots,V_{k}$ according to, and
      ordered lexicographically by, $f$.}

    {\hfill\emph{$\triangleright$ for all $\alpha, \beta \in V$,
      there exist unique $i, j \in \{1,\ldots,k\}$
      with $\alpha \in V_{i}$ and $\beta \in V_{j}$;}}

    {\hfill\emph{$i < j$ if and only if
      $f(\alpha) < f(\beta)$.
      Note that f-values are totally ordered.}}

    \For{$i \in \{1, \ldots, k\}$}

      \State{\label{line-new-label}
      $y_{i} \coloneqq [y, x, L, f(\min(V_{i}))]$}
      \Comment{\emph{$y_{i}$ is the new label for the vertices in $V_{i}$.}}

    \EndFor{}

    \State{Replace $(y, V)$ in $\textsc{NewLabels}$ by
      $(y_{1}, V_{1}), \ldots, (y_{k}, V_{k})$, in this order.}

    \If{$k > 1$}

    \State{Remove $(y, V)$ from $\textsc{ToProcess}$, if present.}

    \State{Add $(y_{1}, V_{1})$, \ldots, $(y_{k}, V_{k})$
           to the end of $\textsc{ToProcess}$, in this order.}

    \EndIf{}
    \EndFor{}
    \EndWhile{}
    \State{\Return{\textsc{NewLabels}.}}
  \end{algorithmic}
\end{algorithm}

To summarise, given a labelled digraph, Algorithm~\ref{alg-equitable} repeatedly
tests whether each set of vertices with the same label satisfies the condition
in Definition~\ref{defn-equitable}.  For each such set and label, either the
condition is satisfied, and a new label for this set is devised that encodes
the old label and information about how the condition was satisfied, or the
condition is not satisfied, and the vertices are given new labels that
encode the old label and information about why the new labels were
created.

By choosing meaningful vertex labels this way, rather than retaining the
existing labels and defining new labels arbitrarily, we can distinguish
more pairs of labelled digraphs as non-isomorphic via
Lemma~\ref{lem-equitable}\ref{item-equitable-isos}. The next
example illustrates this
principle.

\begin{example}\label{ex-equitable-iso}
  Let $\Gamma$ be the labelled digraph on $\Omega$ with all possible
  arcs, and let $\Delta$ be the labelled digraph on $\Omega$ without arcs,
  where every vertex and arc in $\Gamma$ and $\Delta$ has the label $x$,
  for some arbitrary but fixed $x \in \labelSet$.
  Then Lemma~\ref{lem-equitable}\ref{item-equitable-isos} allows us to
  algorithmically deduce that $\Gamma$ and $\Delta$ are non-isomorphic, even
  though both are regular (i.e.\ every vertex has a common number of
  in-neighbours, and a common number of out-neighbours), and they even have the
  same induced automorphism group, namely $\Sym{\Omega}$.
  The $\equitableFunc$ procedure from Algorithm~\ref{alg-equitable} assigns the
  vertices in $\Gamma$ a label
  that encodes that each vertex has $|\Omega|$ in- and out-neighbours, and it
  assigns the vertices in $\Delta$ a label that encodes that each vertex has
  no in- or out-neighbours. Therefore, the labels given by $\Equitable{\Gamma}$ and
  $\Equitable{\Delta}$ are different, and so $\Gamma$ and $\Delta$ are
  non-isomorphic by Lemma~\ref{lem-equitable}\ref{item-equitable-isos}.
  \textit{A note of warning:} the choice of new labels plays a role!
  If new labels were instead, say, chosen to be incrementally increasing
  integers starting at $1$, then we would have $\Equitable{\Gamma} =
  \Equitable{\Delta}$, and the deduction that we explained above would not be
  possible.
\end{example}

In the previous example it is obvious to us the digraphs are non-isomorphic,
but for many more complicated
examples, Lemma~\ref{lem-equitable}\ref{item-equitable-isos} can still be used
to detect less obvious non-isomorphism.

\subsubsection{Strong and weak approximations via equitable vertex labelling}

\begin{definition}[Strong equitable labelling]\label{defn-strong-approx}
 Let $\equitableFunc$ be the function defined by
  Algorithm~\ref{alg-equitable}, and let $\stackS, \stackT \in
    \Stacks{\Omega}$.
  Then there exist $k, l \in \N_{0}$, labels $x_{1}, \ldots, x_{k}$, and
  $y_{1},
  \ldots, y_{l}$, and subsets $U_{1}, \ldots, U_{k}, V_{1}, \ldots, V_{l}
  \subseteq \Omega$ such that
  \begin{align*}
    \Equitable{\Squash{\stackS}} & =
    [(x_{1}, U_{1}), \ldots, (x_{k}, U_{k})],
    \ \text{and}                     \\
    \Equitable{\Squash{\stackT}} & =
    [(y_{1}, V_{1}), \ldots, (y_{l}, V_{l})].
  \end{align*}
  Let $G$ denote the stabiliser of the list $[U_{1}, \ldots, U_{k}]$ in
  $\Sym{\Omega}$,
  and define
  \[
    \Approx{\stackS}{\stackT} =
    \begin{cases}
      G \cdot h
                  & \text{if\ }
                    |\stackS| = |\stackT|,\
                    k = l,
                    \text{\ and for all}\ i,\,
                    x_{i} = y_{i} \ \text{and}\ |U_{i}| = |V_{i}|; \\
      \varnothing & \text{otherwise,}
    \end{cases}
  \]
  where $h \in \Sym{\Omega}$ is any permutation with the property that
  $U_{i}^{h} = V_{i}$ for each $i \in \{1,\ldots,k\}$.  Note that for all $g, h
  \in \Sym{\Omega}$, $U_{i}^{g} = U_{i}^{h}$ for all $i$ if and only if $g$ and
  $h$ represent the same right coset of $G$ in $\Sym{\Omega}$.
  Finally, we define
  \[\Fixed{\stackS} = [u_{i_{1}}, \ldots, u_{i_{m}}],\]
  where $i_{1} < \cdots < i_{m}$ and
  the sets $U_{i_{j}} = \{u_{i_{j}}\}$ for each $j \in \{1,\ldots,m\}$ are
  exactly the singletons amongst $U_{1}, \ldots, U_{k}$.
\end{definition}

\begin{definition}[Weak equitable labelling]\label{defn-weak-approx}
  Let $\equitableFunc$ be the function defined by Algorithm~\ref{alg-equitable},
  and let $\stackS, \stackT \in \Stacks{\Omega}$.  For each $i \in
    \{1, \ldots, |\stackS|\}$, $j \in \{1, \ldots, |\stackT|\}$,
  there exist $k_{i}, l_{j} \in \N_{0}$, labels $x_{i,1}, \ldots, x_{i,k_{i}},
  y_{j,1}, \ldots, y_{j,l_{j}}$, and subsets
  $U_{i,1}, \ldots, U_{i,k_{i}}, V_{j,1}, \ldots, V_{j,l_{j}} \subseteq \Omega$ such that
  \begin{align*}
    \Equitable{\stackS[i]} & =
    [(x_{i, 1}, U_{i, 1}), \ldots, (x_{i, k_{i}}, U_{i, k_{i}})],
    \ \text{and}               \\
    \Equitable{\stackT[j]} & =
    [(y_{j, 1}, V_{j, 1}), \ldots, (y_{j, l_{j}}, V_{j, l_{j}})].
  \end{align*}

  If either $|\stackS| \neq |\stackT|$,
  or else if $k_{i} \neq l_{i}$ for some $i \in \{1, \ldots, |\stackS|\}$,
  or else if $x_{i, j} \neq y_{i, j}$ for some $i \in \{1, \ldots, |\stackS|\}$
  and $j \in \{1, \ldots, k_{i}\}$, then we define
  $\Approx{\stackS}{\stackT} = \varnothing$.
  Otherwise, we proceed by `intersecting' the equitable vertex
  labellings for $\stackS$, and we do the same with those for $\stackT$.

  More specifically, we define functions $f$ and $g$ that map vertices to lists
  of finite length with entries in $\N$. For each $\alpha \in \Omega$, we define
  $f(\alpha)$ to be a list of
  length $|\stackS|$ where, for each $i \in \{1,\ldots,|S|\}$, $f(\alpha)[i]$ is the unique $j \in \{1, \ldots, k_{i}\}$ such that $\alpha \in
  U_{i, j}$.  Similarly, for each $\alpha \in \Omega$, we define $g(\alpha)$ to
  be a list of length $|\stackT|$ where, for each $i \in \{1,\ldots,|T|\}$, $g(\alpha)[i]$ is the unique $j \in \{1, \ldots, k_{i}\}$ such that
  $\alpha \in V_{i, j}$.  Therefore $f$ and $g$, respectively, encode the
  equitable label of a vertex at each level of $\stackS$ and $\stackT$.
  Then we define subsets $W_{1}, \ldots, W_{m}$ of $\Omega$ according to, and
  ordered lexicographically by, $f$-value, and similarly we define subsets
  $T_{1}, \ldots, T_{n}$ of $\Omega$ via $g$.

  Given all of this, we let $G$ denote the stabiliser of $[W_{1}, \ldots, W_{m}]$
  in $\Sym{\Omega}$ and define
  \[
    \Approx{\stackS}{\stackT} =
    \begin{cases}
      G \cdot h
                  & \text{if}\
                    |\stackS| = |\stackT|,\ m = n,\ \text{and for all}\ i, \\
                  & \quad |W_{i}| = |T_{i}|\
                    \text{and}\ f(\min(W_{i})) = g(\min(T_{i})), \\
      \varnothing & \text{otherwise,}
    \end{cases}
  \]
  where $h \in \Sym{\Omega}$ is any permutation with the property that
  $W_{i}^{h} = T_{i}$ and $\min(W_{i})$ is the minimum with respect to the
  ordering of $\Omega$. Finally, we define
  \[\Fixed{\stackS} = [w_{i_{1}}, \ldots, w_{i_{t}}],\]
  where $i_{1} < \cdots < i_{t}$ and the sets $W_{i_{j}} =
  \{w_{i_{j}}\}$ for each $j \in \{1,\ldots,t\}$ are exactly the singletons
  amongst $W_{1}, \ldots, W_{m}$.
\end{definition}

The following lemma holds by Lemma~\ref{lem-equitable}.

\begin{lemma}\label{lem-weak-strong-approx}
  The functions $\approxFunc$ from Definitions~\ref{defn-strong-approx}
  and~\ref{defn-weak-approx} are isomorphism approximators, and the functions
  $\fixedFunc$ are fixed-point approximators.
\end{lemma}

\subsection{Comparing approximators}\label{sec-approx-comparison}

In this section, we give an example that compares the isomorphism
approximators from Sections~\ref{sec-exact-approx} and~\ref{sec-equitable}.
In principle, approximations via weak equitable labellings should be the
cheapest to compute, and those via canonising should be
the most expensive. On the other hand, those via weak equitable labelling
should be the least accurate, and those via canonising the most accurate.
The reason that strong equitable labelling sometimes provides better approximations
than weak equitable labelling is that it considers all of the entries of the
stack simultaneously, whereas the weak version only considers each entry of
the stack individually.

\begin{example}\label{ex-approx}
  Let the labelled digraphs $\Gamma_{1},
  \Gamma_{2}, \Delta_{1}$, and $\Delta_{2}$
  be defined as in
  Figure~\ref{fig-ex-approx}.
  The label of every vertex in $\Gamma_{1},
  \Gamma_{2}, \Delta_{1}$, and $\Delta_{2}$ is \exLabel{white}, and each arc
  has the label \exLabel{solid} or
  \exLabel{dashed}, according to its depiction.
  Every vertex in $\Squash{[\Gamma_{1}, \Gamma_{2}]}$ and $\Squash{[\Delta_{1},
  \Delta_{2}]}$ has the same label $[\exLabel{white}, \exLabel{white}]$; arcs
  with label $[\exLabel{solid}, \#]$ are shown as \exLabel{solid}, arcs with
  label $[\#, \exLabel{dashed}]$ are shown as \exLabel{dashed}, and arcs with
  label $[\exLabel{solid}, \exLabel{dashed}]$ are shown as \exLabel{dotted}.
  We order labels via:
  \[
    \exLabel{dashed} < \exLabel{solid} < \exLabel{white} <
    [\exLabel{white}, \exLabel{white}] <
    [\#, \exLabel{dashed}] <
    [\exLabel{solid}, \#] <
    [\exLabel{solid}, \exLabel{dashed}].
  \]

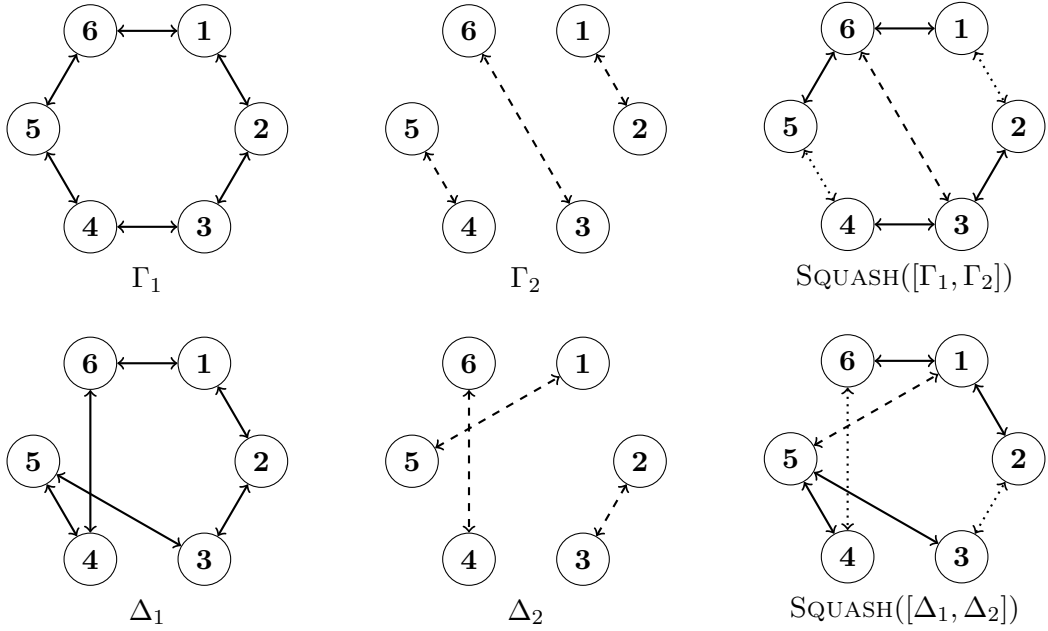
\begin{figure}[!ht]
  \centering
  \begin{tikzpicture}
    \tikzstyle{white}=[circle, draw=black]
    \foreach \x in {1,2,3,4,5,6} {
        \node[white] (\x) at (-\x*60+120:1.5cm) {$\mathbf{\x}$};};
    \foreach \x/\y in {1/2,2/3,3/4,4/5,5/6,6/1} {\arcSym{\x}{\y}};
    \node at (0, -2.0) {$\Gamma_{1}$};
  \end{tikzpicture}
  \qquad\quad
  \begin{tikzpicture}
    \tikzstyle{white}=[circle, draw=black]
    \foreach \x in {1,2,3,4,5,6} {
        \node[white] (\x) at (-\x*60+120:1.5cm) {$\mathbf{\x}$};};
    \foreach \x/\y in {1/2,3/6,4/5} {\arcSymDash{\x}{\y}};
    \node at (0, -2.0) {$\Gamma_{2}$};
  \end{tikzpicture}
  \qquad\quad
  \begin{tikzpicture}
    \tikzstyle{white}=[circle, draw=black]
    \foreach \x in {1,2,3,4,5,6} {
        \node[white] (\x) at (-\x*60+120:1.5cm) {$\mathbf{\x}$};};
    \foreach \x/\y in {2/3,3/4,5/6,6/1} {\arcSym{\x}{\y}};
    \foreach \x/\y in {3/6} {\arcSymDash{\x}{\y}};
    \foreach \x/\y in {1/2,4/5} {\arcSymDot{\x}{\y}};
    \node at (0, -2.0) {$\Squash{[\Gamma_{1}, \Gamma_{2}]}$};
  \end{tikzpicture}

  \vspace{4mm}

  \begin{tikzpicture}
    \tikzstyle{white}=[circle, draw=black]
    \foreach \x in {1,2,3,4,5,6} {
        \node[white] (\x) at (-\x*60+120:1.5cm) {$\mathbf{\x}$};};
    \foreach \x/\y in {6/4,4/5,5/3,3/2,2/1,1/6} {\arcSym{\x}{\y}};
    \node at (0, -2.0) {$\Delta_{1}$};
  \end{tikzpicture}
  \qquad\quad
  \begin{tikzpicture}
    \tikzstyle{white}=[circle, draw=black]
    \foreach \x in {1,2,3,4,5,6} {
        \node[white] (\x) at (-\x*60+120:1.5cm) {$\mathbf{\x}$};};
    \foreach \x/\y in {6/4,5/1,3/2} {\arcSymDash{\x}{\y}};
    \node at (0, -2.0) {$\Delta_{2}$};
  \end{tikzpicture}
  \qquad\quad
  \begin{tikzpicture}
    \tikzstyle{white}=[circle, draw=black]
    \foreach \x in {1,2,3,4,5,6} {
        \node[white] (\x) at (-\x*60+120:1.5cm) {$\mathbf{\x}$};};
    \foreach \x/\y in {4/5,5/3,2/1,1/6} {\arcSym{\x}{\y}};
    \foreach \x/\y in {5/1} {\arcSymDash{\x}{\y}};
    \foreach \x/\y in {6/4,3/2} {\arcSymDot{\x}{\y}};
    \node at (0, -2.0) {$\Squash{[\Delta_{1}, \Delta_{2}]}$};
  \end{tikzpicture}
  \caption[IsoApprox]{\label{fig-ex-approx}
  Pictures of the labelled digraphs on $\{1,\ldots,6\}$ from
  Example~\ref{ex-approx}.
  }
\end{figure}

  \begin{description}[leftmargin=0mm]
    \item[Weak equitable labelling:]

      Since $\Gamma_{1}$, $\Gamma_{2}$, $\Delta_{1}$, and $\Delta_{2}$ are
      regular (i.e.\ in each of them, all vertices have a common
      number of in-neighbours and a common number of out-neighbours),
      the equitable vertex labelling algorithm cannot
      make progress, as it only considers each labelled digraph individually.

      More specifically, Algorithm~\ref{alg-equitable} gives
      the label $[\exLabel{white}, \exLabel{white}, [\exLabel{solid}], [(2,2)]]$
      to every vertex in $\Gamma_{1}$ and $\Delta_{1}$
      (encoding that every \exLabel{white} vertex
      has two \exLabel{white} in-neighbours and two \exLabel{white}
      out-neighbours via \exLabel{solid} arcs), and it labels every vertex in
      $\Gamma_{2}$ and $\Delta_{2}$ with $[\exLabel{white}, \exLabel{white},
      [\exLabel{dashed}], [(1,1)]]$ (since every \exLabel{white} vertex has
      one \exLabel{white} out-neighbour and one \exLabel{white} in-neighbour via
      \exLabel{dashed} arcs).

      Therefore, weak equitable labelling gives the worst possible
      overestimation
      \[
        \Approx{[\Gamma_{1}, \Gamma_{2}]}{[\Delta_{1}, \Delta_{2}]} =
        \Sn{6}.
      \]

    \item[Strong equitable labelling:]

      Algorithm~\ref{alg-equitable} assigns
      the new label
      \[
        [[\exLabel{white}, \exLabel{white}],\
         [\exLabel{white}, \exLabel{white}],\
         [[\#, \exLabel{dashed}],\
         [\exLabel{solid}, \#]],\
         [(1,1), (2,2)]],
      \]
      to the vertices $3$ and $6$
      of $\Squash{[\Gamma_{1}, \Gamma_{2}]}$ and
      the vertices $1$ and $5$ of $\Squash{[\Delta_{1}, \Delta_{2}]}$.
      This encodes that these vertices (which previously had label
      $[\exLabel{white}, \exLabel{white}]$) each have one in- and one
      out-neighbour with label $[\exLabel{white}, \exLabel{white}]$ via $[\#,
      \exLabel{dashed}]$ arcs, and two such in-neighbours and two such
      out-neighbours via $[\exLabel{solid}, \#]$ arcs.
      In addition, the algorithm then labels the remaining
      vertices, namely $1, 2, 4$, and $5$ in $\Squash{[\Gamma_{1}, \Gamma_{2}]}$,
      and $2, 3, 4$, and $6$ in $\Squash{[\Delta_{1}, \Delta_{2}]}$, as
      \[
        [[\exLabel{white}, \exLabel{white}],\
         [\exLabel{white}, \exLabel{white}],\
         [[\exLabel{solid}, \#], [\exLabel{solid}, \exLabel{dashed}]],\
         [(1, 1), (1, 1)]].
      \]
      For each squashed labelled digraph, the algorithm updates these new labels
      with information about why these sets of vertices cannot be further
      subdivided.  Ultimately, strong equitable labelling gives
      \[
        \Approx{[\Gamma_{1}, \Gamma_{2}]}{[\Delta_{1}, \Delta_{2}]} =
        \< (3\,6), (1\,2), (1\,2\,4\,5) \> \cdot (1\,2\,3)(5\,6).
      \]
      Note that
      $|\Approx{[\Gamma_{1}, \Gamma_{2}]}{[\Delta_{1}, \Delta_{2}]}| =
      4! \cdot 2! = 48$, and so this is a much smaller overestimate.

      The coset representative $g \coloneqq (1\,2\,3)(5\,6)$ was chosen
      arbitrarily from $\Sn{6}$, subject to satisfying the
      property that ${\{1,2,4,5\}}^{g} = \{2,3,4,6\}$ and ${\{3,6\}}^{g} =
      \{1,5\}$.  Note that $g$ happens \emph{not} to be an
      isomorphism from $[\Gamma_{1}, \Gamma_{2}]$ to $[\Delta_{1}, \Delta_{2}]$.

    \item[Canonising and computing exactly:]

      We compute
      (using \textsc{Bliss}~\cite{bliss} via
      the \textsc{GAP}~\cite{GAP4} package \textsc{Digraphs}~\cite{digraphs})
      that
      $\Auto{\Squash{[\Gamma_{1}, \Gamma_{2}]}}
      =
      \< (1\,2)(3\,6)(4\,5), (1\,4)(2\,5)(3\,6) \>$
      and
      $(1\,2\,3\,5\,6)$ induces an isomorphism from
      $\Squash{[\Gamma_{1}, \Gamma_{2}]}$ to
      $\Squash{[\Delta_{1}, \Delta_{2}]}$.
      Thus
      \[
        \Iso{[\Gamma_{1}, \Gamma_{2}]}{[\Delta_{1}, \Delta_{2}]} =
        \< (1\,2)(3\,6)(4\,5), (1\,4)(2\,5)(3\,6) \> \cdot (1\,2\,3\,5\,6).
      \]
      In particular,
      $|\Iso{[\Gamma_{1}, \Gamma_{2}]}{[\Delta_{1}, \Delta_{2}]}| = 4$, which
      shows us how far away we still were from a perfect estimate with the
      other approximators.
  \end{description}

\end{example}

\section{Adding information to stacks with refiners}\label{sec-refiners}

In this section we introduce and discuss refiners for labelled digraph stacks.
We use refiners to encode information about a search problem into the stacks
around which the search is organised, in order to prune the search space.

\begin{definition}\label{defn-refiner}
  A \emph{refiner} for a set of permutations $U \subseteq \Sym{\Omega}$ is a
  pair of functions $(f_{L}, f_{R})$ from $\Stacks{\Omega}$ to itself, such
  that, for all $\stackS, \stackT \in \Stacks{\Omega}$ with
  $\stackS \cong \stackT$:
  \[
    U \cap \Iso{\stackS}{\stackT}
    \subseteq
    U \cap \Iso{f_{L}(\stackS)}{f_{R}(\stackT)}.
  \]
\end{definition}

While refiners depend on a subset of $\Sym{\Omega}$, we do not include this in
our notation in order to make it less complicated.
Note that the condition in Definition~\ref{defn-refiner} is satisfied for all non-isomorphic labelled
digraph stacks $\stackS$ and $\stackT$, and so the condition that
$\stackS \cong \stackT$ in Definition~\ref{defn-refiner} could
be removed without altering the notion of a refiner.

As a trivial example, every pair of functions from $\Stacks{\Omega}$ to itself
is a refiner for the empty set.  It is valid, and indeed common, to search for
the empty set:  for instance, one might wish to use the techniques in this paper
to search for the set of isomorphisms from one labelled digraph to another that, in the
end, prove to be non-isomorphic.  Thus it is important that
Definition~\ref{defn-refiner} accommodates the empty set.

The functions $f_{L}$ and $f_{R}$ of a refiner $(f_{L}, f_{R})$ are also
permitted to produce empty labelled digraph stacks.
If, for example, we set $f$ to be the constant function that maps every labelled
digraph stack on $\Omega$ to $\EmptyStack{\Omega}$, then $(f, f)$ is a refiner
for any set $U \subseteq \Sym{\Omega}$. This is because every permutation in
$\Sym{\Omega}$, by definition, induces an automorphism of $\EmptyStack{\Omega}$.
It follows that $U \cap
  \Iso{f(\stackS)}{f(\stackT)} = U \cap \Sym{\Omega} = U$ for all
$\stackS, \stackT \in \Stacks{\Omega}$ in this case.

In the following lemma, we formulate additional equivalent definitions of
refiners.

\begin{lemma}\label{lem-refiner-equiv-definitions}
  Let $(f_{L}, f_{R})$ be a pair of functions from $\Stacks{\Omega}$ to itself
  and let $U \subseteq \Sym{\Omega}$.  Then the following are equivalent:
  \begin{enumerate}[label=\textrm{(\roman*)}]
    \item\label{item-refiner-initial}
          $(f_{L}, f_{R})$ is a refiner for $U$.

    \item\label{item-refiner-long}
          For all isomorphic $\stackS, \stackT \in \Stacks{\Omega}$:
          \[U \cap \Iso{\stackS}{\stackT}
            =
            U \cap \Iso{\stackS \Vert f_{L}(\stackS)}
            {\stackT \Vert f_{R}(\stackT)}.\]

    \item\label{item-refiner-individual-perm}
          For all isomorphic $\stackS, \stackT \in \Stacks{\Omega}$ and
          \(g \in U\):
          \[\text{if } {\stackS}^{g} = \stackT,
            \text{then } {f_{L}(\stackS)}^{g} = f_{R}(\stackT).\]
  \end{enumerate}
\end{lemma}

\begin{proof}
  $\ref{item-refiner-initial}\Rightarrow\ref{item-refiner-long}$.
  Let $\stackS, \stackT \in \Stacks{\Omega}$, and suppose that
  $\stackS$ and $\stackT$ are isomorphic.
  Then
  \(
  U \cap \Iso{\stackS}{\stackT}
  \subseteq
  U \cap \Iso{f_{L}(\stackS)}{f_{R}(\stackT)}
  \)
  by assumption, and since $\stackS$ and $\stackT$ have equal lengths,
  it follows that
  \[
    \Iso{\stackS}{\stackT} \cap
    \Iso{f_{L}(\stackS)}{f_{R}(\stackT)}
    =
    \Iso{\stackS \Vert f_{L}(\stackS)}
    {\stackT \Vert f_{R}(\stackT)}
  \]
  by Remark~\ref{rmk-stack-iso-auto}.
  Hence
  \begin{align*}
    U \cap \Iso{\stackS}{\stackT}
     & =
    U \cap \Iso{\stackS}{\stackT}
    \cap \big(U \cap \Iso{f_{L}(\stackS)}{f_{R}(\stackT)}\big) \\
     & =
    U \cap \big(\Iso{\stackS}{\stackT}
    \cap \Iso{f_{L}(\stackS)}{f_{R}(\stackT)}\big)             \\
     & =
    U \cap \Iso{\stackS \Vert f_{L}(\stackS)}
    {\stackT \Vert f_{R}(\stackT)}.
  \end{align*}

  $\ref{item-refiner-long}\Rightarrow\ref{item-refiner-individual-perm}$.
  Let $\stackS, \stackT \in \Stacks{\Omega}$ be isomorphic,
  and let $u \in U$.  If
  $\stackS^{u} = \stackT$, then $u \in
    \Iso{\stackS}{\stackT}$ by definition, and so $u \in
    \Iso{\stackS \Vert f_{L}(\stackS)}{\stackT \Vert
      f_{R}(\stackT)}$ by assumption.
  Since $\stackS$ and $\stackT$ have equal lengths, and $\stackS
    \Vert f_{L}(\stackS)$ and $\stackT \Vert f_{R}(\stackT)$ have
  equal lengths, it follows that so too do $f_{L}(\stackS)$ and
  $f_{R}(\stackT)$.
  Then ${f_{L}(\stackS)}^{u} = f_{R}(\stackT)$,
  since for each
  $i \in \{1, \ldots, |f_{L}(\stackS)|\}$,
  \[
    {f_{L}(\stackS)[i]}^{u}
      =
      {{\left(\stackS \Vert f_{L}(\stackS)\right)}[|\stackS|+i]}^{u}
    =
    {\left(\stackT \Vert f_{R}(\stackT)\right)}[|\stackT|+i]
      =
      {f_{R}(\stackT)}[i].
  \]

  $\ref{item-refiner-individual-perm}\Rightarrow\ref{item-refiner-initial}$.
  This implication is immediate.
\end{proof}

Perhaps
Lemma~\ref{lem-refiner-equiv-definitions}\ref{item-refiner-long} most clearly
indicates the relevance of refiners to search.

Suppose that we wish to search for the intersection $U_{1} \cap \cdots \cap
U_{n}$ of some subsets of $\Sym{\Omega}$.  Let $i \in \{1,\ldots,n\}$, let $(f_{L}, f_{R})$ be a refiner for
$U_{i}$, and let $\stackS$ and $\stackT$ be isomorphic labelled digraph
stacks on $\Omega$, such that $\Iso{\stackS}{\stackT}$ overestimates (i.e.\
contains) $U_{1} \cap \cdots \cap U_{n}$.

We may use the refiner $(f_{L}, f_{R})$ to \emph{refine} the pair of stacks
$(\stackS, \stackT)$: we apply the functions $f_{L}$ and $f_{R}$, respectively,
to the stacks $\stackS$ and $\stackT$ and obtain an extended pair of stacks
$({\stackS \Vert f_{L}(\stackS)}, {\stackT \Vert f_{R}(\stackT)})$.  We call
this process \emph{refinement}.  Note that the refiner for $U_{i}$ need not
consider the other sets in the intersection.

By Lemma~\ref{lem-refiner-equiv-definitions}\ref{item-refiner-long}, the set of
induced isomorphisms $\Iso{\stackS \Vert f_{L}(\stackS)}{\stackT \Vert
f_{R}(\stackT)}$ contains the elements of $U_{i}$ that belonged to
$\Iso{\stackS}{\stackT}$. Since $U_{i}$ contains $U_{1} \cap
\cdots \cap U_{n}$, it follows that $\Iso{\stackS \Vert f_{L}(\stackS)}{\stackT
\Vert f_{R}(\stackT)}$ is again an overestimate for $U_{1} \cap \cdots \cap
U_{n}$; it is contained in the previous overestimate
by Remark~\ref{rmk-stack-iso-auto}.
Moreover, $\Iso{\stackS \Vert f_{L}(\stackS)}{\stackT \Vert f_{R}(\stackT)}$ may
lack some elements of $\Iso{\stackS}{\stackT} \setminus (U_{1} \cap \cdots \cap
U_{n})$, in which case we have produced a smaller overestimate for the result,
and thereby reduced the size of the remaining search space.
We may then repeat this process, perhaps with a different refiner,
in the hope of reducing the search space further still.

The condition in
Lemma~\ref{lem-refiner-equiv-definitions}\ref{item-refiner-individual-perm} is
often most convenient for verifying that a pair of functions is a refiner for
some set, as is done in Example~\ref{ex-perfect-set}.

\begin{example}[Refiner for set stabiliser and transporter in
    $\Sym{\Omega}$]\label{ex-perfect-set}
  Let $A, B \subseteq \Omega$ and let
  \[S_{A, B} = \set{g \in \Sym{\Omega}}{A^{g} = B}\]
  denote the set of permutations of $\Omega$ that map $A$ to $B$.  Note that
  $S_{A, A}$ is the set stabiliser $\Sym{\Omega}_{A}$ of $A$ in
  $\Sym{\Omega}$, and that in general, either $S_{A, B}$ is empty, or
  it is a right coset of $\Sym{\Omega}_{A}$ and a left coset of
  $\Sym{\Omega}_{B}$ in $\Sym{\Omega}$.

  Define a labelled digraph $\Gamma_{A}$ without arcs, where the vertices in $A$
  have the label $\exLabel{in}$, and the remaining vertices have the label
  $\exLabel{out}$.  Furthermore, let $\textsc{Stab}_{A}$ be the function that
  maps every labelled digraph stack on $\Omega$ to the stack $[\Gamma_{A}]$.
  Define $\Gamma_{B}$ and $\textsc{Stab}_{B}$ analogously.
  Then $(\textsc{Stab}_{A}, \textsc{Stab}_{B})$ is a refiner for the set $S_{A,
        B}$ by
  Lemma~\ref{lem-refiner-equiv-definitions}\ref{item-refiner-individual-perm},
  since
  \[
    {\textsc{Stab}_{A}(\stackS)}^{g}
    = {[\Gamma_{A}    ]}^{g}
    =  [\Gamma_{A}^{g}]
      =  [\Gamma_{B}    ] =
    \textsc{Stab}_{B}(\stackT)
  \]
  for all $\stackS, \stackT \in \Stacks{\Omega}$ and for all $g \in
    S_{A, B}$.
\end{example}

The refiner in Example~\ref{ex-perfect-set} is particularly straightforward: the
functions $\textsc{Stab}_{A}$ and $\textsc{Stab}_{B}$ are constant, and they
return stacks of length one containing labelled digraphs without arcs and only
two different vertex labels. Moreover, the isomorphisms between these stacks are
precisely the permutations in $\Sym{\Omega}$ that map $A$ to $B$ as sets.

Note that when Example~\ref{ex-perfect-set} gives a refiner $(f_{L}, f_{R})$ for a subgroup of
$\Sym{\Omega}$ rather than just a subset, for example when $A = B$ and the
subgroup is the setwise stabiliser of $A$ in $\Sym{\Omega}$, then $f_{L} = f_{R}$.
Lemma~\ref{lem-group-refiner-symmetric}
shows that this property is shared by every refiner for a set that contains the
identity map on $\Omega$.

\begin{lemma}[\mbox{cf.~\cite[Prop 2]{leon1997},~\cite[Lemma 6]{leon1991}}]
  \label{lem-group-refiner-symmetric}
  Let $(f_{L}, f_{R})$ be a refiner for a subset $U \subseteq \Sym{\Omega}$ that
  contains the identity map, $\idOmega$.  Then $f_{L} = f_{R}$.
\end{lemma}

\begin{proof}
  Let $\stackS \in \Stacks{\Omega}$ be arbitrary.  Since $\idOmega
  \in U$ and $(f_{L}, f_{R})$ is a refiner for $U$, it follows by
  Lemma~\ref{lem-refiner-equiv-definitions}\ref{item-refiner-individual-perm}
  that $f_{L}(\stackS) = {f_{L}(\stackS)}^{\idOmega} =
  f_{R}(\stackS)$.
\end{proof}

Lemma~\ref{lem-refiner-equiv-definitions}\ref{item-refiner-individual-perm}
implies the following lemma.

\begin{lemma}\label{lem-simple-refiner}
  Let $f$ be a function from $\Stacks{\Omega}$ to itself, and let $U$ be a
  subset of $\Sym{\Omega}$ containing $\idOmega$.
  Then $(f, f)$ is a refiner for $U$ if and only if
  \(f({\stackS}^{g}) = {f(\stackS)}^{g}\)
  for all $g \in U$ and $\stackS \in \Stacks{\Omega}$.
\end{lemma}

Next, we see that any refiner for a non-empty set can be
derived from a function $f$ that satisfies the condition in
Lemma~\ref{lem-simple-refiner}.

\begin{lemma}\label{lem-refiner-right-coset}
  Let $U$ be a non-empty subset of $\Sym{\Omega}$, fix $h \in U$
  arbitrarily, and let $f$ and $g$ be functions from $\Stacks{\Omega}$
  to itself.
  Then the following are equivalent:
  \begin{enumerate}[label=\textrm{(\roman*)}]
    \item\label{item-refiner-right-coset-fg}
      $(f, g)$ is a refiner for $U$.
    \item\label{item-refiner-right-coset-ff-Uh-1}
      $(f, f)$ is a refiner for $U h^{-1}$,
  and
  $g(\stackS) = f(\stackS^{h^{-1}}){}^{h}$ for all $\stackS \in \Stacks{\Omega}$.
  \end{enumerate}
  In particular,
  if $U$ is a right coset of a subgroup $G \leq \Sym{\Omega}$, then
  $(f, g)$ is a refiner for the coset $U = G h$
  if and only if
  $(f, f)$ is a refiner for the group $G$, and
  $g(\stackS) = f(\stackS^{h^{-1}}){}^{h}$ for all $\stackS \in \Stacks{\Omega}$.
\end{lemma}

\begin{proof}
  $\ref{item-refiner-right-coset-fg}
   \Rightarrow\ref{item-refiner-right-coset-ff-Uh-1}$.
  Let $\stackS \in \Stacks{\Omega}$.
  Since $(f, g)$ is a refiner for $U$, it follows by
  Lemma~\ref{lem-refiner-equiv-definitions}\ref{item-refiner-individual-perm}
  that
  $g(\stackS^{h}) = f(\stackS){}^{h}$.
  Moreover, $\stackS$ was chosen arbitrarily, and so
  $g(\stackS) = g((\stackS^{h^{-1}}){}^{h})
  = f(\stackS^{h^{-1}}){}^{h}$.
  Furthermore, if $y \in U h^{-1}$ is arbitrary, then
  \[
    {f(\stackS)}^{y} =
    ({f(\stackS)}^{y h}){}^{h^{-1}} =
      g(\stackS^{y h}){}^{h^{-1}} =
      f(\stackS^{y h h^{-1}}){}^{h h^{-1}} =
    f(\stackS^{y}).
  \]
  Therefore $(f, f)$ is a refiner for $U h^{-1}$ by
  Lemma~\ref{lem-refiner-equiv-definitions}\ref{item-refiner-individual-perm}.

  $\ref{item-refiner-right-coset-ff-Uh-1}
   \Rightarrow\ref{item-refiner-right-coset-fg}$.
  Let $\stackS, \stackT \in \Stacks{\Omega}$, suppose that $\stackS$ and
  $\stackT$ are isomorphic, and let $x \in U$.
  Since $(f, f)$ is a refiner for $U h^{-1}$ and $x h^{-1} \in U h^{-1}$, it
  follows by Lemma~\ref{lem-simple-refiner} that $f{(\stackS)}^{x h^{-1}} =
    f(\stackS^{x h^{-1}})$. Thus, if $\stackS^{x} = \stackT$, then
  \[
    {f(\stackS)}^{x}
    =
    {f(\stackS)}^{x h^{-1} h}
    =
    {f(\stackS^{x h^{-1}})}^{h}
    =
    f(\stackT^{h^{-1}}){}^{h}
    =
    g(\stackT),
  \]
  and so $(f, g)$ is a refiner for $U$ by
  Lemma~\ref{lem-refiner-equiv-definitions}\ref{item-refiner-individual-perm}.
\end{proof}

For some pairs of functions, such as those in Example~\ref{ex-perfect-set}
and the upcoming Example~\ref{ex-set-of-sets}, one may use the
following results to show that the pair gives a refiner.

\begin{lemma}\label{lem-refiner-nice-condition}
  Let $U \subseteq \Sym{\Omega}$, and let $f_{L}$, $f_{R}$ be
  functions from $\Stacks{\Omega}$ to itself such that
  \(
    U \subseteq \Iso{f_{L}(\stackS)}{f_{R}(\stackT)}
  \)
  for all isomorphic $\stackS, \stackT \in \Stacks{\Omega}$.
  Then $(f_{L}, f_{R})$ is a refiner for $U$.
\end{lemma}

\begin{proof} Let $\stackS, \stackT \in \Stacks{\Omega}$ be isomorphic.
  By Remark~\ref{rmk-stack-iso-auto} and the assumption on $(f_{L}, f_{R})$ we
  have that
  \begin{align*} U \cap \Iso{\stackS}{\stackT} & = \left( U \cap
    \Iso{f_{L}(\stackS)}{f_{R}(\stackT)} \right) \cap \Iso{\stackS}{\stackT} \\
    & = U \cap \left( \Iso{\stackS}{\stackT} \cap
    \Iso{f_{L}(\stackS)}{f_{R}(\stackT)} \right) \\ & = U \cap \Iso{\stackS
    \Vert f_{L}(\stackS)} {\stackT \Vert f_{R}(\stackT)}.  \end{align*}
  Therefore, by
  Lemma~\ref{lem-refiner-equiv-definitions}\ref{item-refiner-long}, $(f_{L},
  f_{R})$ is a refiner for $U$.  \end{proof}

\begin{cor}\label{cor-refiner-nice-condition-subgroup}
  Let $G \leq \Sym{\Omega}$, and let $f$ be a function from $\Stacks{\Omega}$ to
  itself with constant value $\stackS \in \Stacks{\Omega}$, such that $G \leq \Auto{\stackS}$.  Then
  $(f, f)$ is a refiner for $G$.
\end{cor}

\begin{example}[Refiner for set of subsets stabiliser and
  transporter]\label{ex-set-of-sets}

  Let $k \in \N_{0}$ and $U_{i} \subseteq \Omega$ for all $i \in \{1, \ldots
  k\}$, let $\mathcal{U} = \{ U_{1}, \ldots, U_{k} \}$, and let
  $\Gamma_{\mathcal{U}}$ be the labelled digraph
  on $\Omega$ whose set of arcs is
  \[
  \set{(\alpha, \beta) \in \Omega \times \Omega}{\alpha \neq \beta\ \text{and}\
  \{\alpha, \beta\} \subseteq U_{i}\ \text{for some}\ i};
  \]
  where the label of each vertex $\alpha \in \Omega$ is
  a list of length $\max\set{|U_{i}|}{i \in \{1, \ldots, k\}}$,
  with $i\textsuperscript{th}$ entry
  \[(|\set{j \in \{1, \ldots, k\}}{\alpha \in U_{j}\ \text{and}\ |U_{j}| = i}|,
  \,k),\]
  and the label of each arc $(\alpha, \beta)$ in the digraph
  is a list of the same length, with $i\textsuperscript{th}$ entry
  \[
  (|\set{j \in \{1,\ldots,k\}}{\alpha,\beta \in U_{j}\ \text{and}\ |U_{j}| = i}|,
   \,k).
  \]
  The label of a vertex (or arc) encodes, for each size of
  subset, the number of all subsets that have that size and contain
  that vertex (or arc).
  For every $\stackS \in \Stacks{\Omega}$, we define $f_{\mathcal{U}}(\stackS) =
  [\Gamma_{\mathcal{U}}]$.
  In addition,
  for all $g \in \Sym{\Omega}$,
  we define $\mathcal{U}^{g} = \{U_{1}^{g}, \ldots, U_{k}^{g}\}$.

  Let $\mathcal{U}$ and $\mathcal{V}$ be arbitrary sets of subsets of $\Omega$.
  Since the labelled digraphs $\Gamma_{\mathcal{U}}$ and $\Gamma_{\mathcal{V}}$
  were defined so that
  $\set{g \in \Sym{\Omega}}{\mathcal{U}^{g} = \mathcal{V}} \subseteq
  \Iso{\Gamma_{\mathcal{U}}}{\Gamma_{\mathcal{V}}}$, it follows by
  Lemma~\ref{lem-refiner-nice-condition} that
  $(f_{\mathcal{U}}, f_{\mathcal{V}})$ is a refiner for the set $\set{g
  \in \Sym{\Omega}}{\mathcal{U}^{g} = \mathcal{V}}$,
  and Corollary~\ref{cor-refiner-nice-condition-subgroup} yields that
  $(f_{\mathcal{U}}, f_{\mathcal{U}})$ is a refiner for the group
  $\set{g \in \Sym{\Omega}}{\mathcal{U}^{g} = \mathcal{U}}$.

  For a specific example, we
  consider the sets of subsets
  $\mathcal{U} \coloneqq \{\{1\}, \{1,2,3\}, \{2,4\}\}$,
  and $\mathcal{V} \coloneqq \{\{5\}, \{2,3,4\}, \{3,4\}\}$.
  Both $\mathcal{U}$ and $\mathcal{V}$ contain three subsets, which have sizes
  $1$, $2$ and $3$, and so, at least superficially, it seems plausible there may
  exist elements of $\Sn{5}$ that map $\mathcal{U}$ to $\mathcal{V}$.  In order
  to search for the transporter set $\set{g \in \Sn{5}}{\mathcal{U}^{g} =
  \mathcal{V}}$, then (with all the following notation as defined above) we can
  use the refiner $(f_{\mathcal{U}}, f_{\mathcal{V}})$ to produce labelled
  digraphs $\Gamma_{\mathcal{U}}$ and $\Gamma_{\mathcal{V}}$, such that
  $\Iso{\Gamma_{\mathcal{U}}}{\Gamma_{\mathcal{V}}}$ contains the transporter
  set. These labelled digraphs are depicted in Figure \ref{fig-set-of-sets};
  although we do not give the correspondence explicitly, a pair of vertices or a
  pair of arcs have the same visual style if and only if they have the same
  label.

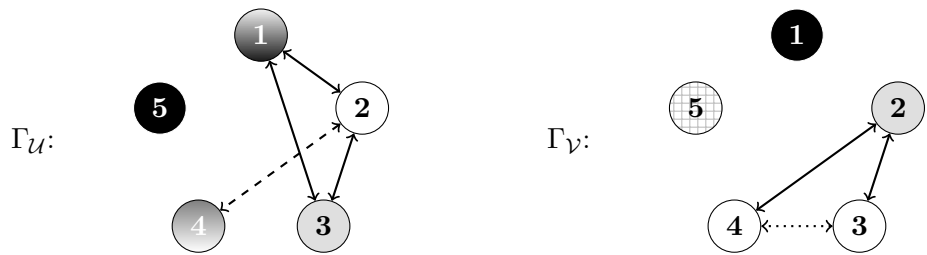
\begin{figure}[!ht]
  \begin{center}
    \begin{tikzpicture}
      \tikzstyle{white}=[circle, draw=black]
      \tikzstyle{grey}=[circle, draw=black, fill=gray!25]
      \tikzstyle{black}=[circle, draw=white, fill=black!100]
      \tikzstyle{dots}=[circle, draw=black, pattern=dots, pattern color=gray!50]

      \node[style={
            circle,
            draw=black,
            shade,
            bottom color=black!90, top color=gray!10,
            align=center}]
      (1) at (90:1.4cm)   {$\color{white}\mathbf{1}$};
      \node[white]  (2) at (18:1.4cm)   {$\mathbf{2}$};
      \node[grey]   (3) at (-54:1.4cm)  {$\mathbf{3}$};
      \node[style={
            circle,
            draw=black,
            shade,
            top color=gray!100, bottom color=white, align=center}]
      (4) at (-126:1.4cm) {$\color{white}\mathbf{4}$};
      \node[black]  (5) at (162:1.4cm)  {$\color{white}\mathbf{5}$};

      \arcSym{1}{2}
      \arcSym{2}{3}
      \arcSym{1}{3}
      \arcSymDash{2}{4}

      \node at (3, 0) {};
      \node at (-3, 0) {$\Gamma_{\mathcal{U}}$:};
    \end{tikzpicture}
    \quad
    \begin{tikzpicture}
      \tikzstyle{white}=[circle, draw=black]
      \tikzstyle{grey}=[circle, draw=black, fill=gray!25]
      \tikzstyle{black}=[circle, draw=white, fill=black!100]
      \tikzstyle{str-ne}=[circle, draw=black, pattern=grid,
                          pattern color=gray!50]
      \tikzstyle{white}=[circle, draw=black]

      \node[black]  (1) at (90:1.4cm)   {$\color{white}\mathbf{1}$};
      \node[grey]   (2) at (18:1.4cm)   {$\mathbf{2}$};
      \node[white]  (3) at (-54:1.4cm)  {$\mathbf{3}$};
      \node[white]  (4) at (-126:1.4cm) {$\mathbf{4}$};
      \node[str-ne] (5) at (162:1.4cm)  {$\mathbf{5}$};

      \arcSym{2}{4}
      \arcSym{2}{3}
      \arcSymDot{3}{4}

      \node at (3, 0) {};
      \node at (-3, 0) {$\Gamma_{\mathcal{V}}$:};
    \end{tikzpicture}
  \end{center}
  \caption[set-of-sets]{
    Demonstration of the labelled digraphs
    $\Gamma_{\mathcal{U}}$ and $\Gamma_{\mathcal{V}}$
    from
    Example~\ref{ex-set-of-sets}, for the sets of subsets
    $\mathcal{U} \coloneqq \{\{1\}, \{1,2,3\}, \{2,4\}\}$
    and $\mathcal{V} \coloneqq \{\{5\}, \{2,3,4\}, \{3,4\}\}$
    of $\{1,\ldots,5\}$.
  }\label{fig-set-of-sets}
\end{figure}

  There are many ways to show that $\Gamma_{\mathcal{U}}$ and
  $\Gamma_{\mathcal{V}}$ are non-isomorphic: for example, they have
  different numbers of arcs.  Hence no permutation in $\Sn{5}$
  maps $\mathcal{U}$ to $\mathcal{V}$.
\end{example}

\subsection{Perfect refiners}\label{sec-perfect-refiners}

Refiners differ in their ability to encode information into a pair
of labelled digraph stacks.  For some sets, there are refiners that
capture all of the information about the set. Such refiners are the
focus of this section.

\begin{lemma}\label{lem-perfect-is-refiner}
  Let $U \subseteq \Sym{\Omega}$, and let $f_{L}, f_{R}$ be
  functions from $\Stacks{\Omega}$ to itself such that
  \[
    U \cap \Iso{\stackS}{\stackT}
    =
    \Iso{\stackS \Vert f_{L}(\stackS)}
        {\stackT \Vert f_{R}(\stackT)}
  \]
  for all $\stackS, \stackT \in \Stacks{\Omega}$.
  Then $(f_{L}, f_{R})$ is a refiner for $U$.
\end{lemma}

\begin{proof}
  Let $\stackS, \stackT \in \Stacks{\Omega}$ be isomorphic.

  The hypothesis implies that
  $\Iso{\stackS \Vert f_{L}(\stackS)}{\stackT \Vert f_{R}(\stackT)} \subseteq
  U$, and hence that
  \[
    U \cap \Iso{\stackS}{\stackT} =
    U \cap \Iso{\stackS \Vert f_{L}(\stackS)}{\stackT \Vert f_{R}(\stackT)}.
  \]
  Thus $(f_{L}, f_{R})$ is a refiner for $U$ by
  Lemma~\ref{lem-refiner-equiv-definitions}\ref{item-refiner-long}.
\end{proof}

Refiners with the property from Lemma~\ref{lem-perfect-is-refiner} are called
\emph{perfect refiners}.
Roughly speaking, a perfect refiner $(f_{L}, f_{R})$ for a subset $U \subseteq
\Sym{\Omega}$ is used during a search algorithm to take a pair of isomorphic
labelled digraph stacks $\stackS$ and $\stackT$, and refine the stacks in
such a way as to leave exactly those isomorphisms from $\stackS$ to $\stackT$
that are contained in $U$. In particular, a perfect refiner never needs to be
applied more than once in any branch of a search, because all information
about $U$ is already encoded into the stacks after its first
application.

Next we give alternative ways of proving that a pair of functions
forms a perfect refiner for a particular set.

\begin{lemma}\label{lem-perfect-nice-condition}
  Let $U \subseteq \Sym{\Omega}$, and let $f_{L}, f_{R}$ be
  functions from $\Stacks{\Omega}$ to itself such that
  \(
    U = \Iso{f_{L}(\stackS)}{f_{R}(\stackT)}
  \)
  for all $\stackS, \stackT \in \Stacks{\Omega}$.
  Then $(f_{L}, f_{R})$ is a perfect refiner for $U$.
\end{lemma}

\begin{proof}
  Let $\stackS, \stackT \in \Stacks{\Omega}$ be isomorphic.
  Using Remark~\ref{rmk-stack-iso-auto} and the assumption on $(f_{L}, f_{R})$,
  it follows that
  \[
    U \cap \Iso{\stackS}{\stackT}
    =
    \Iso{f_{L}(\stackS)}{f_{R}(\stackT)}
    \cap
    \Iso{\stackS}{\stackT}
    =
    \Iso{\stackS \Vert f_{L}(\stackS)}
        {\stackT \Vert f_{R}(\stackT)}.
        \qedhere
  \]
\end{proof}

\begin{cor}\label{cor-perfect-nice-condition-subgroup}
  Let $G \leq \Sym{\Omega}$, and let $f$ be a function from
  $\Stacks{\Omega}$ to itself with constant value $\stackS \in \Stacks{\Omega}$, such that $G =
  \Auto{\stackS}$.  Then $(f, f)$ is a perfect refiner for $G$.
\end{cor}

We have already seen a perfect refiner in Example~\ref{ex-perfect-set}; this is
particularly straightforward to verify with
Lemma~\ref{lem-perfect-nice-condition}.  We give
several further examples of perfect refiners in
Section~\ref{sec-examples-perfect}.
Not every subset of $\Sym{\Omega}$ has a perfect refiner, however.

\begin{lemma}\label{lem-when-perfect-refiners-exist}
  Let $U \subseteq \Sym{\Omega}$.
  Then there exists a perfect refiner for $U$ if and only if
  $U = \Iso{\stackS}{\stackT}$ for some $\stackS, \stackT \in \Stacks{\Omega}$.
\end{lemma}

\begin{proof}
  $(\Rightarrow)$
  Let $(f_{L}, f_{R})$ be a perfect refiner for $U$.
  The result follows by applying the condition in
  Lemma~\ref{lem-perfect-is-refiner} with $S = T = \EmptyStack{\Omega}$.

  $(\Leftarrow)$
  Let $\stackS, \stackT \in \Stacks{\Omega}$ be such that $U = \Iso{\stackS}{\stackT}$.
  Define $f_{L}$ and $f_{R}$ to be functions from $\Stacks{\Omega}$ to itself with
  constant values $\stackS$ and $\stackT$, respectively. Then
  $(f_{L}, f_{R})$ is a perfect refiner for $U$ by Lemma~\ref{lem-perfect-nice-condition}.
\end{proof}

Lemma~\ref{lem-when-perfect-refiners-exist} implies that a non-empty subset
has a perfect refiner if and only if it is a coset of the $\Sym{\Omega}$-induced
automorphism group of a labelled digraph. Note that not every subgroup
of $\Sym{\Omega}$ is the automorphism group of a labelled digraph.

\subsubsection{Examples of perfect refiners}\label{sec-examples-perfect}

In this section, we give examples of perfect refiners for subgroups and
their cosets, in order to thoroughly explain, especially in the
first example, the idea of a perfect refiner.
As we saw in Lemmas~\ref{lem-simple-refiner}
and~\ref{lem-refiner-right-coset},
the
crucial step when creating a refiner for a subgroup $G \leq \Sym{\Omega}$, or
for one of its cosets, is to define a function $f$ from $\Stacks{\Omega}$ to
itself such that $f(\stackS^{g}) = {f(\stackS)}^{g}$ for all $\stackS \in
\Stacks{\Omega}$ and $g \in G$.

\begin{example}[Perfect refiner for permutation centraliser and conjugacy]\label{ex-perfect-perm-centraliser}

  For every $g \in \Sym{\Omega}$, let $\Gamma_{g}$ be the labelled digraph on
  $\Omega$ whose set of arcs is
  \(
  \set{(\alpha, \beta) \in \Omega \times \Omega}
  {\alpha^{g} = \beta},
  \)
  and in which all labels are defined to be $0$.  For every $\stackS \in
  \Stacks{\Omega}$, define
  $f_{g}(\stackS) = [\Gamma_{g}]$.
  Let $g, h \in \Sym{\Omega}$ be arbitrary.  Then $(f_{g}, f_{g})$ is a perfect
  refiner for the centraliser of $g$ in $\Sym{\Omega}$ by
  Corollary~\ref{cor-perfect-nice-condition-subgroup}, and by
  Lemma~\ref{lem-perfect-nice-condition}, $(f_{g}, f_{h})$ is a perfect refiner
  for the set $\set{x \in \Sym{\Omega}}{g^{x} = h}$.

  We illustrate one such instance of this perfect refiner.
  Let $g = (1\,2)(3\,6\,5) \in \Sn{6}$,
  let $U$ denote the centraliser of $g$ in $\Sn{6}$, and define
  the labelled digraph $\Gamma_{g}$ and function $f_{g}$ as above.
  A diagram of $\Gamma_{g}$ is shown in Figure~\ref{fig-perm-centraliser}.
  Note that there is a loop at vertex $4$, and only at vertex
  $4$, because $4$ is the unique fixed point of $g$ on $\{1,\ldots,6\}$.

\begin{figure}[!ht]
  \begin{center}
    \begin{tikzpicture}
      \foreach \x in {1,2,3,4,5,6} {
        \node[circle, draw=black] (\x) at (-\x*60+120:1.5cm) {$\mathbf{\x}$};};

      \arcSym{1}{2}
      \arc{3}{6}
      \arc{6}{5}
      \arc{5}{3}
      \looparcL{4}

      \node at (3, 0) {};
      \node at (-3, 0) {$\Gamma_{g}$:};
    \end{tikzpicture}
  \end{center}

  \caption[Perfect refiner for permutation centraliser]{
    The labelled digraph $\Gamma_{g}$ for $g = (1\,2)(3\,6\,5)$,
    from Example~\ref{ex-perfect-perm-centraliser}.
  }\label{fig-perm-centraliser}
\end{figure}
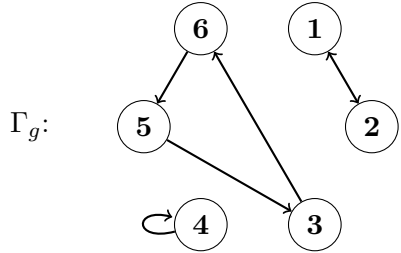

  In order to
  use Corollary~\ref{cor-perfect-nice-condition-subgroup}
  to verify that
  $(f_{g}, f_{g})$ is a perfect refiner for $U$,
  we must prove that
  $\Auto{[\Gamma_{g}]}=U$.
  Note that $\Auto{[\Gamma_{g}]} = \Auto{\Gamma_{g}}$.
  Every automorphism of $\Gamma_{g}$ fixes the unique
  vertex with a loop, and it also stabilises the connected components (because
  they have different sizes), and induces automorphisms on them.
  Therefore
  $\Auto{\Gamma_{g}}$ is contained in the subgroup $\langle (1\,2), (3\,5),
  (3\,6) \rangle$ of $\Sn{6}$.  But none of the transpositions in
  $\langle (3\,5), (3\,6) \rangle$ is an automorphism of $\Gamma_{g}$, because
  the arcs between $3$, $5$, and $6$ only go in one direction.  Hence
  $\Auto{\Gamma_{g}} = \< (1\,2), (3\,6\,5) \> = U$, as required.
\end{example}

\begin{example}[Perfect refiner for labelled digraph automorphism and
    isomorphism]\label{ex-perfect-digraph}

  For every labelled digraph $\Gamma$ on $\Omega$ and every $\stackS
    \in \Stacks{\Omega}$, let $f_{\Gamma}(\stackS) = [\Gamma]$.  Note that
  ${f_{\Gamma}(\stackS)}^{g} = [\Gamma^{g}] = [\Gamma] =
    f_{\Gamma}(\stackS)$ for all $g \in \Auto{\Gamma}$.

  Let $\Gamma$ and $\Delta$ be arbitrary labelled digraphs on $\Omega$.
  Then
  $(f_{\Gamma}, f_{\Delta})$ is a perfect refiner for $\Iso{\Gamma}{\Delta}$
  by Lemma~\ref{lem-perfect-nice-condition},
  and
  $(f_{\Gamma}, f_{\Gamma})$ is a perfect refiner for $\Auto{\Gamma}$
  by Corollary~\ref{cor-perfect-nice-condition-subgroup}.
\end{example}

\begin{example}[Perfect refiner for list of subsets stabiliser and
    transporter]\label{ex-perfect-list-of-sets}

  Whenever $k \in \N_{0}$ and $U_{i} \subseteq \Omega$ for each $i \in \{1,
  \ldots, k\}$ and $\mathcal{U} \coloneqq [U_{1}, \ldots, U_{k}]$, we let
  $\Gamma_{\mathcal{U}}$ be the labelled digraph on $\Omega$ without arcs,
  where the label of each vertex $\alpha \in \Omega$ is $\set{i \in \{1,
  \ldots, k\}}{\alpha \in U_{i}}$. For every
  $\stackS \in \Stacks{\Omega}$, define $f_{\mathcal{U}}(\stackS) =
  [\Gamma_{\mathcal{U}}]$.  If $g \in \Sym{\Omega}$, then
  $\mathcal{U}^{g} \coloneqq [U_{1}^{g}, \ldots, U_{k}^{g}]$.

  Let $\mathcal{U}$ and $\mathcal{V}$ be arbitrary lists of subsets of $\Omega$
  with all the notation as explained above.
  Then
  $(f_{\mathcal{U}}, f_{\mathcal{V}})$ is a perfect refiner for the set $\set{g
  \in \Sym{\Omega}}{\mathcal{U}^{g} = \mathcal{V}}$ by
  Lemma~\ref{lem-perfect-nice-condition},
  and
  $(f_{\mathcal{U}}, f_{\mathcal{U}})$ is a perfect refiner for the group
  $\set{g \in \Sym{\Omega}}{\mathcal{U}^{g} = \mathcal{U}}$ by
  Corollary~\ref{cor-perfect-nice-condition-subgroup}.

  To illustrate this, let
  $\Omega = \{1,\ldots,6\}$ and $\mathcal{U} = [\{1,3,6\}, \{3,5\}, \{2,4\}, \{2,3,4\}]$.
  It follows that
  \[\set{g \in \Sn{6}}{{[\{1,3,6\}, \{3,5\}, \{2,4\}, \{2,3,4\}]}^{g}
  = \mathcal{U}}
  =
  \Auto{\Gamma_{\mathcal{U}}}
  =
  \< (1\,6), (2\,4) \>.\]
\end{example}

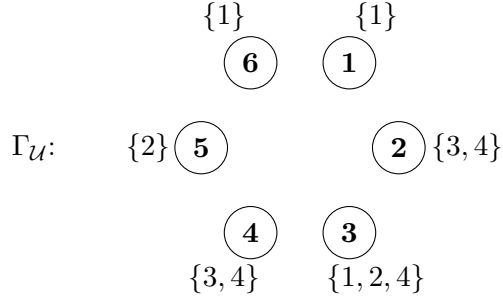
\begin{figure}[!ht]
  \begin{center}
    \begin{tikzpicture}
      \foreach \x in {1,2,3,4,5,6} {
        \node[circle, draw=black] (\x) at (-\x*60+120:1.3cm) {$\mathbf{\x}$};};

      \node at (60:2.0cm)   {$\{1\}$};
      \node at (0:2.2cm)    {$\{3,4\}$};
      \node at (-60:2.0cm)  {$\{1,2,4\}$};
      \node at (-120:2.0cm) {$\{3,4\}$};
      \node at (-180:2.0cm) {$\{2\}$};
      \node at (-240:2.0cm) {$\{1\}$};
      \node at (3.5, 0) {};
      \node at (-3.5, 0) {$\Gamma_{\mathcal{U}}$:};
    \end{tikzpicture}
  \end{center}
  \caption{
    The labelled digraph $\Gamma_{[\{1,3,6\}, \{3,5\}, \{2,4\}, \{2,3,4\}]}$ from
    Example~\ref{ex-perfect-list-of-sets}.
  }\label{fig-list-of-sets}
\end{figure}

If $n,m \in \N$ and we encode a list $[x_{1}, \ldots, x_{m}]$ in $\Omega$ as the
list of singleton
subsets $[\{x_{1}\}, \ldots, \{x_{m}\}]$, and we encode a subset $\{y_{1},
\ldots, y_{n}\} \subseteq \Omega$ as the list $[\{y_{1},
\ldots, y_{n}\}]$, then we see that Example~\ref{ex-perfect-list-of-sets} can be
used to create perfect refiners for the sets of permutations that stabilise or
transport lists in $\Omega$ or subsets of $\Omega$.

\begin{example}[Perfect refiner for set of disjoint subsets stabiliser and
  transporter;
  Figure~\ref{fig-set-of-disjoint-sets}]\label{ex-perfect-set-of-disjoint-sets}

  For every set of disjoint subsets $\mathcal{U} \coloneqq \{U_{1}, \ldots,
  U_{k}\}$,
  where $k \in \N_{0}$ and $U_{i} \subseteq \Omega$ for all $i \in \{1, \ldots,
  k\}$,
  let $\Gamma_{\mathcal{U}}$ be the labelled digraph on $\Omega$
  with arcs
  \[
  \set{(\alpha, \beta) \in \Omega \times \Omega}{\alpha \neq \beta\ \text{and}\
  \{\alpha, \beta\} \subseteq U_{i}\ \text{for some}\ i},
  \]
  where vertices in $U_{1} \cup \cdots \cup U_{k}$ have label $1$,
  and all other vertices and arcs have label $0$.
  For every $\stackS \in \Stacks{\Omega}$,
  define $f_{\mathcal{U}}(\stackS) = [\Gamma_{\mathcal{U}}]$.

  Let $\mathcal{U}$ and $\mathcal{V}$ be arbitrary sets of disjoint subsets of
  $\Omega$.  Then $(f_{\mathcal{U}}, f_{\mathcal{V}})$ is a perfect refiner for
  the set $\set{g \in \Sym{\Omega}}{\mathcal{U}^{g} = \mathcal{V}}$ by
  Lemma~\ref{lem-perfect-nice-condition}, and by
  Corollary~\ref{cor-perfect-nice-condition-subgroup}, $(f_{\mathcal{U}},
  f_{\mathcal{U}})$ is a perfect refiner for the group $\set{g \in
  \Sym{\Omega}}{\mathcal{U}^{g} = \mathcal{U}}$.

  We demonstrate this in the case of $\mathcal{U}
  \coloneqq \{\{1,2\},\{3\}\}$ and $\mathcal{V} \coloneqq \{\{3,4\},\{2\}\}$,
  with the aim of describing the
  transporter set $T \coloneqq \set{g \in \Sn{4}}{\mathcal{U}^{g} =
  \mathcal{V}}$.
  First, we build the labelled digraphs $\Gamma_{\mathcal{U}}$ and
  $\Gamma_{\mathcal{V}}$, which are shown in Figure~\ref{fig-set-of-disjoint-sets}.
  The function $f_{\mathcal{U}}$ has constant value $[\Gamma_{\mathcal{U}}]$,
  and the function $f_{\mathcal{V}}$ has constant value
  $[\Gamma_{\mathcal{V}}]$. Since $(f_{\mathcal{U}},
  f_{\mathcal{V}})$ is a perfect refiner for $T$, we can describe $T$ by
  describing $\Iso{[\Gamma_{\mathcal{U}}]}{[\Gamma_{\mathcal{V}}]}
  = \Iso{\Gamma_{\mathcal{U}}}{\Gamma_{\mathcal{V}}}$.
  It is clear that $\<(1\,2)\>$ is the automorphism group of
  $\Gamma_{\mathcal{U}}$ induced by $\Sn{4}$, and that
  $(1\,3\,2\,4)$ induces an isomorphism from $\Gamma_{\mathcal{U}}$ to
  $\Gamma_{\mathcal{V}}$.
  Therefore, $T$ is the right coset $\<(1\,2)\> \cdot (1\,3\,2\,4)$
  of $\<(1\,2)\>$ in $\Sn{4}$.
\end{example}

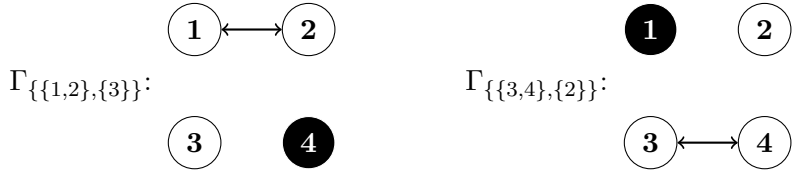
\begin{figure}[!ht]
  \begin{center}
    \begin{tikzpicture}
      \tikzstyle{white}=[circle, draw=black]
      \tikzstyle{black}=[circle, draw=white, fill=black!100]

      \node[white] (11) at (0, 1.5) {$\mathbf{1}$};
      \node[white] (12) at (1.5, 1.5) {$\mathbf{2}$};
      \node[white] (13) at (0, 0) {$\mathbf{3}$};
      \node[black] (14) at (1.5, 0) {$\color{white}\mathbf{4}$};
      \node (100) at (-1.5, 0.75) {$\Gamma_{\{\{1,2\},\{3\}\}}$:};
      \arcSym{11}{12}

      \node[black] (21) at (6,   1.5) {$\color{white}\mathbf{1}$};
      \node[white] (22) at (7.5, 1.5) {$\mathbf{2}$};
      \node[white] (23) at (6, 0) {$\mathbf{3}$};
      \node[white] (24) at (7.5, 0) {$\mathbf{4}$};
      \node (200) at (4.5, 0.75) {$\Gamma_{\{\{3,4\},\{2\}\}}$:};
      \node (300) at (9, 1) {}; 
      \arcSym{23}{24}
    \end{tikzpicture}
  \end{center}
  \caption[set-of-sets-1]{
    The labelled digraphs $\Gamma_{\{\{1,2\},\{3\}\}}$ and
    $\Gamma_{\{\{3,4\},\{2\}\}}$ from
    Example~\ref{ex-perfect-set-of-disjoint-sets}.
    The black vertices are those with label $0$.
  }\label{fig-set-of-disjoint-sets}
\end{figure}

An ordered partition of $\Omega$ is a list of subsets of $\Omega$, and an
unordered partition of $\Omega$ is a set of disjoint subsets of $\Omega$.
Therefore, Example~\ref{ex-perfect-list-of-sets}
can be used to create a perfect refiner for the set of permutations that stabilises
any particular ordered partition, or the set of permutations that transports one
ordered partition to another.
Similarly, Example~\ref{ex-perfect-set-of-disjoint-sets} can be used to create perfect
refiners for the analogous sets of permutations that involve unordered partitions.

\subsection{Refiners given by a fixed sequence of
  stacks}\label{sec-refiners-via-stack}

We have seen that when creating a refiner
for a set $U$, it is necessary to construct a
function $f$ from $\Stacks{\Omega}$ to itself that satisfies $f(\stackS^{g}) =
{f(\stackS)}^{g}$ for all $\stackS \in \Stacks{\Omega}$ and $g \in U$.  For many
of the refiners that we wish to implement, such as those given in
Section~\ref{sec-examples-perfect}, this does not cause significant difficulty,
especially when the value of the function $f$ does not depend on its input.
However, there are some refiners that we wish to define, especially
refiners for an arbitrary subgroup of $\Sym{\Omega}$ specified by a
generating set, or for a coset of such a subgroup, where satisfying the condition
described above is difficult in practice. In this section, we give an example of
this difficulty, and then present a general method for overcoming this problem.

\begin{example}\label{ex-refiner-fixed-stack}
  Let $\Omega = \{1, \ldots, 6\}$,
  let \(G = \< (1\,2), (3\,4), (5\,6), (1\,3\,5)(2\,4\,6) \>\),
  let $\Gamma$ be the labelled digraph on $\Omega$ without arcs where vertex $1$
  has label $\exLabel{black}$, vertex $2$ has label $\exLabel{grey}$, and the
  remaining vertices have label $\exLabel{white}$, and define $\stackS =
  [\Gamma]$ and $\stackT = [\Gamma^{(1\,3\,5)(2\,4\,6)}]$.

  Suppose that we are
  searching for the intersection $D$ of a number of subsets of $\Sym{\Omega}$, one
  of which is $G$, and suppose that $\Iso{\stackS}{\stackT}$ represents the
  current overestimate for the solution.
  We wish to give a refiner for $G$
  (which, by Lemma~\ref{lem-group-refiner-symmetric},
  has the form $(f, f)$ for some function $f$)
  that works by encoding relevant
  information about the orbit structure of $G$ into a new labelled digraph
  stack, since we know that the elements of $D$ respect the
  orbit structure of $G$
  (which is not immediately of much use here, because $G$ is transitive).

  However, there is further information that the refiner can use.  Since
  $\Iso{\stackS}{\stackT}$ is an overestimate for $D$, we know that every
  element of $D$ induces an isomorphism from $\stackS$ to
  $\stackT$.
  In particular, if $\fixedFunc$ is a fixed-point approximator, then every element of $D$ maps the list $\Fixed{\stackS}$ to the list $\Fixed{\stackT}$.
  A fixed-point approximator could give $\Fixed{\stackS} = [1,2]$
  and $\Fixed{\stackT} = [3,4]$; suppose that this is the case.
  Since all elements of $D$ map $[1,2]$ to $[3,4]$
  and are also contained in $G$, it follows that
  they are contained in $G_{[1,2]} \cdot h$, the right coset of the
  stabiliser of $[1,2]$ in $G$ determined by any permutation $h$ in $G$ that maps
  $[1,2]$ to $[3,4]$, such as, for example, $h\coloneqq(1\,3\,5)(2\,4\,6) \in G$.

  This means that we should be able to define $f(\stackS)$ and $f(\stackT)$ in
  terms of the orbits of the pointwise stabilisers of $[1,2]$ and $[3,4]$ in
  $G$, respectively, which are $G_{[1,2]} = \< (3\,4) (5\,6) \>$ and $G_{[3,4]}
    = \< (1\,2) (5\,6) \>$.
  Thus one option would be to define $f(\stackS) = f_{\mathcal{U}}(\stackS)$ as
  in Example~\ref{ex-perfect-set-of-disjoint-sets}, for the set of orbits
  $\mathcal{U} = \{\{1\}, \{2\}, \{3,4\}, \{5,6\}\}$ of $G_{[1,2]}$ on $\Omega$,
  and to define $f(\stackT)$ similarly for the set $\mathcal{V} = \{\{1,2\},
  \{3\}, \{4\}, \{5,6\}\}$.
  This is valid, but it is not ideal, since permutations in $\Sym{\Omega}$ can
  rearrange orbits of the same size arbitrarily while mapping $\mathcal{U}$ to
  $\mathcal{V}$, whereas elements of the right coset $G_{[1,2]} \cdot h$ can
  only map an orbit $O \in \mathcal{U}$ to the orbit $O^{h}$.  For instance,
  $(1\,4)(2\,3) \in \Sym{\Omega}$ maps $\mathcal{U}$ to $\mathcal{V}$, but this
  permutation maps the orbit $\{1\}$ of $G_{[1,2]}$ to the orbit $\{4\}$ of
  $G_{[3,4]}$, and $\{4\} \neq {\{1\}}^{h}$.  Therefore, defining the refiner
  $(f, f)$ in this way does not discard some elements that, to us, are obviously not in $D$.
  This is unsatisfactory, so we would like to define $f(\stackS) =
  f_{\mathcal{U}}(\stackS)$ as in Example~\ref{ex-perfect-list-of-sets}, for
  some \emph{ordered} list $\mathcal{U}$ of the orbits $\{1\}, \{2\}, \{3,4\},
  \{5,6\}$ of $G_{[1,2]}$ on $\Omega$.
  But then how should we choose an ordering $\mathcal{V}$ of the orbits
  $\{1,2\}, \{3\}, \{4\}, \{5,6\}$ of $G_{[3,4]}$ on $\Omega$, in order to
  define the corresponding $f(\stackT) (= f_{\mathcal{V}}(\stackT))$
  so that $\Iso{f(\stackS)}{f(\stackT)}$ does not discard elements in $D$?

  Without further techniques, we cannot answer this question.
\end{example}

To overcome this problem, we use a technique similar to that of
Leon~\cite{leon1991}, where we explicitly pre-generate a list of fixed points
and labelled digraph stacks, which are stored and then retrieved during the
search, when needed.  In essence, this allows us to make certain ordering
choices in advance, so that during the search we can consult the fixed initial
choice, and remain consistent with that.  The mathematical foundation of this
technique is described in the following lemma.

\begin{lemma}\label{lem-fixed-stack}
  Let $G \leq \Sym{\Omega}$ and let $\fixedFunc$ be a fixed-point approximator.
  For all $i \in \N_{0}$, let $\stackV_{i} \in \Stacks{\Omega}$ be a labelled
  digraph stack on $\Omega$, and let $F_{i}$ be a list of points in $\Omega$,
  such that the stabiliser $\set{g \in G}{F_{i}^{g} = F_{i}}$
  is a subgroup of $\Auto{\stackV_{i}}$.

  We define a function $f$ from $\Stacks{\Omega}$ to itself as follows.  For
  each $\stackS \in \Stacks{\Omega}$, let $a \in G$ be such that $a$ maps
  $F_{|\stackS|}$ to $\Fixed{\stackS}$ pointwise, if such an element exists,
  and otherwise let $a = \Fail$, and set
  \[
    f(\stackS) =
    \begin{cases}
      \EmptyStack{\Omega} &
        \text{if}\
        a = \Fail, \\
      {(\stackV_{|\stackS|})}^{a} &
        \text{otherwise}.
    \end{cases}
  \]
  Then $(f, f)$ is a refiner for $G$.
\end{lemma}

\begin{proof}
  Note that $f$ is well-defined: if $\stackS \in \Stacks{\Omega}$ and
  $a, b \in G$ both map $F_{|\stackS|}$ to $\Fixed{\stackS}$, then
  $a b^{-1}$ stabilises $F_{|\stackS|}$, and so $a b^{-1} \in
  \Auto{\stackV_{|\stackS|}}$, and hence
  ${(\stackV_{|\stackS|})}^{a} =
  {(\stackV_{|\stackS|})}^{b}$.

  Let $\stackS \in \Stacks{\Omega}$ and $g \in G$ be arbitrary.
  First note that $|\stackS| = |\stackS^{g}|$.
  Let $a$ be an element of $G$ that maps $F_{|\stackS|}$ to
  $\Fixed{\stackS}$, if such an element exists, and set $a = \Fail$
  otherwise.  Similarly, let $b$ be an element of $G$ that maps
  $F_{|\stackS|}$ to $\Fixed{\stackS^{g}}$, if such an element exists,
  and set $b = \Fail$ otherwise.  Note that $a = \Fail$ if and
  only if $b = \Fail$, since $g \in G$ maps
  $\Fixed{\stackS}$ to $\Fixed{\stackS^{g}}$, by
  Definition~\ref{defn-approx-fixed}\ref{item-fixed-invariant}.  Thus, if $a =
    \Fail$, then $f(\stackS^{g})$ and ${f(\stackS)}^{g}$ are
  both empty.
  Otherwise $a \in G$, and so
  $f(\stackS) = {(\stackV_{|\stackS|})}^{a}$
  and $f(\stackS^{g}) = {(\stackV_{|\stackS|})}^{b}$.
  It remains to prove in this case that ${(\stackV_{|\stackS|})}^{a g} =
  {(\stackV_{|\stackS|})}^{b}$, or equivalently that $a g b^{-1} \in
    \Auto{\stackV_{|\stackS|}}$.
  The choice of $a$, $g$ and $b$ implies that $a g b^{-1}$ fixes every entry of
  $F_{|\stackS|}$, and
  so $a g b^{-1} \in \Auto{\stackV_{|\stackS|}}$ by assumption.

  Therefore $f(\stackS^{g}) = {f(\stackS)}^{g}$ for all
  $\stackS \in \Stacks{\Omega}$ and $g \in G$, and so
  $(f, f)$ is a refiner for $G$ by Lemma~\ref{lem-simple-refiner}.
\end{proof}

In isolation, Lemma~\ref{lem-fixed-stack} may seem very abstract. In particular,
the lemma does not specify how the lists $F_{i}$ and the stacks $\stackV_{i}$
should be chosen in the first place. We postpone these details until
Section~\ref{sec-details-refiner-fixed-stack}, because the real usefulness of
this technique becomes apparent with the organisation of our forthcoming algorithms.

Given two lists of points in $\Omega$ and a subgroup $G \leq
\Sym{\Omega}$ given by a generating set, there exist efficient algorithms
that either construct an element of $G$ that maps the first list of points to
the second, or determine that no such element exists.  In
\textsc{GAP}~\cite{GAP4}, this can be achieved via the function
\texttt{RepresentativeAction}.

Note that, in combination with Lemma~\ref{lem-refiner-right-coset},
Lemma~\ref{lem-fixed-stack} may be used to build a refiner for any coset
for which we have a representative.

\section{Distributing stack isomorphisms across new stacks}\label{sec-splitter}

In backtrack search, when it is not readily apparent how to further prune a
search space, it is necessary to divide the search across a number
of subproblems, each of which, being smaller, can be solved more easily. We call
this process \emph{splitting}.

In order to organise a backtrack search around stacks of labelled digraphs,
therefore, we need to be able to implement a version of splitting for labelled
digraphs stacks.  In other words, we need to be able to take a pair of labelled
digraph stacks that represents a potentially large search space, and
define new pairs of stacks that divide the search space amongst themselves in a
sensible way.  Since the search space that a pair of stacks represents is the
overestimated set of isomorphisms from the first stack to the second, it follows
that we need to be able to take one pair of stacks and define new pairs,
such that the originally estimated set of isomorphisms is subdivided across
these new pairs.
In this section, we define and discuss the notion of a splitter
for labelled digraph stacks.

\begin{definition}\label{defn-splitter}
  A \emph{splitter} for an isomorphism approximator $\approxFunc$ is a function
  $\splitFunc$ that maps one pair of labelled digraph stacks on $\Omega$ to a
  finite list of pairs of stacks,
  such that for all $\stackS,
  \stackT \in \Stacks{\Omega}$ with $|\Approx{\stackS}{\stackT}| \geq 2$,
  \[
    \Split{\stackS}{\stackT}
    =
    [
    (\stackS_{1}, \stackT_{1}),
    (\stackS_{2}, \stackT_{2}),
    \ldots,
    (\stackS_{m}, \stackT_{m})
    ]
  \]
  for some $m \in \N_{0}$ and $\stackS_{1}, \ldots, \stackS_{m}, \stackT_{1},
  \ldots, \stackT_{m} \in \Stacks{\Omega}$, and:
  \begin{enumerate}[label=\textrm{(\roman*)}]
    \item\label{item-splitter-disjoint}
      $\Iso{\stackS}{\stackT}
        =
        \Iso{\stackS \Vert \stackS_{1}}{\stackT \Vert \stackT_{1}}
        \mathop{\dot{\cup}}
        \cdots
        \mathop{\dot{\cup}}
        \Iso{\stackS \Vert \stackS_{m}}{\stackT \Vert \stackT_{m}}$
        (a disjoint union).

      \item\label{item-splitter-smaller}
      $|\Approx{\stackS \Vert \stackS_{i}}{\stackT \Vert \stackT_{i}}|
       <
       |\Approx{\stackS}{\stackT}|$ for all $i \in \{1, \ldots, m\}$.

    \item\label{item-splitter-stab-first}
      If $\stackS = \stackT$, then $\stackS_{1} = \stackT_{1}$.

    \item\label{item-splitter-invariant}
      For all $\stackU \in \Stacks{\Omega}$ with $|\Approx{\stackS}{\stackU}|
      \geq
      2$, there exists some $n \in \N$ and $\stackU_{i} \in \Stacks{\Omega}$ for
      all $i \in \{1, \ldots, n\}$ such that
      \[
        \Split{\stackS}{\stackU}
        =
        [
        (\stackS_{1}, \stackU_{1}),
        (\stackS_{1}, \stackU_{2}),
        \ldots,
        (\stackS_{1}, \stackU_{n})
        ].
      \]
  \end{enumerate}
\end{definition}

We briefly explain the purpose of the conditions in
Definition~\ref{defn-splitter}, retaining its notation.
In
Definition~\ref{defn-splitter}, we do not restrict the behaviour of a splitter
when given stacks $\stackS$ and $\stackT$ with
$|\Approx{\stackS}{\stackT}| \leq 1$, since this situation never occurs in the
algorithms of Section~\ref{sec-search}.
Throughout the
following explanation, we assume that $|\Approx{\stackS}{\stackT}| \geq 2$.

The new subproblems to which the splitter gives rise are the pairs of stacks of
the form $(\stackS \Vert \stackS_{i}, \stackT \Vert \stackT_{i})$ for each $i \in \{1,\ldots,m\}$.
Definition~\ref{defn-splitter}\ref{item-splitter-disjoint} ensures that the set of isomorphisms induced by $\Sym{\Omega}$ from $\stackS$ to $\stackT$ is
shared amongst these new pairs of stacks.  In other words, each $g \in
\Iso{\stackS}{\stackT}$ belongs to a unique set $\Iso{\stackS_{i}}{\stackT_{i}}$
for some $i\in \{1,\ldots,m\}$, and so $g$ belongs to the unique set $\Iso{\stackS \Vert
\stackS_{i}}{\stackT \Vert \stackT_{i}}$ for the same $i$.  This means that each
solution to a search problem appears in exactly one branch of the search tree, and
so we do not waste resources by repeatedly discovering the same solution.
Definition~\ref{defn-splitter}\ref{item-splitter-disjoint} also implies that no
new isomorphisms are introduced by splitting.

Definition~\ref{defn-splitter}\ref{item-splitter-smaller} ensures that each new
pair of stacks gives rise to a strictly smaller subproblem than does the
original pair of stacks, which will be required to show that our algorithms
terminate.  It would be ideal for each new approximation of the form
$\Approx{\stackS \Vert \stackS_{i}}{\stackT \Vert \stackT_{i}}$ to be
\emph{disjoint} from the other new approximations, but we do not need to require
this.

Definition~\ref{defn-splitter}\ref{item-splitter-stab-first}
and~\ref{item-splitter-invariant} are technical conditions, not of
deep mathematical importance,
but they are used in
Sections~\ref{sec-generators} and~\ref{sec-r-base}.
Definition~\ref{defn-splitter}\ref{item-splitter-stab-first}
is useful in a notational sense when it comes to describing an algorithm to
search for a generating set for a subgroup.
Definition~\ref{defn-splitter}\ref{item-splitter-invariant} implies that
\(\stackS_{1} = \stackS_{2} = \cdots = \stackS_{m}\). This is useful when it comes
to applying refiners of the kind introduced in Section~\ref{sec-refiners-via-stack}.

The following lemma shows a way of giving a splitter by specifying its behaviour
on the left stack that it is given.

\begin{lemma}\label{lem-splitter-creator}
  Let $\approxFunc$ be an isomorphism approximator, and let $f$ be any
  function from $\Stacks{\Omega}$ to itself such that, for all
  $\stackS \in \Stacks{\Omega}$:
  \[\text{if}\ |\Approx{\stackS}{\stackS}| \geq 2,
    \ \text{then}\
    |\Approx{\stackS \Vert f(\stackS)}{\stackS \Vert f(\stackS)}|
    <
    |\Approx{\stackS}{\stackS}|.
  \]
  Let $\stackS, \stackT \in \Stacks{\Omega}$. If $|\Approx{\stackS}{\stackT}|
  \leq 1$, then we define $\Split{\stackS}{\stackT}$ to be empty. Otherwise, we
  choose a fixed enumeration $\stackT_{1}, \ldots, \stackT_{m}$ of
  $\set{{f(\stackS)}^{g}}{g \in \Approx{\stackS}{\stackT}}$, in particular we set $\stackT_{1} =
  f(\stackS) = {f(\stackS)}^{\idOmega}$ if $\stackS = \stackT$, and we set
  $\Split{\stackS}{\stackT} = [(f(\stackS), \stackT_{1}), \ldots, (f(\stackS),
  \stackT_{m})]$.
  Then $\splitFunc$ is a splitter for $\approxFunc$.
\end{lemma}

\begin{proof}
  Let $\stackS, \stackT \in \Stacks{\Omega}$ with
  $|\Approx{\stackS}{\stackT}| \geq 2$, and assume that
  $\Split{\stackS}{\stackT}$ and
  $[(f(\stackS), \stackT_{1}), \ldots, (f(\stackS),
  \stackT_{m})]$ are defined as in the statement of the lemma.

  We first prove that the equation in
  Definition~\ref{defn-splitter}\ref{item-splitter-disjoint} holds.  Since
  $|\stackS| = |\stackT|$ by
  Definition~\ref{defn-approx-iso}\ref{item-approx-different-lengths}, it
  follows from Remark~\ref{rmk-stack-iso-auto}
  that the left hand side of this equation contains the right hand side,
  so it remains to show the reverse inclusion, and that the right hand side is a
  disjoint union.
  Let $g \in \Iso{\stackS}{\stackT} \subseteq \Approx{\stackS}{\stackT}$ be
  arbitrary.
  Since ${f(\stackS)}^{g} = \stackT_{i}$ for some $i \in \{1, \ldots, m\}$, it
  follows that $g \in \Iso{\stackS \Vert f(\stackS)}{\stackT \Vert
  \stackT_{i}}$.
  If there exists some $h \in \Iso{\stackS \Vert f(\stackS)}{\stackT \Vert
  \stackT_{i}} \cap \Iso{\stackS \Vert f(\stackS)}{\stackT \Vert \stackT_{j}}$
  for some $i, j \in \{1, \ldots, m\}$, then in particular $\stackT_{i} =
  {f(\stackS)}^{h} = \stackT_{j}$, and so $i = j$.

  To show that Definition~\ref{defn-splitter}\ref{item-splitter-smaller} holds,
  let $i \in \{1, \ldots, m\}$.
  If $\Approx{\stackS \Vert f(\stackS)}{\stackT \Vert \stackT_{i}} =
  \varnothing$, then we are done, so suppose otherwise.  By
  Definition~\ref{defn-approx-iso}\ref{item-approx-right-coset-of-aut} and by
  assumption,
  \begin{align*}
    |\Approx{\stackS \Vert f(\stackS)}{\stackT \Vert \stackT_{i}}|
    =
    |\Approx{\stackS \Vert f(\stackS)}{\stackS \Vert f(\stackS)}|
    <
    |\Approx{\stackS}{\stackS}|
    =
    |\Approx{\stackS}{\stackT}|,
  \end{align*}
  as required.
  Definition~\ref{defn-splitter}\ref{item-splitter-stab-first}
  and~\ref{item-splitter-invariant} hold by construction.
\end{proof}

Let the notation of Lemma~\ref{lem-splitter-creator} hold.
Note that we imposed no further conditions in Lemma~\ref{lem-splitter-creator}
on the enumeration of $\set{{f(\stackS)}^{g}}{g \in \Approx{\stackS}{\stackT}}$
beyond the condition when $\stackS = \stackT$, because there is no mathematical
need to, as long as it is consistent.
This set can be computed via the orbit of $f(\stackS)$ under the action of
$\Approx{\stackS}{\stackS}$.
Indeed, if $h \in \Approx{\stackS}{\stackT}$, then
\begin{align*}
  \set{{f(\stackS)}^{g}}{g \in \Approx{\stackS}{\stackT}}
  & =
  \set{{f(\stackS)}^{g}}{g \in \Approx{\stackS}{\stackS} \cdot h} \\
  & =
  \set{{f(\stackS)}^{x}}{x \in \Approx{\stackS}{\stackS}}^{h} \\
  & =
  {\left({f(\stackS)}^{\Approx{\stackS}{\stackS}}\right)}^{h},
\end{align*}
by
Definition~\ref{defn-approx-iso}\ref{item-approx-right-coset-of-aut}.
In particular,
$|\Split{\stackS}{\stackT}| = |{f(\stackS)}^{\Approx{\stackS}{\stackS}}|$.

In the following definition, we give a specific instance of a splitter that can
be obtained with Lemma~\ref{lem-splitter-creator}.  Here, still using the
notation of Lemma~\ref{lem-splitter-creator}, appending the stack $f(\stackS)$
to the stack $\stackS$ corresponds to stabilising $f(\stackS)$ in the current
approximation of $\Auto{\stackS}$; the stacks of the form
$\stackT_{i}$ correspond to the images of $f(\stackS)$ under
$\Approx{\stackS}{\stackT}$.
\begin{definition}[Fixed point splitter]\label{defn-point-splitter}
  For all $\alpha \in \Omega$, let $\Gamma_{\alpha} = (\Omega,
  \varnothing, \labelFunc)$ be the labelled digraph on $\Omega$ where
  $\labelFunc(\alpha) =
  1$ and $\labelFunc(\beta) = 0$ for all $\beta \in \Omega \setminus
  \{\alpha\}$.
  Note that $\Gamma_{\alpha}^{g} = \Gamma_{\alpha^{g}}$ for all $g \in
  \Sym{\Omega}$.
  Let $\approxFunc$ be any isomorphism approximator such that
  $\Approx{\stackU \Vert [\Gamma_{\alpha}]}{\stackU \Vert [\Gamma_{\alpha}]}
  \leq \Approx{\stackU}{\stackU} \cap \set{g \in \Sym{\Omega}}{\alpha^{g} =
  \alpha}$
  for all $\alpha \in \Omega$ and $\stackU \in \Stacks{\Omega}$.
  We define a function $f$ from $\Stacks{\Omega}$ to itself by
  \[
    f(\stackS) =
    \begin{cases}
        \EmptyStack{\Omega}
        &
        \text{if}\ |\Approx{\stackS}{\stackS}| \leq 1, \\
        [\Gamma_{\alpha}]
        &
        \text{otherwise, where} \
        \alpha \coloneqq
        \min\{\min(\mathcal{O})\,:\, \mathcal{O}\ \text{is an orbit of} \\
        &
        \qquad\quad
        \Approx{\stackS}{\stackS}\
        \text{of minimal size, subject to}\
        |\mathcal{O}| \geq 2 \},
    \end{cases}
  \]
  for all $\stackS \in \Stacks{\Omega}$.
  Finally, define $\splitFunc$ as in Lemma~\ref{lem-splitter-creator},
  for the isomorphism approximator $\approxFunc$ and the function $f$.
\end{definition}

The following corollary holds by Lemma~\ref{lem-splitter-creator}.
The crucial step
is showing that the function $f$ has the required property; this follows by the
careful choice of isomorphism approximator in
Definition~\ref{defn-point-splitter}.

\begin{cor}\label{cor-point-splitter}
  The function $\splitFunc$ from Definition~\ref{defn-point-splitter} is a splitter for
  any isomorphism approximator that satisfies the condition in
  Definition~\ref{defn-point-splitter}.
\end{cor}

\section{The search algorithm}\label{sec-search}

In this section, we present our main algorithms, which combine the tools from
Sections~\ref{sec-stacks}--\ref{sec-splitter} to solve search problems in
$\Sym{\Omega}$.  A version of our algorithms is implemented in the
\textsc{GraphBacktracking}~\cite{GAPpkg} package for \textsc{GAP}~\cite{GAP4}.

Let $U_{1}, \ldots, U_{k} \subseteq \Sym{\Omega}$, and suppose that we have
a collection of refiners for these subsets, an isomorphism approximator, and a
corresponding splitter.
The section proceeds as follows:
in Section~\ref{sec-basic-search} we show how it is
possible to use these refiners, the approximator, and the splitter to perform a backtrack
search that finds one or all elements of the intersection $U_{1} \cap \cdots \cap U_{k}$.
In Section~\ref{sec-generators}, we describe how, when the result is
known to form a subgroup of $\Sym{\Omega}$, it is possible to search for a base
and strong generating set for the subgroup (see for
example~\cite[p.~101]{dixonmortimer}),
rather than for the set of all its elements. This is also useful when searching
for a coset of a subgroup.
Finally, in
Section~\ref{sec-r-base}, we explain the use of the refiners described in
Section~\ref{sec-refiners-via-stack}.

\subsection{The basic procedure}\label{sec-basic-search}

What follows is a high-level description of Algorithm~\ref{alg-search},
which is the main algorithm of this section.  This algorithm comprises the
\Call{Search}{} and \Call{Refine}{} procedures, and begins with a call to
the \Call{Search}{} procedure on line~\ref{line-start-search}.
We say that the algorithm \emph{backtracks} when it finishes executing one
recursive call to the \Call{Search}{} procedure, and goes back to the point
where it was initiated, in order to continue.

Each subset of $\Sym{\Omega}$ given as input to Algorithm~\ref{alg-search}
should be specified in such a way that it is computationally inexpensive to test
whether or not an arbitrary element of $\Sym{\Omega}$ belongs to the set.
For example, the set could be a subgroup specified by a generating set,
or it could be defined as the subset of elements of $\Sym{\Omega}$ that conjugate
the subgroup $G$ to the subgroup $H$, where $G$ and $H$ are given by generating
sets.  Note that the number of subsets given as input does not need to equal the
number of given refiners. For instance, a subset could have multiple refiners,
or none.

\begin{algorithm}[!ht]
  \caption{A recursive algorithm using labelled digraph stacks to search
    in $\Sym{\Omega}$.}\label{alg-search}

  \begin{algorithmic}[1]

    \item[\textbf{Input:}]
    a sequence of subsets
    $U_{1}, \ldots, U_{k} \subseteq \Sym{\Omega}$ in which membership is
    easily tested;

    a sequence $(f_{L,1}, f_{R,1}), \ldots, (f_{L,m}, f_{R,m})$, where
    each pair is a refiner for some $U_{j}$;

    an isomorphism approximator $\approxFunc$ and a splitter $\textsc{Split}$
    for $\approxFunc$.

    \item[\textbf{Output:}]
    all elements of the intersection $U_{1} \cap \cdots \cap U_{k}$,
    which we refer to as `solutions'.

    \vspace{2mm}
    \Procedure{Search}{$\stackS, \stackT$}
    \Comment{The main recursive search procedure.}\label{line-search-proc-start}

    \State{\((\stackS, \stackT) \gets \Call{Refine}{\stackS, \stackT}\)}
    \Comment{Refine the given stacks.}\label{line-refine}

    \Case{$\Approx{\stackS}{\stackT} = \varnothing$}
      \Comment{Nothing found in the present branch: backtrack.}\label{line-case-empty}
      \State{\Return{\(\varnothing\)}}\label{line-no-isos}
    \EndCase{}

    \Case{$\Approx{\stackS}{\stackT} = \{h\}$ for some $h$}
      \Comment{$h$ is the sole potential solution
      here.}\label{line-case-one-candidate}
      \If{
          $\stackS^{h} = \stackT$ and
          $h \in U_{1} \cap \cdots \cap U_{k}$}\label{line-check-perm}
      \State{\Return{$\{h\}$}}
      \Comment{$h$ is a solution in
        $\Iso{\stackS}{\stackT}$: backtrack.}\label{line-single-solution}
      \Else{}
      \State{\Return{$\varnothing$}}
      \Comment{$h$ is not a solution in
        $\Iso{\stackS}{\stackT}$: backtrack.}\label{line-not-solution}
      \EndIf{}
    \EndCase{}

    \Case{$|\Approx{\stackS}{\stackT}| \geq 2$}
      \Comment{Multiple potential solutions.}\label{line-multiple-candidates}

      \State{\Return{\( \hspace{-6mm}
        \displaystyle\bigcup_{(\stackS_{i}, \stackT_{i})
          \in
          \Call{Split}{\stackS, \stackT}}\hspace{-6mm}
        \Call{Search}{\stackS \Vert \stackS_{i}, \stackT \Vert \stackT_{i}}\)}}
      \Comment{Split, and search recursively.}\label{line-split}
    \EndCase{}

    \EndProcedure{}

    \vspace{1mm}
    \Procedure{Refine}{$\stackS, \stackT$}
    \Comment{Attempt to remove non-solutions from
      $\Iso{\stackS}{\stackT}$.}\label{line-refine-proc-start}

    \While{\(\Approx{\stackS}{\stackT} \neq
    \varnothing\)}
    \Comment{Proceed while there are potential
    solutions.}\label{line-while-begin}\label{line-refine-empty}

    \State{\((\stackS', \stackT') \gets
      (\stackS, \stackT)\)}
    \Comment{Save the stacks before the next round of
    refinements.}\label{line-refine-store}

    \For{\(i \in \{1, \ldots, m\}\) \textbf{and while} \(|\stackS| =
    |\stackT|\)}\label{line-refine-loop}
      \State{$(\stackS, \stackT) \gets
          (\stackS \Vert f_{L,i}(\stackS),
          \stackT \Vert f_{R,i}(\stackT))$}
      \Comment{Apply each refiner in turn.}\label{line-apply-refiner}
    \EndFor{}

    \If{\(|\Approx{\stackS}{\stackT}| \not<
      |\Approx{\stackS'}{\stackT'}|\)}\label{line-refine-not-smaller}

      \State{\Return{\((\stackS', \stackT')\)}}
      \Comment{Stop: the last refinements seemingly made no
        progress.}\label{line-refine-loop-end}

    \EndIf{}

    \EndWhile{}

    \State{\Return{\((\stackS, \stackT)\)}}\Comment{Stop:
    $\Approx{\stackS}{\stackT} = \varnothing$: no solutions in this
    branch.}\label{line-refine-return-empty}

    \EndProcedure{}

    \vspace{1mm}
    \State{\Return{\Call{Search}{$\EmptyStack{\Omega},
            \EmptyStack{\Omega}$}}}\label{line-start-search}

  \end{algorithmic}
\end{algorithm}

At any point during the execution of the algorithm, we have a pair of labelled
digraph stacks $(\stackS, \stackT)$ whose corresponding set
$\Iso{\stackS}{\stackT}$ overestimates the set of solutions to the current problem.
(The current problem might be the full problem, or it might be a subproblem
produced by a splitter.)
However, since we do not necessarily wish to calculate
\(\Iso{\stackS}{\stackT}\) exactly, in practice we only have access to
\(\Approx{\stackS}{\stackT}\), an overestimate for \(\Iso{\stackS}{\stackT}\).

The \Call{Search}{} procedure is first called
on line~\ref{line-start-search}, and later it may be called recursively on
line~\ref{line-split}.  As we prove in Lemma~\ref{lem-alg-search-is-correct},
this procedure takes a pair of labelled digraph stacks $\stackS$
and $\stackT$, and returns the set of all elements in $U_{1} \cap \cdots \cap
U_{k}$ that induce isomorphisms from $\stackS$ to $\stackT$.
It does so by first using the \Call{Refine}{} procedure to refine the pair of
stacks that it is given.

Roughly speaking, the \Call{Refine}{} procedure uses refiners to encode
information about the search problem into the stacks $\stackS$ and $\stackT$,
thereby potentially reducing the size of the remaining search space, without
losing any valid solutions. We state and prove this in a precise way in
Lemma~\ref{lem-alg-refine-is-correct}.

The \Call{Refine}{} procedure repeatedly applies each refiner in turn, until
either it determines that there are no induced isomorphisms from the current
first stack to the current second stack (and hence there are no solutions to the
current problem), or it realises that the most recent round of refiner
applications failed to lead to a smaller approximation (which we interpret as
an indication that the refiners are unable to encode further useful information
into the stacks).

The next step of the \Call{Search}{} procedure is determined by the
value of the isomorphism approximator. The algorithm has reached a leaf
of the search tree if the isomorphism approximator determines that there is at
most one solution to the current problem, in which case the \Call{Search}{}
procedure takes the appropriate behaviour, and backtracks. Otherwise, the
\Call{Search}{} procedure uses a splitter to divide the current problem into
smaller subproblems. In more detail:

\begin{itemize}
  \item
    If the approximator determines that the pair of stacks is non-isomorphic,
    then the algorithm backtracks, because it has proved that there are no
    solutions to the current problem.
  \item
    If the approximator estimates that there is a single potential isomorphism
    from the first stack to the second, then the \Call{Search}{} procedure tests
    whether this element is both an isomorphism and a solution to the search problem, and returns it
    if so. The algorithm then backtracks, since the current
    problem has been exhaustively searched.
  \item
    If the approximator estimates that there are at least two isomorphisms from
    the first stack to the second (and therefore, there are at least two
    potential solutions), then the \Call{Search}{} procedure uses a splitter to
    produce pairs of labelled digraph stacks that represent smaller subproblems,
    and the \Call{Search}{} procedure is then called recursively on these new
    pairs.  This constructs the set of solutions to each of these subproblems,
    and the union of these sets is the set of solutions to the current
    problem.
\end{itemize}

It remains to prove that, given a valid combination of inputs, and after a
finite number of steps (Lemma~\ref{lem-alg-terminates}),
Algorithm~\ref{alg-search} returns the stated output
(Theorem~\ref{thm-alg-is-correct}). First, the definition of the \Call{Refine}{}
procedure yields the following lemma:\

\begin{lemma}\label{lem-refine-gives-smaller-approx}
  Let $\approxFunc$ and the \Call{Refine}{} procedure be defined as in
  Algorithm~\ref{alg-search}. Then
  \(
  |\Approx{\Call{Refine}{\stackS, \stackT}}{}| \leq |\Approx{\stackS}{\stackT}|
  \)
  for all $\stackS, \stackT \in \Stacks{\Omega}$.
\end{lemma}

\begin{lemma}\label{lem-alg-terminates}
  Given valid input, Algorithm~\ref{alg-search} terminates after a finite
  number of steps.
\end{lemma}

\begin{proof}
  We assume that computing the value of a refiner, the isomorphism approximator,
  or the splitter each counts as a single step in the execution of
  Algorithm~\ref{alg-search}.

  We first prove that the \Call{Refine}{} procedure, given the pair of stacks
  $(\stackS, \stackT)$, terminates after a finite number of steps.
  If $\Approx{\stackS}{\stackT} = \varnothing$, then the procedure terminates
  immediately. Otherwise, the procedure runs the loop on
  lines~\ref{line-refine-empty}--\ref{line-refine-loop-end}. We will show that
  this loop terminates.

  As in line~\ref{line-refine-store} of Algorithm~\ref{alg-search}, let
  $(\stackS', \stackT')$ denote the pair of stacks
  at the beginning of the first iteration of the loop, and let
  $(\stackS, \stackT)$ be the pair of stacks at the end of this iteration,
  after applying the refiners.
  The loop will iterate again if and only if
  $0 < |\Approx{\stackS}{\stackT}| < |\Approx{\stackS'}{\stackT'}|$.
  Note that an isomorphism approximator gives a subset of $\Sym{\Omega}$, which
  is finite by definition.  Define $t = |\Approx{\stackS}{\stackT}| \in \N_{0}$.

  In the case that the loop iterates again,
  then at the beginning of its next iteration, we redefine
  $(\stackS', \stackT') \coloneqq (\stackS, \stackT)$ on
  line~\ref{line-refine-store}, and we obtain the new pair of stacks
  $(\stackS, \stackT)$ by applying the refiners again.
  Now, the loop will iterate again if and only if
  $0 < |\Approx{\stackS}{\stackT}| < t$.
  By continuing to argue in this way, we see that for every iteration
  of the loop, we can add an entry to a strictly decreasing sequence of
  non-negative integers that begins with $t$.  Therefore, the loop iterates only
  a finite number of times.

  To complete the proof, we prove by induction that, for all $n \in \N_{0}$,
  the \Call{Search}{}
  procedure terminates when given
  stacks $\stackS, \stackT \in \Stacks{\Omega}$ with
  $|\Approx{\stackS}{\stackT}| = n$.
  Recall that an isomorphism approximator always gives a finite subset,
  by definition.

  In the inductive base case of $|\Approx{\stackS}{\stackT}| = 0$, the pair $(\stackS,
  \stackT)$ is replaced on line~\ref{line-refine} by another pair of stacks
  $(\stackS, \stackT)$ with $\Approx{\stackS}{\stackT} = \varnothing$
  (Lemma~\ref{lem-refine-gives-smaller-approx}) and the procedure terminates on
  line~\ref{line-no-isos}.
  Let $n \in \N$, assume that $\Call{Search}{\stackS, \stackT}$ terminates for
  all $\stackS, \stackT \in \Stacks{\Omega}$ with $|\Approx{\stackS}{\stackT}|
  < n$, and let $\stackS, \stackT \in \Stacks{\Omega}$ be such that
  $|\Approx{\stackS}{\stackT}| = n$.  On line~\ref{line-refine}, the pair
  $(\stackS, \stackT)$ is replaced by another pair of stacks $(\stackS,
  \stackT)$ with $|\Approx{\stackS}{\stackT}| \leq n$
  (Lemma~\ref{lem-refine-gives-smaller-approx}).  If
  $|\Approx{\stackS}{\stackT}| \in \{0,1\}$, then the \Call{Search}{} procedure
  terminates on one of lines~\ref{line-no-isos},~\ref{line-single-solution},
  or~\ref{line-not-solution}. Otherwise, on line~\ref{line-split}, the splitter
  is used to create a finite list of new pairs of labelled digraph stacks.  For
  each such pair $(\stackS_{i}, \stackT_{i})$, it follows by
  Definition~\ref{defn-splitter}\ref{item-splitter-smaller} that
  \[
  |\Approx{\stackS \Vert \stackS_{i}}{\stackT \Vert \stackT_{i}}|
  <
  |\Approx{\stackS}{\stackT}| \leq n.
  \]
  Therefore, each call $\Call{Search}{\stackS \Vert \stackS_{i}, \stackT \Vert
  \stackT_{i}}$ terminates by the inductive hypothesis. Since there are only
  finitely many of these calls to \Call{Search}{} on line~\ref{line-split}, and
  since this is the last line of the \Call{Search}{} procedure, it follows that
  $\Call{Search}{\stackS, \stackT}$ as a whole terminates.

  Therefore the call to
  \Call{Search}{} on line~\ref{line-start-search}~--~and thus
  Algorithm~\ref{alg-search}~--~terminates.
\end{proof}

\begin{lemma}\label{lem-alg-refine-is-correct}
  Let $\stackS, \stackT \in \Stacks{\Omega}$, and let the notation of
  Algorithm~\ref{alg-search} hold.  Then
  \(
  (U_{1} \cap \cdots \cap U_{k})
    \cap \Iso{\stackS}{\stackT}
  =
  (U_{1} \cap \cdots \cap U_{k})
    \cap \Iso{\Call{Refine}{\stackS, \stackT}}{}
  \).
\end{lemma}

\begin{proof}
  The \Call{Refine}{} procedure returns a pair of labelled digraph stacks that
  is obtained from the original pair $(\stackS, \stackT)$ only by the
  application of refiners on line~\ref{line-apply-refiner} to stacks of equal
  lengths.  Therefore, it suffices to show that if $i \in \{1, \ldots, m\}$ and
  $|\stackS| = |\stackT|$, then
  \begin{equation}\label{eq-alg-refine-is-correct}
    (U_{1} \cap \cdots \cap U_{k})
      \cap
    \Iso{\stackS}{\stackT}
    =
    (U_{1} \cap \cdots \cap U_{k})
      \cap
    \Iso{\stackS \Vert f_{L,i}(\stackS)}{\stackT \Vert f_{R,i}(\stackT)}.
    \tag{$\star$}
  \end{equation}
  If $\Iso{\stackS}{\stackT} = \varnothing$, then since $|\stackS| =
  |\stackT|$, Remark~\ref{rmk-stack-iso-auto}
  implies that $\Iso{\stackS \Vert f_{L,i}(\stackS)}{\stackT \Vert
  f_{R,i}(\stackT)}$ is empty, and so~\eqref{eq-alg-refine-is-correct}
  holds.  Otherwise, since there exists $j \in \{1, \ldots, k\}$ such that
  $(f_{L,i}, f_{R,i})$ is a refiner for the set $U_{j}$, it follows by
  Lemma~\ref{lem-refiner-equiv-definitions}\ref{item-refiner-long} that
  \[
  U_{j}
    \cap \Iso{\stackS}{\stackT}
  =
  U_{j}
    \cap \Iso{\stackS \Vert f_{L,i}(\stackS)}{\stackT \Vert f_{R,i}(\stackT)}.
  \]
  Since $U_{1} \cap \cdots \cap U_{k} \subseteq U_{j}$, it follows
  that~\eqref{eq-alg-refine-is-correct} holds, as required.
\end{proof}

\begin{lemma}\label{lem-alg-search-is-correct}
  Let $\stackS, \stackT \in \Stacks{\Omega}$, and let the notation of
  Algorithm~\ref{alg-search} hold. Then
  \[
  \Call{Search}{\stackS, \stackT}
  =
  (U_{1} \cap \cdots \cap U_{k}) \cap \Iso{\stackS}{\stackT}.
  \]
  In particular,
  $\Call{Search}{\stackS, \stackS} = (U_{1} \cap \cdots \cap U_{k}) \cap
  \Auto{\stackS}$, and so if $U_{1} \cap \cdots \cap U_{k}$ is a subgroup of
  $\Sym{\Omega}$, then $\Call{Search}{\stackS, \stackS}$ is a subgroup of
  $\Sym{\Omega}$, too.
\end{lemma}

\begin{proof}
  We proceed by induction on
  $|\Approx{\stackS}{\stackT}|$,
  as in the proof of Lemma~\ref{lem-alg-terminates}.
  If $|\Approx{\stackS}{\stackT}| = 0$, then
  $\Iso{\stackS}{\stackT} = \varnothing$ by
  Definition~\ref{defn-approx-iso}\ref{item-approx-true-overestimate}, and so
  $(U_{1} \cap \cdots \cap U_{k}) \cap \Iso{\stackS}{\stackT} = \varnothing$.
  Furthermore, if $|\Approx{\stackS}{\stackT}| = 0$,
  then \Call{Search}{$\stackS, \stackT$} returns
  $\varnothing$ on line~\ref{line-no-isos}. Thus we have established the base
  case.

  Let $n \in \N$, assume that the statement holds
  for all $\stackS, \stackT \in \Stacks{\Omega}$ with
  $|\Approx{\stackS}{\stackT}| < n$, and let $\stackS, \stackT \in
  \Stacks{\Omega}$ with $|\Approx{\stackS}{\stackT}| = n$.  Note that
  $|\Approx{\Call{Refine}{\stackS, \stackT}{}}{}| \leq n$ by
  Lemma~\ref{lem-refine-gives-smaller-approx}, and that
  \[
  (U_{1} \cap \cdots \cap U_{k})
    \cap \Iso{\stackS}{\stackT}
  =
  (U_{1} \cap \cdots \cap U_{k})
    \cap \Iso{\Call{Refine}{\stackS, \stackT}}{}
  \]
  by
  Lemma~\ref{lem-alg-refine-is-correct}.
  On line~\ref{line-refine}, the pair $(\stackS, \stackT)$ is replaced by the
  value of $\Call{Refine}{\stackS, \stackT}$, which by the above equation leaves
  the value of $(U_{1} \cap \cdots \cap U_{k}) \cap \Iso{\stackS}{\stackT}$
  unchanged.

  If $|\Approx{\stackS}{\stackT}| = 0$, then as above, the procedure
  correctly returns $\varnothing$ on line~\ref{line-no-isos}.

  If $\Approx{\stackS}{\stackT} = \{h\}$ for some $h \in \Sym{\Omega}$, then
  $(U_{1} \cap \cdots \cap U_{k}) \cap \Iso{\stackS}{\stackT}$ is either empty
  or equal to $\{h\}$ by
  Definition~\ref{defn-approx-iso}\ref{item-approx-true-overestimate}. This
  is decided on line~\ref{line-check-perm}, and the correct answer is
  returned on line~\ref{line-single-solution} or~\ref{line-not-solution}, as
  required.

  In the final case, $|\Approx{\stackS}{\stackT}| \geq 2$, and
  on line~\ref{line-split}
  the procedure
  uses the splitter to
  produce a list of pairs of new labelled digraph stacks
  $[(\stackS_{1}, \stackT_{1}), \ldots, (\stackS_{t}, \stackT_{t})]$
  for some $t \in \N$, and it returns the union of the sets
  $\Call{Search}{\stackS \Vert \stackS_{1}, \stackT \Vert \stackT_{1}},
  \ldots,
  \Call{Search}{\stackS \Vert \stackS_{t}, \stackT \Vert \stackT_{t}}$.
  It suffices to prove that this union is equal to
  $(U_{1} \cap \cdots \cap U_{k}) \cap \Iso{\stackS}{\stackT}$.
  Recall that
  \[\Iso{\stackS}{\stackT}
    =
    \Iso{\stackS \Vert \stackS_{1}}{\stackT \Vert \stackT_{1}}
    \mathop{\dot{\cup}}
    \cdots
    \mathop{\dot{\cup}}
    \Iso{\stackS \Vert \stackS_{m}}{\stackT \Vert \stackT_{m}},
  \]
  by Definition~\ref{defn-splitter}\ref{item-splitter-disjoint}, and that
  $|\Approx{\stackS \Vert \stackS_{i}}{\stackT \Vert \stackT_{i}}| <
  |\Approx{\stackS}{\stackT}| \leq n$ for all $i \in \{1, \ldots, t\}$ by
  Definition~\ref{defn-splitter}\ref{item-splitter-smaller}.
  By the inductive hypothesis, it follows that
  \[
  \Call{Search}{\stackS \Vert \stackS_{i}, \stackT \Vert \stackT_{i}}
    =
  (U_{1} \cap \cdots \cap U_{k})
    \cap \Iso{\stackS \Vert \stackS_{i}}{\stackT \Vert \stackT_{i}}.
  \]
  for all $i \in \{1, \ldots, t\}$, and so
  \begin{align*}
    \Call{Search}{\stackS, \stackT}
      & =
      \Call{Search}{\stackS \Vert \stackS_{1}, \stackT \Vert \stackT_{1}}
        \cup \cdots \cup
      \Call{Search}{\stackS \Vert \stackS_{t}, \stackT \Vert \stackT_{t}} \\
      & \hspace{-1.5cm} =
      \big( U_{1} \cap \cdots \cap U_{k}
        \cap \Iso{\stackS \Vert \stackS_{1}}{\stackT \Vert \stackT_{1}}\!
      \big)
      \cup \cdots \cup
      \big( U_{1} \cap \cdots \cap U_{k}
        \cap \Iso{\stackS \Vert \stackS_{t}}{\stackT \Vert \stackT_{t}}\!
      \big) \\
      & \hspace{-1.5cm} =
      (U_{1} \cap \cdots \cap U_{k})
        \cap
      \big(
        \!
        \Iso{\stackS \Vert \stackS_{1}}{\stackT \Vert \stackT_{1}}
          \cup \cdots \cup
        \Iso{\stackS \Vert \stackS_{t}}{\stackT \Vert \stackT_{t}}
        \!
      \big) \\
      & \hspace{-1.5cm} =
      (U_{1} \cap \cdots \cap U_{k}) \cap \Iso{\stackS}{\stackT}.
      \qedhere
  \end{align*}
\end{proof}

\begin{theorem}\label{thm-alg-is-correct}
  Algorithm~\ref{alg-search} is correct.
  In other words, using the notation of Algorithm~\ref{alg-search},
  \(
  \Call{Search}{\EmptyStack{\Omega}, \EmptyStack{\Omega}}
  =
  (U_{1} \cap \cdots \cap U_{k})
  \).
\end{theorem}

\begin{proof}
  The result follows by setting $\stackS = \stackT = \EmptyStack{\Omega}$ in
  Lemma~\ref{lem-alg-search-is-correct}.
\end{proof}

Note that Algorithm~\ref{alg-search} finds and returns \emph{all} elements
of its output, as proved in Theorem~\ref{thm-alg-is-correct}.
To search for a single element of the desired intersection,
one needs to modify the \Call{Search}{} procedure to terminate on
line~\ref{line-single-solution} as soon as the first solution to the search problem is found.

\begin{definition}\label{defn-search-single}
  Define \Call{SearchSingle}{} to be the procedure obtained from the
  \Call{Search}{} procedure of Algorithm~\ref{alg-search} by completely
  terminating the recursion with a solution to the search problem the first time that it finds
  one (line~\ref{line-single-solution}), and by recursively calling
  \Call{SearchSingle}{} on line~\ref{line-split} rather than \Call{Search}{}.
\end{definition}

\begin{cor}\label{cor-search-single-is-correct}
  Let the notation of Algorithm~\ref{alg-search} and
  Definition~\ref{defn-search-single} hold. Then
  \[
  \Call{SearchSingle}{\EmptyStack{\Omega}, \EmptyStack{\Omega}}
  =
  \begin{cases}
    \varnothing & \text{if}\ U_{1} \cap \cdots \cap U_{k} = \varnothing, \\
    \{g\}       & \text{if}\ U_{1} \cap \cdots \cap U_{k} \neq \varnothing,
  \end{cases}
  \]
  where $g$ is the first element of $U_{1} \cap \cdots \cap U_{k}$ found during the search, in the case that this intersection
  is non-empty.
\end{cor}

\begin{proof}
  Termination and correctness follow almost exactly as in
  Lemmas~\ref{lem-alg-terminates} and~\ref{lem-alg-search-is-correct}.
\end{proof}

\subsection{Searching for a generating set of a subgroup}\label{sec-generators}

In Theorem~\ref{thm-alg-is-correct} we showed that Algorithm~\ref{alg-search}
can be used to find the set of all solutions to a given problem, and in
Corollary~\ref{cor-search-single-is-correct} we showed that a slightly adapted
version of Algorithm~\ref{alg-search} can be used to search for a single
solution in the case that one exists, and to return $\varnothing$ otherwise.

Searching for a single solution is especially useful when one wishes to find an
isomorphism from one combinatorial structure to another, or to prove that none
exists.

It is typically most efficient to compute with a permutation group when it is
specified by a base and strong generating set.  In this section, we show how it
is possible to modify Algorithm~\ref{alg-search} (resulting in
Algorithm~\ref{alg-genset-search}) to search for a base and strong generating
set, in the case that the intersection of the given subsets of $\Sym{\Omega}$ is
a subgroup of $\Sym{\Omega}$.  We also show how the partially-constructed
generating set can be used to prune the search tree as the algorithm progresses.

This algorithm is also useful when searching for an intersection of (right)
cosets.  Suppose that $k,m \in \N$ and that $U_{1}, \ldots, U_{k}$ is a list of right cosets of
subgroups of $\Sym{\Omega}$, and that $(f_{L,1},f_{R,1}), \ldots,
(f_{L,m},f_{R,m})$ is a list of refiners for some of those cosets.  In order to
compactly describe their intersection, we can first use the
\Call{SearchSingle}{} procedure as shown in
Corollary~\ref{cor-search-single-is-correct}: this either shows that $U_{1} \cap
\cdots \cap U_{k}$ is empty, or it produces a representative element $g \in
U_{1} \cap \cdots \cap U_{k}$.
In this latter case, then for all $i \in \{1, \ldots, k\}$, it follows that
$U_{i} g^{-1}$ is a subgroup of $\Sym{\Omega}$, and
Lemma~\ref{lem-refiner-right-coset} implies that $(f_{L,i},f_{L,i})$ is a
refiner for $U_{i} g^{-1}$. Note that we can easily test for membership in
$U_{i} g^{-1}$ if and only if we can easily test for membership in $U_{i}$.
Therefore we may use Algorithm~\ref{alg-genset-search} to search for a
generating set of $U_{1} g^{-1} \cap \cdots \cap U_{k} g^{-1}$.  Since
\[
  U_{1} \cap \cdots \cap U_{k}
    = (U_{1} g^{-1} g) \cap \cdots \cap (U_{k} g^{-1} g)
    = (U_{1} g^{-1} \cap \cdots \cap U_{k} g^{-1}) g,
\]
it follows that this generating set, along with the representative element $g$,
gives a compact description for the intersection $U_{1} \cap \cdots \cap U_{k}$.

The correctness of Algorithm~\ref{alg-genset-search} relies on the
following rather technical lemmas.
Algorithm~\ref{alg-genset-search} applies Lemma~\ref{lem-genset} recursively
to find a base and strong generating set.

\begin{lemma}\label{lem-pre-genset}
  Let the notation of Algorithm~\ref{alg-search} hold, suppose that $U_{1} \cap
  \cdots \cap U_{k}$ is a subgroup of $\Sym{\Omega}$, and let $\stackS \in
  \Stacks{\Omega}$ be arbitrary. Then the following hold:
  \begin{enumerate}[label=\textrm{(\roman*)}]
    \item\label{item-genset-refine-symm}
      $\Call{Refine}{\stackS, \stackS} = (\stackT, \stackT)$ for some
      $\stackT \in \Stacks{\Omega}$.

    \item\label{item-genset-approx-non-empty}
      $\{\idOmega\} \subseteq \Approx{\stackT}{\stackT}$.

    \item\label{item-genset-id}
      If $\Approx{\stackT}{\stackT} = \{\idOmega\}$, then
      $\Call{Search}{\stackS, \stackS} = \{\idOmega\}$.

    \item\label{item-genset-more}
      If $|\Approx{\stackT}{\stackT}| \geq 2$, then
      there exists $t \in \N$ and $\stackS_{1}, \ldots,
      \stackS_{t} \in \Stacks{\Omega}$ such that
      $\Split{\stackT}{\stackT} = [
        (\stackS_{1}, \stackS_{1}),
        (\stackS_{1}, \stackS_{2}),
        \ldots,
        (\stackS_{1}, \stackS_{t})]$.
  \end{enumerate}
\end{lemma}

\begin{proof}
  \begin{enumerate}
    \item[\ref{item-genset-refine-symm}]
      Lemma~\ref{lem-group-refiner-symmetric} implies that every refiner given
      as input to Algorithm~\ref{alg-search} is a pair of equal functions.
      Therefore \Call{Refine}{} maps pairs of equal stacks to pairs of equal
      stacks.

    \item[\ref{item-genset-approx-non-empty}]
      By Lemmas~\ref{lem-alg-refine-is-correct}
      and~\ref{lem-alg-search-is-correct} and
      Definition~\ref{defn-approx-iso}\ref{item-approx-true-overestimate}, it
      follows that
      \begin{align*}
        \{\idOmega\}
          \subseteq
          \Call{Search}{\stackS, \stackS}
          & =
          (U_{1} \cap \cdots \cap U_{k}) \cap \Auto{\stackS}
          \\
          & =
          (U_{1} \cap \cdots \cap U_{k})
          \cap
          \Auto{\Call{Refine}{\stackS, \stackS}}
          \\
          & =
          (U_{1} \cap \cdots \cap U_{k})
          \cap
          \Auto{\stackT}
          \\
          & \subseteq
          \Auto{\stackT} \subseteq \Approx{\stackT}{\stackT}.
      \end{align*}

    \item[\ref{item-genset-id}]
      Here the containments in the proof
      of~\ref{item-genset-approx-non-empty} become equalities, and the result
      follows.

    \item[\ref{item-genset-more}]
      This follows by
      Definition~\ref{defn-splitter}\ref{item-splitter-stab-first}
      and~\ref{item-splitter-invariant}.
      \qedhere
  \end{enumerate}
\end{proof}

\begin{lemma}\label{lem-genset}
  Let the notation of Algorithm~\ref{alg-search} hold,
  assume that $U_{1} \cap \cdots \cap U_{k}$ is a subgroup, let $\stackS
  \in \Stacks{\Omega}$ with $|\Approx{\Call{Refine}{\stackS, \stackS}}{}| \geq
  2$,
  and define $(\stackT, \stackT) = \Call{Refine}{\stackS, \stackS}$ and
  \(\Split{\stackT}{\stackT} = [
    (\stackS_{1}, \stackS_{1}),
    \ldots
    (\stackS_{1}, \stackS_{t})
  ]\)
  as in Lemma~\ref{lem-pre-genset}.
  \begin{enumerate}[label=\textrm{(\roman*)}]
    \item\label{item-genset-stabiliser}
      $\Call{Search}{\stackT \Vert \stackS_{1}, \stackT \Vert \stackS_{1}}$ is
      the stabiliser of $\stackS_{1}$ in $\Call{Search}{\stackS, \stackS}$.

    \item\label{item-genset-cosets}
      For all $i \in \{2, \ldots, t\}$, either the set $\Call{Search}{\stackT
      \Vert \stackS_{1}, \stackT \Vert \stackS_{i}}$ is empty, or it is a
      right coset of $\Call{Search}{\stackT \Vert \stackS_{1}, \stackT
      \Vert \stackS_{1}}$ in $\Call{Search}{\stackS, \stackS}$.

    \item\label{item-genset-description}
      The subgroup $\Call{Search}{\stackS, \stackS}$ is generated by any of its
      subsets that contains a generating set for $\Call{Search}{\stackT \Vert
      \stackS_{1}, \stackT \Vert \stackS_{1}}$ and an element from each
      of the non-empty sets amongst
      \[\Call{Search}{\stackT \Vert \stackS_{1}, \stackT \Vert \stackS_{2}},
      \ldots, \Call{Search}{\stackT \Vert \stackS_{1}, \stackT \Vert
      \stackS_{t}}.\]

    \item\label{item-genset-optimisation}
      Let $\{j_{1}, \ldots, j_{l}\} \subseteq \{2, \ldots, t\}$.
      Suppose that, for each $j \in \{j_{1}, \ldots, j_{l}\}$, either
      we have fixed some element $y_{j} \in \Call{Search}{\stackT \Vert
      \stackS_{1}, \stackT \Vert \stackS_{j}}$,
      or we have determined that
      $\Call{Search}{\stackT \Vert \stackS_{1}, \stackT \Vert \stackS_{j}}$ is
      empty.
      Let $Y_{0}$ be the set of elements $y_{j}$ that we fixed
      when $\Call{Search}{\stackT \Vert \stackS_{1},
      \stackT \Vert \stackS_{j}} \neq \varnothing$,
      let $Y$ be any generating set for $\Call{Search}{\stackT \Vert
      \stackS_{1}, \stackT \Vert \stackS_{1}}$, and define
      \( X = Y \cup Y_{0} \).
      If there exists some
      $g \in \<X\>$,
      $i \in \{2, \ldots, t\} \setminus \{j_{1}, \ldots, j_{l}\}$,
      and $j \in \{j_{1}, \ldots, j_{l}\}$
      such that $\stackS_{i} = \stackS_{j}^{g}$, then
      $\Call{Search}{\stackT \Vert \stackS_{1}, \stackT \Vert \stackS_{i}}
      \subseteq \<X\>$.
  \end{enumerate}
\end{lemma}

\begin{proof}
  \begin{enumerate}
    \item[\ref{item-genset-stabiliser}]
      By Lemmas~\ref{lem-alg-refine-is-correct}
      and~\ref{lem-alg-search-is-correct}, Remark~\ref{rmk-stack-iso-auto},
      and as in the proof of
      Lemma~\ref{lem-pre-genset}\ref{item-genset-approx-non-empty},
      \begin{align*}
        \Call{Search}{\stackT \Vert \stackS_{1}, \stackT \Vert \stackS_{1}}
        & =
        (U_{1} \cap \cdots \cap U_{k})
          \cap \Auto{\stackT \Vert \stackS_{1}}
          \\
        & =
        (U_{1} \cap \cdots \cap U_{k})
          \cap \Auto{\stackT} \cap \Auto{\stackS_{1}}
          \\
        & =
        (U_{1} \cap \cdots \cap U_{k})
          \cap \Iso{\Call{Refine}{\stackS, \stackS}}{} \cap \Auto{\stackS_{1}}
          \\
        & =
        \Call{Search}{\stackS, \stackS} \cap \Auto{\stackS_{1}}.
      \end{align*}

    \item[\ref{item-genset-cosets}]
      Let $i \in \{2,\ldots,t\}$.
      By inspecting the \Call{Search}{} procedure, it is clear that
      \begin{equation*}
        \Call{Search}{\stackS, \stackS} =
            \Call{Search}{\stackT \Vert \stackS_{1}, \stackT \Vert \stackS_{1}}
            \cup \cdots \cup
            \Call{Search}{\stackT \Vert \stackS_{1}, \stackT \Vert \stackS_{t}}.
      \end{equation*}
      Hence $\Call{Search}{\stackT \Vert \stackS_{1}, \stackT \Vert
      \stackS_{i}} \subseteq \Call{Search}{\stackS, \stackS}$.
      Suppose there exists some element $x \in
      \Call{Search}{\stackT \Vert \stackS_{1}, \stackT \Vert \stackS_{i}}$.  By
      Lemma~\ref{lem-alg-search-is-correct} and the assumption that $U_{1} \cap
      \cdots \cap U_{k}$ is a subgroup of $\Sym{\Omega}$ containing $x$, it
      follows that
      \begin{align*}
        \Call{Search}{\stackT \Vert \stackS_{1}, \stackT \Vert \stackS_{1}}
          \cdot x
        & =
        \big( (U_{1} \cap \cdots \cap U_{k}) \cap
              \Auto{\stackT \Vert \stackS_{1}}\!
        \big) \cdot x
        \\
        & =
        (U_{1} \cap \cdots \cap U_{k}) \cap
        \Iso{\stackT \Vert \stackS_{1}}{\stackT \Vert \stackS_{i}}
        \\
        & =
        \Call{Search}{\stackT \Vert \stackS_{1}, \stackT \Vert \stackS_{i}}.
      \end{align*}

    \item[\ref{item-genset-description}]
      This follows from~\ref{item-genset-cosets}.

    \item[\ref{item-genset-optimisation}]
      Let $g \in \<X\>$, $i \in \{2, \ldots, t\} \setminus \{j_{1}, \ldots,
      j_{l}\}$, and $j \in \{j_{1}, \ldots, j_{l}\}$ be such that $\stackS_{i} =
      \stackS_{j}^{g}$.
      Note that $X \subseteq \Call{Search}{\stackS, \stackS}$
      by~\ref{item-genset-stabiliser} and~\ref{item-genset-cosets},
      and so in particular, $g \in (U_{1} \cap \cdots \cap
      U_{k}) \cap \Auto{\stackT}$ by Lemmas~\ref{lem-alg-refine-is-correct}
      and~\ref{lem-alg-search-is-correct}.
      Therefore
      $g \in \Iso{\stackT \Vert \stackS_{j}}{\stackT \Vert \stackS_{i}}$,
      and so $\Iso{\stackT \Vert \stackS_{1}}{\stackT
      \Vert \stackS_{j}} \cdot g = \Iso{\stackT \Vert \stackS_{1}}{\stackT \Vert
      \stackS_{i}}$.
      It then follows similarly as in the end of the proof
      of~\ref{item-genset-cosets} that
      \[
      {\Call{Search}{\stackT \Vert \stackS_{1}, \stackT \Vert \stackS_{j}}}
       \cdot g
       =
      \Call{Search}{\stackT \Vert \stackS_{1}, \stackT \Vert \stackS_{i}}.
      \]
      Thus, if $\Call{Search}{\stackT \Vert \stackS_{1}, \stackT \Vert
      \stackS_{j}} \neq \varnothing$,
      and so $y_{j} \in X \cap \Call{Search}{\stackT \Vert \stackS_{1},
      \stackT \Vert \stackS_{j}}$, then
      \begin{align*}
        \Call{Search}{\stackT \Vert \stackS_{1}, \stackT \Vert \stackS_{i}}
        & =
        {\Call{Search}{\stackT \Vert \stackS_{1}, \stackT \Vert \stackS_{j}}}
         \cdot g
        \\
        & =
        \big(
        {\Call{Search}{\stackT \Vert \stackS_{1}, \stackT \Vert \stackS_{1}}}
         \cdot y_{j}
        \big)
         \cdot g
        \\
        & = \<Y\> \cdot (y_{j} g) \subseteq \<X\>.
        \qedhere
      \end{align*}
  \end{enumerate}
\end{proof}

\begin{algorithm}[!ht]
  \caption{Search for a base and strong generating set of a
           subgroup of $\Sym{\Omega}$.}\label{alg-genset-search}

  \begin{algorithmic}[1]

    \item[\textbf{Input:}]
      as in Algorithm~\ref{alg-search}, plus the assumption that
      $U_{1} \cap \cdots \cap U_{k}$ is a subgroup.
    \item[\textbf{Output:}]
      a base and strong generating set for the subgroup $U_{1} \cap \cdots \cap
      U_{k}$.

    \vspace{2mm}

    \State{$\textsc{Base} \gets []$}
    \Comment{The base is initialised as an empty list.}

    \State{\Return{\big\{\Call{SearchGens}{$\EmptyStack{\Omega}$}},\
                     \textsc{Base}\big\}
    }

    \Procedure{SearchGens}{$\stackS$}

    \State{\((\stackT, \stackT) \gets \Call{Refine}{\stackS, \stackS}\)}
    \Comment{Refine the given stacks.}

    \Case{$\Approx{\stackT}{} = \{\idOmega\}$}\label{line-trivial}
      \State{\Return{$\{\idOmega\}$}}
      \Comment{Lemma~\ref{lem-pre-genset}\ref{item-genset-id}}
    \EndCase{}

    \Case{$|\Approx{\stackT}{}| \geq 2$}

      \State{$[(\stackS_{1}, \stackS_{1}), \ldots, (\stackS_{1}, \stackS_{t})]
             \gets \Call{Split}{\stackT, \stackT}$}
      \Comment{Lemma~\ref{lem-pre-genset}\ref{item-genset-more}
              }\label{line-genset-split}

      \State{$\textsc{Base} \gets \textsc{Base} \mathop{\Vert} [\stackS_{1}]$}
      \Comment{Add the stack $\stackS_{1}$ to the base.}

      \State{$X \gets \Call{SearchGens}{\stackT \Vert \stackS_{1}}$}
      \Comment{Recursively find generators for a
      subgroup.}\label{line-gens-recurse}

      \For{$i \in \{2, \ldots, t\}$}

        \If{$\stackS_{i} \not\in \stackS_{j}^{\<X\>}$ for any
            $j \in \{1, \ldots, i - 1\}$}
          \Comment{Pruning;
                   Lemma~\ref{lem-genset}\ref{item-genset-optimisation}}

          \State{$X \gets X\cup \Call{SearchSingle}{\stackT \Vert \stackS_{1},
                                                    \stackT \Vert \stackS_{i}}$}
          \Comment{Search for a coset rep.}
        \EndIf{}
      \EndFor{}

      \State{\Return{$X$}}\label{line-genset-return}
    \EndCase{}

    \EndProcedure{}

    \Procedure{Refine}{$\stackS, \stackT$}
    \Comment{\emph{The \Call{Refine}{} procedure from
    Algorithm~\ref{alg-search}.}}
    \EndProcedure{}

    \Procedure{SearchSingle}{$\stackS, \stackT$}
    \Comment{\emph{The procedure from
    Definition~\ref{defn-search-single}.}}
    \EndProcedure{}

  \end{algorithmic}
\end{algorithm}

Let the notation of Algorithms~\ref{alg-search} and~\ref{alg-genset-search}
hold. We briefly explain
how the \Call{SearchGens}{} procedure has been obtained from the \Call{Search}{}
procedure of Algorithm~\ref{alg-search}.  Given the validity of these
modifications, the correctness of the \Call{SearchGens}{} procedure then follows
from the correctness of the \Call{Search}{} procedure
(Lemma~\ref{lem-alg-search-is-correct}).

Lemma~\ref{lem-pre-genset}\ref{item-genset-approx-non-empty} implies that the
condition on line~\ref{line-case-empty} of the \Call{Search}{} procedure is
never satisfied when $U_{1} \cap \cdots \cap U_{k}$ is a subgroup and the stacks
in question are equal, and so it is unnecessary to include this case in
\Call{SearchGens}{}.  From the same result, it also follows that the condition
on line~\ref{line-case-one-candidate} of the \Call{Search}{} procedure can be
restated as on line~\ref{line-trivial} of \Call{SearchGens}{}, since
$|\Approx{\stackT}{}| = 1$ if and only if $\Approx{\stackT}{} = \{\idOmega\}$.
Note that $\idOmega$ is contained in $U_{1} \cap \cdots \cap
U_{k}$ by assumption and in $\Auto{\stackT}$ by definition, which explains the
remaining simplification of this case.  Finally, it follows from Lemmas~\ref{lem-pre-genset} and~\ref{lem-genset} and the
correctness of the \Call{SearchSingle}{} procedure
(Corollary~\ref{cor-search-single-is-correct}) that line~\ref{line-split} of
\Call{Search}{} can be replaced by
lines~\ref{line-genset-split}--\ref{line-genset-return} in \Call{SearchGens}{}.
Thus we have proved the following lemma:

\begin{lemma}\label{lem-alg-searchgens-is-correct}
  Let $\stackS \in \Stacks{\Omega}$ and let the notation of
  Algorithm~\ref{alg-genset-search} hold.  Then
  $\Call{SearchGens}{\stackS}$
  is a generating set for $U_{1} \cap \cdots \cap U_{k} \cap \Auto{\stackS} =
  \Call{Search}{\stackS, \stackS}$.
\end{lemma}

That Algorithm~\ref{alg-genset-search} terminates given any valid input can be
proved in a very similar way
to Lemma~\ref{lem-alg-terminates}.  Thus we present the
main result of this section.

\begin{theorem}\label{thm-alg-gens-is-correct}
  In the notation of Algorithm~\ref{alg-genset-search},
  $\Call{SearchGens}{\EmptyStack{\Omega}}$ is a strong generating set for $U_{1}
  \cap \cdots \cap U_{k}$ relative to the base \textsc{Base}. In other words,
  Algorithm~\ref{alg-genset-search} returns a base and strong generating set for
  its input.
\end{theorem}

\begin{proof}
  Given Lemma~\ref{lem-alg-searchgens-is-correct},
  $\Call{SearchGens}{\EmptyStack{\Omega}}$ is a generating set for the
  subgroup $U_{1} \cap \cdots \cap U_{k}$, so it remains to show that
  \textsc{Base} is a base, relative to which the generating set is strong.

  Firstly, if $\Approx{\Call{Refine}{\EmptyStack{\Omega},
  \EmptyStack{\Omega}}}{} = \{\idOmega\}$, then
  $\Call{SearchGens}{\EmptyStack{\Omega}}$ returns
  the generating set $\{\idOmega\}$ without modifying the variable
  \textsc{Base}, which is therefore still an empty list. This is a base and
  strong generating set for the trivial subgroup of $\Sym{\Omega}$, and so this
  case is complete.

  Otherwise, if
  $|\Approx{\Call{Refine}{\EmptyStack{\Omega}, \EmptyStack{\Omega}}}{}| \geq 2$,
  then we define $\stackT_{0} = \stackS_{1,0} = \EmptyStack{\Omega}$,
  and iteratively define $\stackT_{i+1}$ and $\stackS_{1,i+1}$ for
  $i = 0, 1, \ldots$, so long as
  \(
  |\Approx{\Call{Refine}{\stackT_{i} \Vert \stackS_{1,i},
                         \stackT_{i} \Vert \stackS_{1,i}}}{}| \geq 2
  \),
  via
  $(\stackT_{i + 1}, \stackT_{i + 1}) =
  \Call{Refine}{\stackT_{i} \Vert \stackS_{1,i},
                \stackT_{i} \Vert \stackS_{1,i}}$
  and
  $\Call{Split}{\stackT_{i+1}, \stackT_{i+1}} =
  [(\stackS_{1,i+1},\stackS_{1,i+1}), \ldots ]$.
  Thus $\stackT_{i + 1}$ is the stack obtained by refining $\stackT_{i} \Vert
  \stackS_{1,i}$, and $\stackS_{1,i+1}$ is the left-hand stack obtained by
  splitting $\stackT_{i+1}$.
  Let $r \in \N$ be the maximum value for which we defined $T_{r}$,
  which means that
  \(
  \Approx{\Call{Refine}{\stackT_{r} \Vert \stackS_{1,r},
                        \stackT_{r} \Vert \stackS_{1,r}}}{} = \{\idOmega\}
  \).

  It is straightforward to see that the sequence of stacks $\stackT_{0} \Vert
  \stackS_{1,0}, \stackT_{1} \Vert \stackS_{1,1}, \ldots, \stackT_{r} \Vert
  \stackS_{1,r}$ is exactly the sequence of stacks upon which the recursive
  procedure \Call{SearchGens}{} is called during the execution of
  Algorithm~\ref{alg-genset-search}, on line~\ref{line-gens-recurse}.
  Therefore
  \begin{align*}
    \{\idOmega\} = \Call{SearchGens}{\stackT_{r} \Vert \stackS_{1,r}}
              \subseteq
              \cdots
              & \subseteq
              \Call{SearchGens}{\stackT_{1} \Vert \stackS_{1,1}} \\
              & \subseteq
              \Call{SearchGens}{\stackT_{0} \Vert \stackS_{1,0}} \\
              & = \Call{SearchGens}{\EmptyStack{\Omega}},\ \text{and so}\\
    \{\idOmega\} = \< \Call{SearchGens}{\stackT_{r} \Vert \stackS_{1,r}} \>
              \leq
              \cdots
              & \leq
              \< \Call{SearchGens}{\stackT_{1} \Vert \stackS_{1,1}} \>
              \\
              & \leq
              \< \Call{SearchGens}{\stackT_{0} \Vert \stackS_{1,0}} \>
              = U_{1} \cap \cdots \cap U_{k}.
  \end{align*}
  Lemma~\ref{lem-genset}\ref{item-genset-stabiliser} and
  Lemma~\ref{lem-alg-searchgens-is-correct} imply that
  $\< \Call{SearchGens}{\stackT_{i} \Vert \stackS_{1,i}} \>$ is
  the stabiliser of $\stackS_{1,i}$ in $\< \Call{SearchGens}{\stackT_{i-1} \Vert
  \stackS_{1,i-1}} \>$
  for each $i \in \{1, \ldots, r\}$.
  In other words, Algorithm~\ref{alg-genset-search}
  constructs a stabiliser chain
  for $U_{1} \cap \cdots \cap U_{k}$ relative to
  \textsc{Base}. This proves the result.
\end{proof}

Typically, a base for a subgroup of $\Sym{\Omega}$ is assumed to be a list of
points in $\Omega$ itself, as opposed to a list of arbitrary objects upon which
the group acts.  This latter more general definition is the one that we have
used so far in this paper.  In order to use Algorithm~\ref{alg-genset-search} to
obtain a base consisting of points in $\Omega$, one can use the splitter from
Definition~\ref{defn-point-splitter}: using the notation of this definition and
an arbitrary $\alpha \in \Omega$, a permutation stabilises the stack
$[\Gamma_{\alpha}]$ if and only if it stabilises the point $\alpha$.  Therefore,
a generating set for a subgroup of $\Sym{\Omega}$ is strong with respect to the
list of stacks $[[\Gamma_{\alpha_{1}}], \ldots,
[\Gamma_{\alpha_{r}}]]$ if and only if it is strong with respect to
$[\alpha_{1}, \ldots, \alpha_{r}]$.

\subsection{Computing with a fixed sequence of left-hand
stacks}\label{sec-r-base}

In this section, we discuss a consequence of the setup of our definitions and
algorithms, which enables a significant performance optimisation, and through
which the usefulness and practicality of the refiners from
Section~\ref{sec-refiners-via-stack} becomes apparent.
This idea was inspired by, and is closely related to, the $\mathfrak{R}$-base
technique of Jeffrey Leon~\cite[Section~6]{leon1991} for partition backtrack
search, although we present the idea quite differently.

Roughly speaking, we observe that any time Algorithm~\ref{alg-search}
or~\ref{alg-genset-search} is executed to solve a problem, then the left-hand
stack of the ever-present pair is modified with the same sequence of changes in
every branch of the search. In other words, every branch of search has the same
sequence of left-hand stacks, up until the point that the branch ends (different
branches can have different lengths). This means that any entry in this fixed
sequence of left-hand stacks only ever needs to be computed once, and then
stored and recalled for later use. Furthermore, these stacks can give rise to
the fixed stacks and lists of points required by the refiners of
Section~\ref{sec-refiners-via-stack}.

This behaviour emerges, in essence,
because a refiner is a pair of functions of one variable,
rather than a single function of two variables
(Definition~\ref{defn-refiner});
because a non-empty value of an isomorphism approximator is a coset
of a subgroup, where the subgroup depends only on the given left-hand stack
(Definition~\ref{defn-approx-iso}\ref{item-approx-right-coset-of-aut});
and because the left-hand stacks produced by a splitter depend only on
the left-hand stack that it is given
(Definition~\ref{defn-splitter}\ref{item-splitter-invariant}).

\subsubsection{A performance improvement by using a fixed sequence of
left-hand stacks}

\begin{lemma}\label{lem-refine-fixed-left}
  Let the notation of Algorithm~\ref{alg-search} hold, and let $\stackS \in
  \Stacks{\Omega}$.  Then there exist $n, r \in \N$ and a fixed sequence
  of at most $mn$ modifications to $\stackS$ such that, for all $\stackV \in
  \Stacks{\Omega}$, either:
  \begin{enumerate}[label=\textrm{(\roman*)}]
    \item
      $\Call{Refine}{\stackS, \stackV}$ executes line~\ref{line-apply-refiner}
      of Algorithm~\ref{alg-search}
      some $i \in \{0, \ldots, mn\}$ times, performing the first $i$
      modifications to $\stackS$ in turn, and
      $\Approx{\Call{Refine}{\stackS, \stackV}}{} = \varnothing$; or
    \item
      $\Call{Refine}{\stackS, \stackV}$ executes line~\ref{line-apply-refiner}
      of Algorithm~\ref{alg-search}
      exactly $mn$ times, performing all $mn$ modifications to $\stackS$ in
      turn, and $|\Approx{\Call{Refine}{\stackS, \stackV}}{}| = r$.
  \end{enumerate}
\end{lemma}

\begin{proof}
  The \Call{Refine}{} procedure from Algorithm~\ref{alg-search} modifies its
  pair of stacks only on line~\ref{line-apply-refiner}, and the number of
  modifications that it makes is equal to the number of times that
  line~\ref{line-apply-refiner} is executed.
  If the pair of stacks given to the \Call{Refine}{} procedure is $(\stackS,
  \stackT)$, for instance, then the left-hand stack could be modified with the
  sequence of moves:
  \[
    \stackS
    \mathrel{\longrightarrow}
    \stackS \Vert f_{L,1}(\stackS)
    \mathrel{\longrightarrow}
    \big(\stackS \Vert f_{L,1}(\stackS)\big) \Vert
    f_{L,2} \big( \stackS \Vert f_{L,1}(\stackS) \big)
    \mathrel{\longrightarrow}
    \cdots
  \]
  up to some point, and the right-hand stack would be modified in the
  corresponding way:
  \[
    \stackT
    \mathrel{\longrightarrow}
    \stackT \Vert f_{R,1}(\stackT)
    \mathrel{\longrightarrow}
    \big(\stackT \Vert f_{R,1}(\stackT)\big) \Vert
    f_{R,2} \big( \stackT \Vert f_{R,1}(\stackT) \big)
    \mathrel{\longrightarrow}
    \cdots
  \]

  If the $m$-fold \textbf{for} loop on lines~\ref{line-refine-loop}
  and~\ref{line-apply-refiner} is interrupted because the condition $|\stackS| =
  |\stackT|$ fails to be satisfied at some point, then
  $\Approx{\stackS}{\stackT} = \varnothing$ by
  Definition~\ref{defn-approx-iso}\ref{item-approx-different-lengths}.  In this
  case, neither the condition on line~\ref{line-refine-not-smaller} nor the
  condition on line~\ref{line-refine-empty} is satisfied, and so the procedure
  returns its pair of stacks $(\stackS, \stackT)$ on
  line~\ref{line-refine-return-empty}, without further modification.

  Otherwise, the \Call{Refine}{} procedure returns after completing some number
  (perhaps zero) of repetitions of the full \textbf{for} loop from
  lines~\ref{line-refine-loop} and~\ref{line-apply-refiner}.  The procedure
  returns because either the condition on line~\ref{line-refine-not-smaller} is
  satisfied, or the condition on line~\ref{line-refine-empty} is not.

  Let $\stackV_{1} \in \Stacks{\Omega}$, and suppose that the \Call{Refine}{}
  procedure, when given the stacks $(\stackS, \stackV_{1})$, has completed its
  $n$\textsuperscript{th} full iteration of the \textbf{for} loop for some $n
  \in \N$, and suppose that the condition on
  line~\ref{line-refine-not-smaller} is satisfied.
  Let
  $(\stackS', \stackV_{1}')$ denote the pair of stacks immediately
  before the $n$\textsuperscript{th} iteration of the \textbf{for} loop,
  and let
  $(\stackS^{*}, \stackV_{1}^{*})$ denote the pair of stacks immediately
  after it.
  It follows that $0 < r \coloneqq |\Approx{\stackS'}{\stackV_{1}'}| \leq
  |\Approx{\stackS^{*}}{\stackV_{1}^{*}}|$, and so $\Call{Refine}{\stackS,
  \stackV_{1}} = (\stackS', \stackV_{1}')$.

  Next, let $\stackV_{2} \in \Stacks{\Omega}$ and suppose that the
  \Call{Refine}{} procedure, when given the stacks $(\stackS, \stackV_{2})$, has
  also completed its $n$\textsuperscript{th} full iteration of the \textbf{for}
  loop.  By the earlier arguments, the procedure has modified the left-hand
  stack with the exact same sequence of modifications as before, and so there
  exist stacks $\stackV_{2}^{*}, \stackV_{2}' \in \Stacks{\Omega}$ such that
  $(\stackS', \stackV_{2}')$ is the pair of stacks immediately before the
  $n$\textsuperscript{th} iteration of the \textbf{for} loop, and $(\stackS^{*},
  \stackV_{2}^{*})$ is the pair of stacks immediately after it.
  If $\Approx{\stackS^{*}}{\stackV_{2}^{*}} = \varnothing$, then the procedure
  returns on line~\ref{line-refine-return-empty}. Otherwise,
  Definition~\ref{defn-approx-iso}\ref{item-approx-right-coset-of-aut} implies
  that
  \begin{multline*}
    |\Approx{\stackS^{*}}{\stackV_{2}^{*}}| =
    |\Approx{\stackS^{*}}{}| =
    |\Approx{\stackS^{*}}{\stackV_{1}^{*}}|,
    \ \text{and that}\\
    |\Approx{\stackS'}{\stackV_{2}'}| =
    |\Approx{\stackS'}{}| =
    |\Approx{\stackS'}{\stackV_{1}'}|.
  \end{multline*}
  In particular $0 < |\Approx{\stackS'}{\stackV_{2}'}| \leq
  |\Approx{\stackS^{*}}{\stackV_{2}^{*}}|$. Thus the condition on
  line~\ref{line-refine-not-smaller} is satisfied in this case,
  and $\Call{Refine}{\stackS, V_{2}} = (\stackS', \stackV_{2}')$ with
  $|\Approx{\stackS'}{\stackV_{2}'}| = r$.
\end{proof}

\begin{lemma}\label{lem-search-fixed-left}
  Let the notation of Algorithm~\ref{alg-search} hold and let $\stackS \in
  \Stacks{\Omega}$.  Then there exists some $r \in \N \setminus \{1\}$ and fixed
  stacks $\stackS', \stackS'_{1} \in \Stacks{\Omega}$ such that, for all
  $\stackV \in \Stacks{\Omega}$, either:
  \begin{enumerate}[label=\textrm{(\roman*)}]
    \item\label{item-search-fixed-left-1}
      $|\Approx{\Call{Refine}{\stackS, \stackV}}{}| \leq 1$, and
      $\Call{Search}{\stackS, \stackV}$ backtracks on either
      line~\ref{line-no-isos},~\ref{line-single-solution},
      or~\ref{line-not-solution}
      of Algorithm~\ref{alg-search}; or
    \item\label{item-search-fixed-left-2}
      $|\Approx{\Call{Refine}{\stackS, \stackV}}{}| = r$,
      $\Call{Refine}{\stackS, \stackV} = (\stackS', \stackV')$ for some
      labelled digraph stack $\stackV'$, and $\Call{Search}{\stackS, \stackV}$
      recursively calls \Call{Search}{} on line~\ref{line-split} of Algorithm~\ref{alg-search},
      always with first argument $\stackS' \Vert \stackS'_{1}$.
  \end{enumerate}
\end{lemma}

\begin{proof}
  When studying
  the \Call{Search}{} procedure from Algorithm~\ref{alg-search}, it becomes clear that~\ref{item-search-fixed-left-1} is a possibility.  We see
  that~\ref{item-search-fixed-left-2} is the remaining possibility, by noticing that
  $r$ and the fixed stack $\stackS'$ exist by Lemma~\ref{lem-refine-fixed-left}
  (in the notation of Lemma~\ref{lem-refine-fixed-left}, $\stackS'$ is
  obtained by applying the first $m(n - 1)$ modifications to $\stackS$), and
  the stack $\stackS_{1}'$ is fixed by
  Definition~\ref{defn-splitter}\ref{item-splitter-invariant}.
\end{proof}

\begin{cor}\label{cor-fixed-stack}
  In each branch of search, Algorithm~\ref{alg-search} modifies its left-hand
  stack with the same sequence of moves, until the branch ends and the
  algorithm backtracks.
\end{cor}

\begin{proof}
  A branch consists of a sequence of recursive calls to \Call{Search}{},
  beginning with the call on line~\ref{line-start-search}, and ending at some
  depth of the recursion by backtracking on one of
  lines~\ref{line-no-isos},~\ref{line-single-solution},
  or~\ref{line-not-solution}.
  Lemmas~\ref{lem-refine-fixed-left} and~\ref{lem-search-fixed-left} show that,
  if the \Call{Search}{} procedure at some depth of the recursion is given the
  left-hand stack $\stackS$ and any right-hand stack, then the sequence of
  modifications made to the left-hand stack (until either backtracking or
  recursion happens) is independent of the right-hand stack. Furthermore,
  Lemma~\ref{lem-search-fixed-left} shows that at any given depth of recursion,
  the \Call{Search}{} procedure is recursively called with the same left-hand
  stack. The result follows.
\end{proof}

By very similar arguments, Algorithm~\ref{alg-genset-search} also
modifies its left-hand stack with the exact same sequence of moves in each
branch of search, until
the algorithm backtracks.

Corollary~\ref{cor-fixed-stack} shows that we may store the modifications to
the left-hand stacks the first time that they are made, and then we can simply
recall a result whenever it is needed again.  This means that on most
occasions, when applying a refiner, we need only compute the value of the
right-hand stack under the refiner, since we can simply look up the result for
the left-hand stack. This leads to a performance speedup of roughly 50\%.

\subsubsection{Constructing and applying a refiner via the fixed sequence of
left-hand stacks}\label{sec-details-refiner-fixed-stack}

We discuss how to use Lemma~\ref{lem-fixed-stack} to build a
refiner for a group $G$ via the fixed sequence of left-hand stacks.
Using the notation
of this lemma, in order to define the function $f$ such that $(f, f)$ is a
refiner for $G$, for each $i \in \N_{0}$ we must create labelled digraph
stacks $\stackV_{i}$, and lists $F_{i}$ that consist of points in $\Omega$.
We start with $\stackV_{i}$ and $F_{i}$ being undefined for all $i \in \N_{0}$,
and we define $\stackV_{i}$ and $F_{i}$ on-demand as we apply the refiner during
the execution of Algorithm~\ref{alg-search} or~\ref{alg-genset-search}.

Let $\stackS, \stackT \in \Stacks{\Omega}$ have equal lengths.
We apply the refiner as follows.

If $\stackV_{|\stackS|}$ and $F_{|\stackS|}$ have already been defined, then we
can look up and return the stored value of $f(\stackS)$ (if it is already
known), and we can compute $f(\stackT)$ (and $f(\stackS)$, if it is not already
known) as specified in Lemma~\ref{lem-fixed-stack}.  Since we compute with the
same sequence of left-hand stacks in every branch of search, as discussed above,
then it is likely that $f(\stackS)$  has already been computed.

Otherwise, if $\stackV_{|\stackS|}$ and $F_{|\stackS|}$ are still undefined,
then we define $F_{|\stackS|} = \Fixed{\stackS}$, and we
define $\stackV_{|\stackS|}$
to be some arbitrary labelled digraph stack that is preserved by
$G_{\Fixed{\stackS}}$, the stabiliser of $\Fixed{\stackS}$ in $G$.
For example, if we want our refiner to exploit the orbit data of
$G_{\Fixed{\stackS}}$, then we could define $V_{|\stackS|}$ to be the stack
$[\Gamma]$, where $\Gamma$ is a labelled digraph on $\Omega$ without arcs in
which two vertices share a label if and only if they belong to the same orbit of
$G_{\Fixed{\stackS}}$ on $\Omega$.
Alternatively,
$V_{|\stackS|}$ could be a list of all, or some, of the orbital graphs of
$G_{\Fixed{\stackS}}$ on $\Omega$, represented as labelled digraphs.
Given $V_{|\stackS|}$ and $F_{|\stackS|}$, then $f(\stackS) = V_{|\stackS|}$
(and we store its value), and we compute $f(\stackT)$ as in
Lemma~\ref{lem-fixed-stack}.

In order to construct a refiner for the coset $G h$, for some $h \in
\Sym{\Omega}$, we construct the function $f$ as above;
the corresponding refiner is $(f, g)$,
where $g(\stackS) = f(\stackS^{h^{-1}}){}^{h}$ for all $\stackS \in
\Stacks{\Omega}$
(see Lemma~\ref{lem-refiner-right-coset}).

In essence, this technique lets us use the fixed sequence of left-hand stacks to
arbitrarily order objects like orbits and orbital graphs, for use in refiners.
This addresses the problem discussed in Example~\ref{ex-refiner-fixed-stack}, and
thus can lead to more effective refinement.

\section{Experiments}\label{sec-experiments}

In this section, we provide experimental data comparing the behaviour of our
algorithms against partition backtrack, in order to
highlight the potential of our techniques.

In particular, we repeat the experiments of~\cite[Section~6]{newrefiners} (by
the first three authors of the present paper), which showed how orbital graphs
can significantly improve the partition backtrack algorithm when computing
various kinds of set stabilisers and subgroup intersections.
We also investigate some additional challenging problems.

It would not be useful to investigate classes of problems where partition
backtrack already performs very well, and for which there is no necessary or
realistic scope for further improvement.
In addition, there are other classes of problems, such as those that involve
searching for highly-transitive groups, where we would expect all
techniques (including ours) to perform badly, and so it also makes sense to
avoid such problems.
Instead, we have chosen to investigate problems
that are interesting and important in their own right, including ones
that we expect to be hard for many search techniques.

At the time of writing, we have focused on the mathematical theory of our
algorithms, rather than on the speed of our implementations.  Because of this,
we would expect our current implementations to perhaps unfairly struggle in time
comparisons against implementations of partition backtrack, and so such
comparisons would be inappropriate at this point.

Therefore, whereas the experiments in~\cite[Section~6]{newrefiners} analyse the
\emph{time} required by an algorithm to solve a problem, here we analyse the
\emph{size of the search} required by the algorithm to solve it.
We define a \emph{search node} of a search to be an instance of the main
searching procedure being called recursively during its execution;
the size of a search is then its number of search nodes.
If an algorithm requires \(0\) search nodes to solve a problem, then
this means that the algorithm solved the problem without entering recursion.
For the algorithms that we compare, this can only be achieved with a search problem
that has either no solutions, or exactly one.

In general, a backtrack search algorithm spends effort at each
node to prune the search tree and organise the search.  The size
of a search is not obviously related to the time taken to complete it, since a
smaller search typically comes at the cost of spending more effort at each node.
However, the computations at each node of our algorithms are largely
digraph-based, and the very high performance of digraph-based computer programs
such as \bliss~\cite{bliss} and \nauty~\cite{practical2} suggests that, in
practice, such computations could potentially be cheap.
Therefore, with further development, we have reason to believe that,
for problems where our techniques require significantly smaller searches,
the increased time spent at each node could be out-weighed by the smaller number
of nodes in total, giving faster searches.

For the problems that we investigate in
Sections~\ref{sec-grid-groups}--\ref{sec-intransitive-intersection}, we compare
the following techniques:
\begin{enumerate}[label=\textrm{(\roman*)}]
  \item\textsc{Leon}:
    Standard partition backtrack search,
    as described by Jeffrey Leon~\cite{leon1997, leon1991}.

  \item\textsc{Orbital}:
    Partition backtrack search with orbital graph refiners,
    as described in~\cite{newrefiners}.

  \item\textsc{Strong}:
    Backtrack search with labelled digraphs,
    using the
    isomorphism and fixed-point approximators from
    Definition~\ref{defn-strong-approx}
    and the splitter from Definition~\ref{defn-point-splitter}.

  \item\textsc{Full}:
    Backtrack search with labelled digraphs,
    using the isomorphism and fixed-point approximators
    from Definition~\ref{defn-nauty-approx}
    and the splitter from Definition~\ref{defn-point-splitter}.
\end{enumerate}

The \textsc{Leon} technique is roughly the same as backtrack search with
labelled digraphs, where the labelled digraphs in the stack have no arcs.
The \textsc{Orbital} technique is essentially the same as backtrack search with
labelled digraphs, using the `weak equitable labelling' isomorphism and
fixed-point approximators from Definition~\ref{defn-weak-approx}.
The \textsc{Strong} technique considers all labelled digraphs in the stack
simultaneously to make its approximations, while the \textsc{Full} technique,
which completely calculates rather than just approximates, is in principle the most
expensive of the four methods.

We performed our experiments using the \textsc{GraphBacktracking}~\cite{GAPpkg}
and \textsc{BacktrackKit}~\cite{BTpkg} packages for \textsc{GAP}~\cite{GAP4}.
\textsc{BacktrackKit} provides a simple implementation of the algorithms in
\cite{newrefiners,leon1991}, and provides a base for
\textsc{GraphBacktracking}. Where we reproduce experiments from
\cite{newrefiners}, we ensure that we find the same sized searches.

\subsection{Set stabilisers and partition stabilisers in grid
groups}\label{sec-grid-groups}

We first explore the behaviour of our techniques on stabiliser problems in grid
groups. This setting was previously considered
in~\cite[Section~6.1]{newrefiners}, and as mentioned there, these kinds of
problems arise in numerous real-world situations.

\begin{definition}[\mbox{Grid
  group~\cite[Definition~36]{newrefiners}}]\label{def-gridgroup}
  Let \(n \in \N\) and \(\Omega = \{1, \ldots, n\}\).  The direct product
  \(\Sym{\Omega} \times \Sym{\Omega}\) acts faithfully on the Cartesian product
  \(\Omega \times \Omega\) via
  \({(\alpha, \beta)}^{(g, h)} = (\alpha^{g}, \beta^{h})\)
  for all \(\alpha, \beta \in \Omega\) and \(g, h \in \Sym{\Omega}\).
  The \emph{\(n \times n\) grid group} is the image of the embedding of
  \(\Sym{\Omega} \times \Sym{\Omega}\) into \(\Sym{\Omega \times \Omega}\)
  defined by this action.
\end{definition}

Let \(n \in \N\) and \(\Omega = \n\), and let \(G \leq \Sym{\Omega
\times \Omega}\) be the \(n \times n\) grid group.  If we consider \(\Omega
\times \Omega\) to be an \(n \times n\) grid, where the sets of the form
\(\set{(\alpha, \beta)}{\beta \in \Omega}\) and \(\set{(\beta, \alpha)}{\beta
\in \Omega}\) for each \(\alpha \in \Omega\) are the rows and columns of the
grid, respectively, then \(G\) is the subgroup of \(\Sym{\Omega \times \Omega}\)
that preserves the set of rows and the set of columns of the grid.
Note that the $n \times n$ grid group is $2$-closed, which  means that it is
well suited to the techniques of this paper.

The experiments in~\cite{newrefiners} solved two kinds of set stabiliser
problems in grid groups.  We repeat these problems here,
along with an unordered partition stabiliser problem:

\begin{enumerate}[label=\textrm{(\roman*)}]
  \item\label{exp-set-stab-grid-1}
    Compute the stabiliser in $G$ of a subset of \(\Omega \times \Omega\) of
    size \(\lfloor n^{2} / {2} \rfloor\).

  \item\label{exp-set-stab-grid-2}
    Compute the stabiliser in $G$ of a subset of \(\Omega \times \Omega\) that
    has \(\lfloor n / {2} \rfloor\) entries from each grid-row.

  \item\label{exp-part-stab-grid-1}
    If \(2 \mathop{\vert} n\), then compute the stabiliser in $G$ of an unordered
    partition of \(\Omega \times \Omega\) that has two cells, each of size
    \(n^{2} / 2\).

\end{enumerate}

As in~\cite[Section~6.1]{newrefiners}, we compute with the \(n \times n\) grid
group as a subgroup of \(\Sn{n^{2}}\) rather than as a subgroup of $\Sym{\n
\times \n}$.
The algorithms have no prior knowledge of the grid structure that
the group preserves. Tables~\ref{tab-exp-set-stab-grid}
and~\ref{tab-exp-part-stab-grid} show the results concerning the search size
required by the different techniques to solve 50 random problems each of
types~\ref{exp-set-stab-grid-1},~\ref{exp-set-stab-grid-2},
and~\ref{exp-part-stab-grid-1}
in a grid group.
An entry in the `Zero\%' column shows the percentage of problems that
an algorithm solved with a search of size zero.
These columns are omitted when they are all-zero.

\begin{table}[!ht]
  \centering
  \small
  \begin{tabular}{r r r r r r r r r r r r}
    \hline
                      &&&& \multicolumn{2}{c}{\textsc{Orbital},}
                     &&&&& \multicolumn{2}{c}{\textsc{Orbital},} \\
     && {\textsc{Leon}} && \multicolumn{2}{c}{\textsc{Strong}, \textsc{Full}}
    &&& {\textsc{Leon}} && \multicolumn{2}{c}{\textsc{Strong}, \textsc{Full}}
    \\
    \cline{3-3}
    \cline{5-6}
    \cline{9-9}
    \cline{11-12}
    \(n\) &&
    {\scriptsize Median} && {\scriptsize Median} & {\scriptsize Zero\%} &&&
    {\scriptsize Median} && {\scriptsize Median} & {\scriptsize Zero\%}
    \\
    \hline
     3 &&    4   && 2 &  22   &&&     7   && 2 &   0 \\
     4 &&    8   && 0 &  50   &&&     8   && 2 &   0 \\
     5 &&   16   && 2 &  44   &&&    13   && 2 &   0 \\
     6 &&   23   && 0 &  68   &&&    34   && 2 &  20 \\
     7 &&   34   && 0 &  74   &&&    41   && 0 &  54 \\
     8 &&   46   && 0 &  90   &&&    92   && 0 &  68 \\
     9 &&   58   && 0 &  92   &&&   108   && 0 &  54 \\
    10 &&   75   && 0 &  88   &&&   290   && 0 &  86 \\
    11 &&  107   && 0 &  94   &&&   262   && 0 &  90 \\
    12 &&  124   && 0 & 100   &&&  1085   && 0 &  92 \\
    13 &&  155   && 0 & 100   &&&   788   && 0 &  98 \\
    14 &&  185   && 0 &  96   &&& 21774   && 0 &  96 \\
    15 &&  216   && 0 &  98   &&&  2471   && 0 & 100 \\
    \cline{1-1}
    \cline{3-6}
    \cline{9-12}
     && \multicolumn{4}{c}{Problem~\ref{exp-set-stab-grid-1}}
    &&& \multicolumn{4}{c}{Problem~\ref{exp-set-stab-grid-2}}
    \\
    \hline
  \end{tabular}
  \caption{
    Search sizes for 50 instances of
    Problems~\ref{exp-set-stab-grid-1} and~\ref{exp-set-stab-grid-2}
    in the $n \times n$ grid group.
  }\label{tab-exp-set-stab-grid}
\end{table}

\begin{table}[!ht]
  \centering
  \small
  \begin{tabular}{r r r r r r r r}
    \hline
    && \multicolumn{1}{c}{\textsc{Leon}}
    && \multicolumn{1}{c}{\textsc{Orbital}}
    && \multicolumn{2}{c}{\textsc{Strong}, \textsc{Full}}
    \\
    \cline{3-3}
    \cline{5-5}
    \cline{7-8}
    \(n\)
    && {\scriptsize Median} && {\scriptsize Median}
    && {\scriptsize Median} & {\scriptsize Zero\%}
    \\
    \hline
    4   &&  16   &&  16   && 5 &  24 \\
    6   &&  44   &&  36   && 0 &  66 \\
    8   &&  82   &&  64   && 0 &  82 \\
    10  && 129   && 100   && 0 &  88 \\
    12  && 206   && 144   && 0 &  96 \\
    14  && 317   && 196   && 0 & 100 \\
    16  && 504   && 256   && 0 & 100 \\
    18  && 664   && 324   && 0 &  98 \\
    \hline
  \end{tabular}
  \caption{
    Search sizes for 50 instances of
    Problem~\ref{exp-part-stab-grid-1}
    in the $n \times n$ grid group.
  }\label{tab-exp-part-stab-grid}
\end{table}

In~\cite[Section~6.1]{newrefiners}, the \textsc{Orbital} algorithm was much
faster than the classical \textsc{Leon} algorithm at solving
problems of types~\ref{exp-set-stab-grid-1} and~\ref{exp-set-stab-grid-2}.  In
Table~\ref{tab-exp-set-stab-grid}, we see why: \textsc{Orbital} typically
requires no search for these problems. \textsc{Leon} used a total of 65,834
nodes to solve all problems in Problem~\ref{exp-set-stab-grid-1}, and 37,882,616
nodes for Problem~\ref{exp-set-stab-grid-2}, while \textsc{Orbital} required 567
for Problem~\ref{exp-set-stab-grid-1} and 1073 for
Problem~\ref{exp-set-stab-grid-2}.  The same numbers of nodes were also required
for both \textsc{Strong} and \textsc{Full}, since there is no possible
improvement.

In Table~\ref{tab-exp-part-stab-grid}, however, we clearly see the
benefits of our new techniques with unordered partition stabilisers.  For these
problems, partition backtrack~--~\textsc{Leon} and \textsc{Orbital}~--~takes an
increasing number of search nodes, with 140,177 nodes required for \textsc{Leon}
and 57,120 nodes for \textsc{Orbital} to solve all instances of
Problem~\ref{exp-part-stab-grid-1}.  The \textsc{Strong} algorithm, on the other
hand, is powerful enough in almost all cases to solve these same problems
without search, requiring only 450 nodes to solve all problem instances.

\subsection{Intersections of primitive groups with symmetric wreath
products}\label{sec-primitive-intersection}

As in~\cite[Section~6.2]{newrefiners}, we next investigate the behaviours that
the various search techniques have when intersecting primitive groups with
wreath products of symmetric groups.  This gives difficult but
interesting examples of subgroup intersections.  To construct these problems, we
use the primitive groups library, which is included in the
\textsc{PrimGrp}~\cite{primgrp} package for \textsc{GAP}.

For a given a composite \(n \in \{6,\ldots,80\}\), we create the
following problems: for each primitive subgroup \(G \leq \Sn{n}\) that is
neither the symmetric group nor the natural alternating subgroup of \(\Sn{n}\), and for
each proper divisor $d$ of $n$, we construct the wreath product
$\Sn{n/d} \wr \Sn{d}$ as a subgroup of
\(\Sn{n}\), which we then conjugate by a randomly chosen element of
\(\Sn{n}\).  Finally, we use each of the algorithms in turn to compute the
intersection of \(G\) with the conjugated wreath product.
We create 50 such intersection problems for each $n$, $G$, and $d$.

For each \(k \in \{6,\ldots,80\}\), we record the cumulative number of search
nodes that each algorithm requires to solve all of the intersection problems for
all composite \(n \in \{6,\ldots,k\}\).
We display these cumulative totals in Figures~\ref{fig-prim1trans}
and~\ref{fig-prim2trans}.
As in~\cite[Section~6.2]{newrefiners}, we separate the 2-transitive groups from
those primitive groups that are not 2-transitive.  Note that there exist quite a
few values of \(n \in \{6, \ldots, 80\}\) for which every primitive subgroup of
$\Sn{n}$ is 2-transitive.

\begin{figure}[!ht]
  \centering
  \includegraphics[width=0.65\textwidth,keepaspectratio,clip=true]{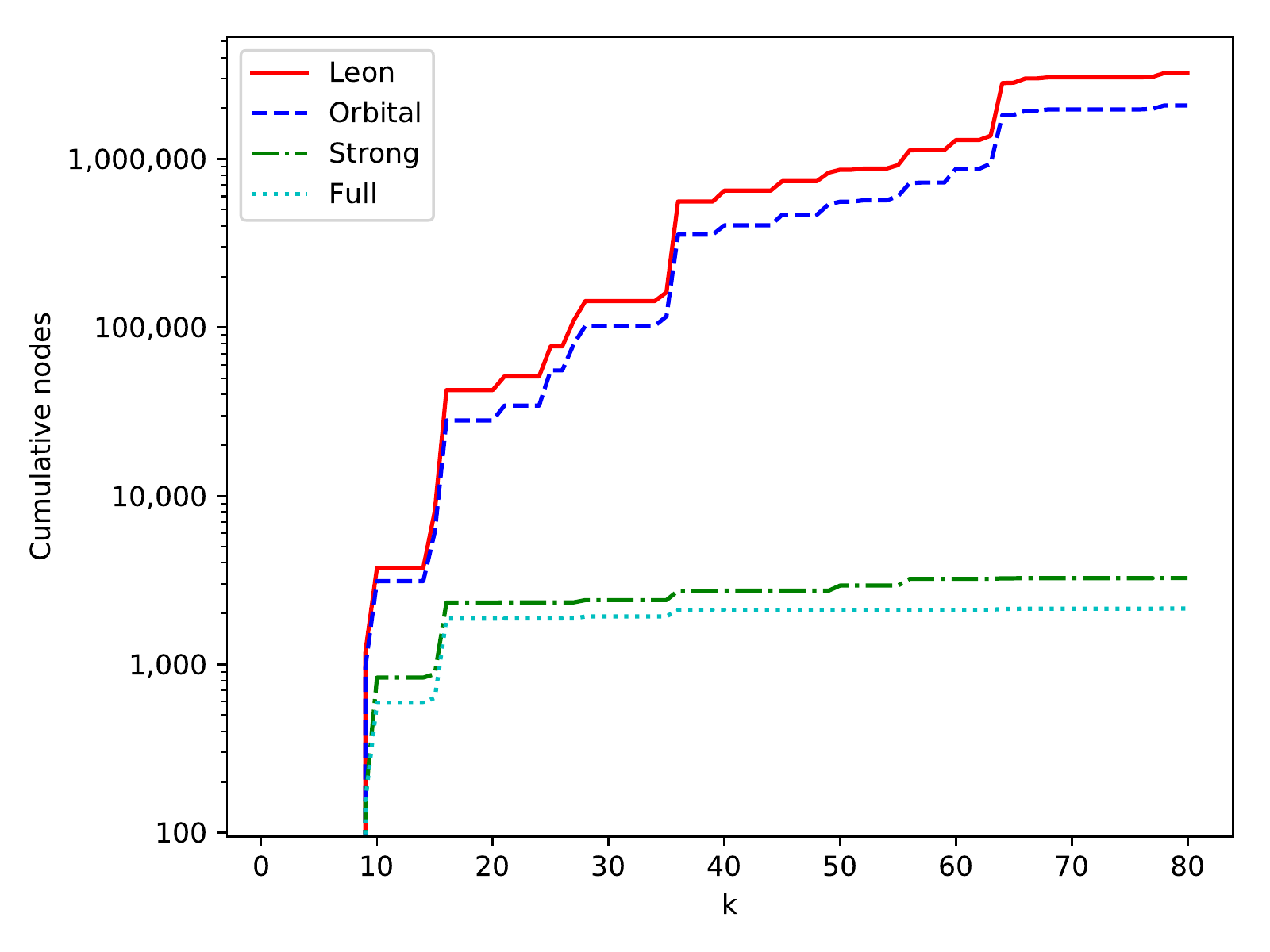}
  \caption{
    Cumulative nodes required to intersect primitive (but
    not 2-transitive) groups with wreath products of symmetric groups, for all
    problems with \(n \in \{6, \ldots, k\}\).
  }\label{fig-prim1trans}
\end{figure}

For the primitive but not 2-transitive groups, the total number of search nodes
required by the \textsc{Leon} algorithm is
3,239,403.
The \textsc{Orbital} algorithm reduces this total search size by approximately
35\%, to
2,079,356,
but the cumulative search size for \textsc{Strong} is much smaller, at only
3,248 nodes, and for \textsc{Full} the cumulative search size is even smaller,
at only 2,140 nodes.

This huge reduction in search size is because the \textsc{Strong} and
\textsc{Full} algorithms solve almost every problem without any search.  Out of
40,150 experiments, the \textsc{Strong} algorithm required search for only 703,
and the \textsc{Full} algorithm required search for only 654. On the other hand,
the \textsc{Leon} and \textsc{Orbital} algorithms required search for every
single problem.

\begin{figure}[!ht]
  \centering
  \includegraphics[width=0.65\textwidth,keepaspectratio,clip=true]{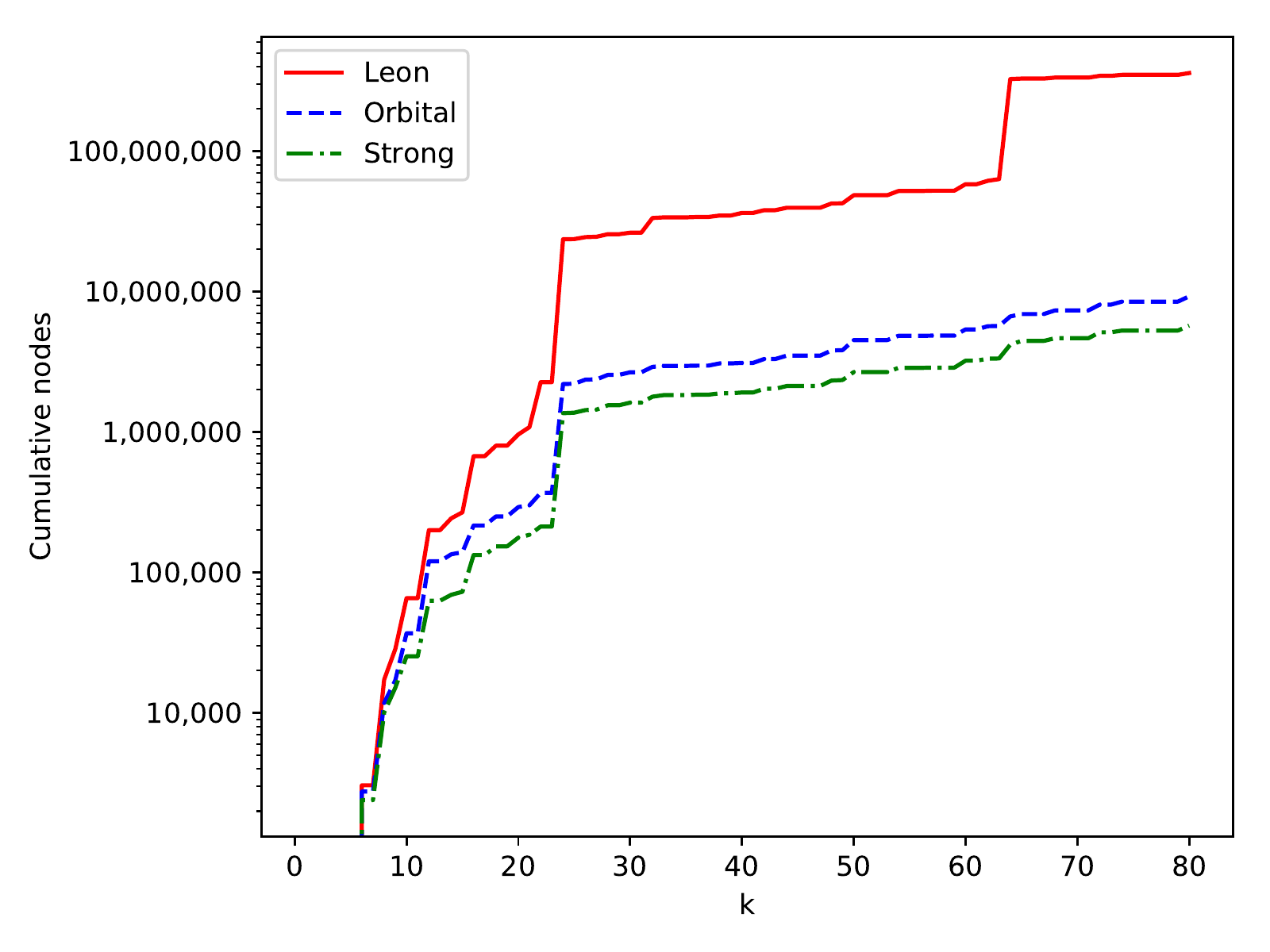}
  \caption{
    Cumulative nodes required to intersect 2-transitive groups
    with wreath products of symmetric groups, for all problems
    with \(n \in \{6, \ldots, k\}\).
    The line for \textsc{Full} is omitted, since at this scale, it is
    indistinguishable from the line for \textsc{Strong}.
  }\label{fig-prim2trans}
\end{figure}

For the intersection problems involving groups that
are at least 2-transitive, the improvement of the new techniques over the
partition backtrack algorithms is much smaller, and all of the algorithms
require a non-zero search size to solve every problem.  This stems from the fact
that a 2-transitive group has a unique orbital graph, which is a complete digraph.

The \textsc{Leon} algorithm needs roughly 359
million search nodes,
while the \textsc{Orbital} algorithm requires roughly 9.23 million,
the \textsc{Strong} algorithm requires roughly 5.72 million,
and the \textsc{Full} algorithm requires roughly 5.59 million.
Therefore
the \textsc{Strong} and \textsc{Full} algorithms still require almost 40\%
fewer nodes than the \textsc{Orbital} algorithm.
Out of the 25,600 total experiments, \textsc{Orbital} is better than
\textsc{Leon} in
22,178 instances.  Of these, \textsc{Strong} is better than \textsc{Orbital} in
2,554 instances, and of these, \textsc{Full} is better than \textsc{Strong} 236
times.  This shows that for 2-transitive groups, there are a relatively small
number of problems where \textsc{Strong} and \textsc{Full} improve upon
\textsc{Orbital}.

\subsection{Intersections of cosets of intransitive
groups}\label{sec-intransitive-intersection}

In this section, we go beyond the experiments of~\cite[Section~6]{newrefiners},
and investigate the behaviour of the
algorithms when intersecting
cosets of intransitive groups that have identical orbits, and where all orbits
have the same size.
We chose these kinds of problems because there should be many instances
that all of the algorithms find difficult,
because of this regularity of orbit structure.

We intersect right cosets of subdirect products of
transitive groups of equal degree.  Although we create them in a random way, we
do not make any claims about their distribution.
Given \(k, n \in \N\),
we randomly choose \(k\) transitive subgroups of \(\Sn{n}\) from the transitive
groups library \textsc{TransGrp}~\cite{transgrp}, each of which we conjugate by a
random element of \(\Sn{n}\), and we create their direct product, $G$, which we
regard as a subgroup of \(\Sn{k n}\).
Then, we randomly sample elements of \(G\) until the subgroup that they generate
is a subdirect product of \(G\).  If this subdirect product is equal to \(G\),
then we abandon the process and start again. Otherwise, the result is a
generating set for what we call a \emph{proper $(k,n)$-subdirect product}.

In our experiments, for various \(k,n \in \N\), we explore the search space
required to determine whether the intersections of pairs of right cosets of
different \((k,n)\)-subdirect products are empty.
To make the problems as hard as possible, we choose coset
representatives that preserve the orbit structure of the
\((k,n)\)-subdirect product.

We performed 50 random instances for each pair $(k,n)$, for all
$k, n \in \{2,\ldots,10\}$, and
we show a representative sample of this data in
Tables~\ref{tab-exp-subdirect-n} and~\ref{tab-exp-subdirect}
and Figure~\ref{fig-77nodes}.
Table~\ref{tab-exp-subdirect-n} shows all results for each $n$,
and Table~\ref{tab-exp-subdirect} gives a more in-depth view for two values of $k$.
The tables omit data for the \textsc{Full} algorithm, because it was mostly
identical to the data for the \textsc{Strong} algorithm,
and it varied in only one instance by more than 1\%.

\begin{table}[!ht]
  \small
  \centering
  \begin{tabular}{r r r r r r r r r r r r r r}
    &&\multicolumn{2}{c}{\textsc{Leon}}
     &&\multicolumn{3}{c}{\textsc{Orbital}}
     &&\multicolumn{3}{c}{\textsc{Strong}}\\
    \cline{3-4} \cline{6-8} \cline{10-12}
    \(n\)
    && {\scriptsize Mean} & {\scriptsize Median}
    && {\scriptsize Mean} & {\scriptsize Median} & {\scriptsize Zero\%}
    && {\scriptsize Mean} & {\scriptsize Median} & {\scriptsize Zero\%} \\
    \hline
    2   &&     3  &  2  &&     2& 2  &14  &&    2 & 2 & 14 \\
    3   &&  1418  &  7  &&    19& 0  &58  &&   19 & 0 & 59 \\
    4   &&  1250  & 12  &&    71& 0  &69  &&   62 & 0 & 70 \\
    5   &&  37924 & 30  && 15576& 10 &14  && 8803 & 0 & 54 \\
    6   &&   584  & 12  &&   254& 6  &36  &&  139 & 0 & 86 \\
    7   && 53612  & 28  && 43555& 14 &0   && 8982 & 0 & 70 \\
    8   &&  1142  &  8  &&   997& 8  &15  &&    4 & 0 & 98 \\
    9   &&  6547  &  9  &&  5562& 9  &2   &&    7 & 0 & 95 \\
    10  &&  8350  & 10  &&  6959& 10 &1   &&    7 & 0 & 97 \\
    \hline
\end{tabular}
  \caption{
    Search sizes for $(k,n)$-subdirect product
    coset intersection problems, where for each $n$,
    we ran 50 experiments for each $k \in \{2,\dots,10\}$.
  }\label{tab-exp-subdirect-n}
\end{table}

\begin{table}[!ht]
  \small
  \centering
  \begin{tabular}{r r r r r r r r r r r r r r}
    &&&\multicolumn{2}{c}{\textsc{Leon}}
     &&\multicolumn{3}{c}{\textsc{Orbital}}
     &&\multicolumn{3}{c}{\textsc{Strong}}\\
    \cline{4-5} \cline{7-9} \cline{11-13}
    \(k\) & \(n\)
    && {\scriptsize Mean} & {\scriptsize Median}
    && {\scriptsize Mean} & {\scriptsize Median} & {\scriptsize Zero\%}
    && {\scriptsize Mean} & {\scriptsize Median} & {\scriptsize Zero\%} \\
    \hline
    4  & 5  && 13683  & 30  && 6356  & 11 & 8  && 6176  & 5 & 40   \\
    4  & 6  && 376    & 18  && 335   & 6  & 8  && 87    & 0 & 76   \\
    4  & 7  && 8612   & 49  && 7065  & 43 & 0  && 6494  & 0 & 54   \\
    4  & 8  && 1133   & 8   && 365   & 8  & 14 && 0     & 0 & 100  \\
    4  & 9  && 1947   & 9   && 621   & 9  & 0  && 0     & 0 & 96   \\
    4  & 10 && 458    & 10  && 410   & 10 & 2  && 0     & 0 & 98   \\
    \hline
    8  & 5  && 119561 & 130 && 42885 & 30 & 17 && 36888 & 0 & 58   \\
    8  & 6  && 70     & 12  && 25    & 0  & 56 && 67    & 0 & 98   \\
    8  & 7  && 19731  & 49  && 11154 & 43 & 0  && 167   & 0 & 86   \\
    8  & 8  && 209    & 8   && 58    & 8  & 12 && 0     & 0 & 100  \\
    8  & 9  && 152    & 9   && 144   & 9  & 2  && 0     & 0 & 100  \\
    8  & 10 && 138    & 10  && 64    & 10 & 2  && 0     & 0 & 100  \\
    \hline
  \end{tabular}
  \caption{
    Search sizes for 50 $(k,n)$-subdirect product
    coset intersection problems.
  }\label{tab-exp-subdirect}
\end{table}

The \textsc{Strong} algorithm solved a large proportion
of problems with zero search.
As $n$ and $k$ increase, we find that
\textsc{Strong} is also able to solve almost all problems without search,
and the remaining problems with very little search. The only problems
where \textsc{Strong} does not perform significantly better are those
involving orbits of size 2 ($n = 2$).
This is not surprising as there are very few possible
orbital graphs for such groups.
We note that the problems with $n = 5$ and $7$ seem particularly difficult.
This is because transitive groups of prime degree are primitive,
and sometimes even 2-transitive,
in which case they do not have useful orbital graphs.

On the other hand, \textsc{Orbital} solved a lot fewer problems without
search, and \textsc{Leon} solved none in this way.
Although the relatively low medians show that all of the algorithms performed
quite small searches for many of the problems, we see a much starker
difference in the mean search sizes.  These means are typically dominated
by a few problems; see Figure~\ref{fig-77nodes}.

\begin{figure}[!ht]
  \centering
  \includegraphics[width=0.65\textwidth,keepaspectratio,clip=true]{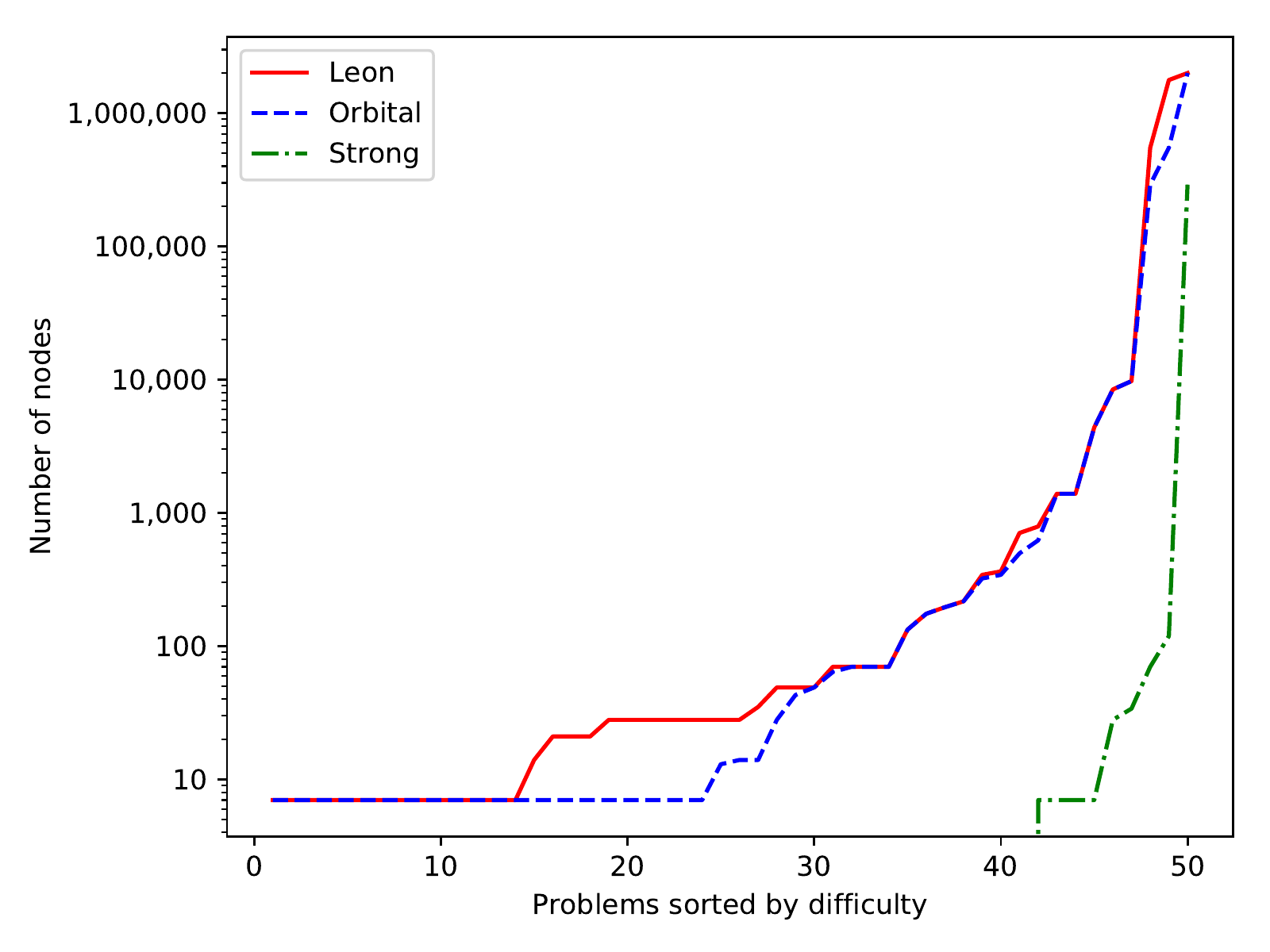}
  \caption{
    Search sizes for 50 (7,7)-subdirect product coset intersection
    problem instances.
    The data for \textsc{Full} is almost identical to the data for
    \textsc{Strong}, and is omitted.
  }\label{fig-77nodes}
\end{figure}

To give a more complete picture of how the algorithms perform,
Figure~\ref{fig-77nodes} shows the search sizes for all 50 intersections problem
that we considered for $n = k = 7$, sorted by difficulty. The data that we
collected in this case was fairly typical.
Figure~\ref{fig-77nodes} shows that
\textsc{Strong} solves almost all problems with very little or no search, and it
only requires more than 50 search nodes for the three hardest problems. On the
other hand, \textsc{Leon} and \textsc{Orbital} need more than 50 nodes for the
18 hardest problems.
All algorithms found around 30\% of the problems easy to solve.
This is because our problem generator randomly produces
easy problems, sometimes.

\section{Conclusions and directions for further work}\label{sec-end}

We have introduced and discussed new data structures and algorithms, using
labelled digraphs, which can be used to substantially reduce the size of a
search required to solve a large range of group and coset problems in
$\Sym{\Omega}$. This work
builds on the earlier partition backtrack framework of
Leon~\cite{leon1997,leon1991}, and also provides an alternative way of viewing
that earlier work.

Our new algorithms often reduce problems that previously involved searches of
hundreds of thousands of nodes into problems that require no search, and can
instead be solved by applying strong equitable labelling to a pair of stacks.
There already exists a significant body of work on efficiently implementing
equitable partitioning and automorphism finding on
digraphs~\cite{bliss,practical2}, which we believe can be generalised to work
incrementally with labelled digraph stacks that grow in length.

We therefore believe there is room for significant performance improvement over
the current state of the art, if time is spent on optimising the implementation
of the algorithms that we have presented here. In future work, we will show how the
algorithms described in this paper can be implemented efficiently, and compare
the speed of various methods for hard search problems.  In particular, we aim
for a better understanding of when partition backtrack is already the best
method available, and when it is worth using our methods. Further, earlier work
which used orbital graphs~\cite{newrefiners} showed that there are often
significant practical benefits to using only some of the possible orbital graphs
in a problem, rather than all of them. We will investigate whether a similar
effect occurs in our methods.

Another direction of research is the development and analysis of new types of
refiners, along with an extension of our methods. For example, we could allow
more substantial changes to the digraphs, such as adding new vertices outside of
$\Omega$. One obvious major area not addressed in this paper is normaliser and
group conjugacy problems, and we plan to look for new refiners for normaliser
calculations.

While the step from ordered partitions to labelled digraphs already adds some
difficulty,
we still think that it is worth considering even more intricate structures.
Why not generalise our ideas to stacks of more general combinatorial structures
defined on a set $\Omega$?  The definitions of a splitter, of an isomorphism
approximator, and of a refiner were essentially independent of the notion of a
labelled digraph, and so they~--~and therefore the algorithms~--~could work for
more general objects around which a search method could be organised.


\end{document}